\documentclass[11pt]{article} % use larger type; default would be 10pt
\usepackage[utf8]{inputenc} % set input encoding (not needed with XeLaTeX)

\usepackage[margin=1in]{geometry} % to change the page dimensions
\usepackage{graphicx} % support the \includegraphics command and options
\usepackage{booktabs} % for much better looking tables
\usepackage{array} % for better arrays (eg matrices) in maths
\usepackage{paralist} % very flexible & customisable lists (eg. enumerate/itemize, etc.)
\usepackage{verbatim} % adds environment for commenting out blocks of text & for better verbatim
\usepackage{mathrsfs}
\usepackage{amssymb}
\usepackage{amsmath,amsfonts,amssymb}
\usepackage{amsthm}
\usepackage{esint}
\usepackage{graphics}
\usepackage{enumerate}
\usepackage{mathtools}
\usepackage{xfrac}
\usepackage{subcaption}
\usepackage{stmaryrd}
\usepackage{mathabx}

\usepackage[usenames,dvipsnames]{xcolor}
\usepackage[colorlinks=true, pdfstartview=FitV, linkcolor=blue, citecolor=blue, urlcolor=blue]{hyperref}
\usepackage[normalem]{ulem}

\usepackage{tikz}
\usetikzlibrary{calc}
\usepackage{pgf}
\usetikzlibrary{external}
\numberwithin{equation}{section}
\numberwithin{figure}{section}

\newtheorem{theorem}{Theorem}[section]

\newtheorem{assumption}[theorem]{Assumption}
\newtheorem{corollary}[theorem]{Corollary}
\newtheorem{proposition}[theorem]{Proposition}
\newtheorem{lemma}[theorem]{Lemma}
\theoremstyle{definition}
\newtheorem{definition}[theorem]{Definition}

\newtheorem*{note}{Note}

\newcommand*{\Id}{\ensuremath{\mathrm{I}_d}}

\newcommand{\M}{\mathscr{M}}

\newcommand*{\N}{\ensuremath{\mathbb{N}}}

\newcommand*{\Z}{\ensuremath{\mathbb{Z}}}

\newcommand*{\R}{\ensuremath{\mathbb{R}}}

\newcommand{\eps}{\varepsilon}

\renewcommand*{\tilde}{\widetilde}

\renewcommand{\P}{\ensuremath{\mathbb{P}}}

\newcommand{\ep}{\eps}

\newcommand{\ap}{\mathfrak{p}}

\DeclareMathOperator{\SL}{SL}
\DeclareMathOperator{\linspan}{span}
\DeclareMathOperator{\Reach}{Reach}
\DeclareMathOperator{\Lip}{Lip}
\DeclareMathOperator{\dist}{dist}
\DeclareMathOperator{\Leb}{Leb}

\DeclareMathOperator{\id}{id}

\DeclareSymbolFont{boldoperators}{OT1}{cmr}{bx}{n}
\SetSymbolFont{boldoperators}{bold}{OT1}{cmr}{bx}{n}
\usepackage{accents}

\newcommand{\T}{\mathbb{T}}

\def\Xint#1{\mathchoice
  {\XXint\displaystyle\textstyle{#1}}%
  {\XXint\textstyle\scriptstyle{#1}}%
  {\XXint\scriptstyle\scriptscriptstyle{#1}}%
  {\XXint\scriptscriptstyle\scriptscriptstyle{#1}}%
  \!\int}
\def\XXint#1#2#3{{\setbox0=\hbox{$#1{#2#3}{\int}$}
    \vcenter{\hbox{$#2#3$}}\kern-.5\wd0}}

\def\fint{\Xint-}

\let\originalleft\left
  \let\originalright\right
\renewcommand{\left}{\mathopen{}\mathclose\bgroup\originalleft}
  \renewcommand{\right}{\aftergroup\egroup\originalright}

\newcommand{\<}{\langle}
\renewcommand{\>}{\rangle}

\newcommand{\E}{\mathbb{E}}

\newcommand{\xspace}{\T^2}
\newcommand{\veespace}{S^1}

\newcommand{\uspace}{H^5_{\mathrm{sol},0}(\xspace; \R^2)}
\newcommand{\uspacenotsol}{H^5(\xspace; \R^2)}
\newcommand{\wspace}{H^4(\xspace)}
\newcommand{\vfspace}{\mathfrak{X}^1}
\newcommand{\dif}{\mathrm{d}}

\renewcommand{\hat}{\widehat}

\makeatletter
\pgfmathdeclarefunction{erf}{1}{%
  \begingroup
  \pgfmathparse{#1 > 0 ? 1 : -1}%
  \edef\sign{\pgfmathresult}%
  \pgfmathparse{abs(#1)}%
  \edef\x{\pgfmathresult}%
  \pgfmathparse{1/(1+0.3275911*\x)}%
  \edef\t{\pgfmathresult}%
  \pgfmathparse{%
    1 - (((((1.061405429*\t -1.453152027)*\t) + 1.421413741)*\t
    -0.284496736)*\t + 0.254829592)*\t*exp(-(\x*\x))}%
  \edef\y{\pgfmathresult}%
  \pgfmathparse{(\sign)*\y}%
  \pgfmath@smuggleone\pgfmathresult%
  \endgroup
}
\makeatother

\usepackage{titlesec}

\newcommand{\addperiod}[1]{#1.}
\titleformat{\section}
{\centering\normalfont\Large}{\thesection.}{0.5em}{}
\titleformat*{\subsection}{\bfseries}
\titleformat{\subsubsection}[runin]
{\normalfont\bfseries}
{\thesubsubsection.}
{0.5em}
{\addperiod}
\titleformat*{\subsubsection}{\normalfont\itshape}
\titleformat*{\paragraph}{\bfseries}
\titleformat*{\subparagraph}{\large\bfseries}

\title{Exponential scalar mixing for the 2D Navier--Stokes equations with degenerate stochastic forcing}

\author{William Cooperman\thanks{Courant Institute of Mathematical Sciences, New York University, NY, USA.
    {\footnotesize \href{mailto: bill@cprmn.org}{bill@cprmn.org}.}
  }
  \and
  Keefer Rowan\thanks{Courant Institute of Mathematical Sciences,  New York University, NY, USA.
    {\footnotesize \href{mailto:keefer.rowan@cims.nyu.edu}{keefer.rowan@cims.nyu.edu}.}
  }
}
\date{\today}

\usepackage[nottoc,notlot,notlof]{tocbibind}

\begin{document}

\maketitle

\begin{abstract}
    We show exponential mixing of passive scalars advected by a solution to the stochastic Navier--Stokes equations with finitely many (e.g.\ four) forced modes satisfying a hypoellipticity condition. Our proof combines the asymptotic strong Feller framework of Hairer and Mattingly with the mixing theory of Bedrossian, Blumenthal, and Punshon-Smith.
\end{abstract}

% \setcounter{tocdepth}{2}
%\tableofcontents

\section{Introduction}
Consider the incompressible two-dimensional stochastic Navier--Stokes equations on the torus,
\begin{equation}
  \label{eq:sns-velocity-form}
  \begin{cases}
    \dot{u}_t = \nu \Delta u_t - u_t \cdot \nabla u_t + \sum_{k \in F} c_k \nabla^\perp \Delta^{-1} e_k(x) \dif W^k_t &\quad \text{for } t > 0,\\
    \nabla \cdot u_t = 0,\\
    u_0 \in \uspace,
  \end{cases}
\end{equation}
where $F \subseteq \Z^2 \setminus \{0\}$ is the finite collection of forced modes, $\{W^k_t\}_{k \in F}$ are independent standard Brownian motions, $\nu > 0$, $c_k \neq 0$ for each $k \in F$, $\{e_k\}_{k \in \Z^2}$ is the real-valued Fourier basis defined in~\eqref{eq:fourier-basis}, and $\uspace$ denotes the divergence-free, mean-zero vector fields in $\uspacenotsol$. We let constants freely depend on $F, \nu, \max_k |c_k| + |c_k|^{-1}$. We also make the following assumption on the forced modes, which exactly matches~\cite{hairer_ergodicity_2006}.
\begin{assumption}\label{asmp:modes}
    We assume throughout that $F = -F \subseteq \Z^2 \setminus \{0\}$ is finite, there exist $k,j \in F$ such that $|k| \ne |j|$, and the $\Z$-linear span of $F$ is equal to $\Z^2$.
\end{assumption}

Our main result is the exponential mixing of passive scalars advected by $u_t$.

\begin{theorem}[Exponential mixing of passive scalars]\label{thm:main-scalar-mixing}
  Let $\varphi_0 \in L^2(\xspace)$ with $\int \varphi_0 = 0$ and for $t > 0$, let $\varphi_t$ solve the transport equation
  \begin{equation}\label{eq:phi-transport}
    \dot{\varphi}_t + u_t \cdot \nabla \varphi_t = 0,
  \end{equation}
  where $u_t$ solves~\eqref{eq:sns-velocity-form} with degenerate forcing satisfying Assumption~\ref{asmp:modes}. Then for any $p>0$, there exists $C(p)>0$ and a random variable $K(p,\varphi_0,u_0)>0$ with the (uniform in $\varphi_0, u_0$) moment bound $\E K^p \leq C$, such that for all $t \geq 0$,
  \begin{equation}\label{eq:main-scalar-almost-sure-decay}
    \|\varphi_t\|_{H^{-1}} \leq K e^{- C^{-1} t} (\|u_0\|_{H^1}^{1/p} + 1)\|\varphi_0\|_{L^2} \quad \text{almost surely.}
  \end{equation}
\end{theorem}

We note that since $\|\varphi_t\|_{L^2} = \|\varphi_0\|_{L^2}$, an interpolation argument shows that $H^{-1}$ plays no special role; we could bound the $H^{-\ep}$ norm for any $\ep>0$ in the same way, provided the constants are allowed to depend on $\ep$. 

We emphasize the presence of the $L^2$ norm on the right-hand side of the mixing estimate~\eqref{eq:main-scalar-almost-sure-decay}. Mixing estimates are often stated with an $H^1$ norm on the right-hand side as, at least in the deterministic setting, for any $T$, one can choose an ``unmixed'' initial data $\varphi_0$ with $\|\varphi_0\|_{L^2} \leq C$ but $\|\varphi_T\|_{H^{-1}} = 1$ by running the equation backward from time $T$. Thus, in the deterministic setting, an estimate like~\eqref{eq:main-scalar-almost-sure-decay} is impossible. In the random case we consider here, some unmixing is of course still possible---in fact it must happen for some data---but the data that unmixes will be different on each event. The estimate~\eqref{eq:main-scalar-almost-sure-decay} shows that it is highly unlikely the same data gets ``unmixed'' across many different events. This unmixing for \textit{some} data means that it is essential that we allow the random prefactor $K$ to depend on the initial data $\varphi_0$, in fact it proves that for every $u_0,p$, we have that almost surely,
\[\sup_{\varphi_0 \in L^2_0(\T^2)} K(p, \varphi_0, u_0) = \infty.\]

The next result we state is then a \textit{universal mixing result}, which gives a mixing rate with random prefactor \textit{independent of initial data}. By the discussion above, we need to take a positive regularity norm, such as $H^1$, on the right-hand side of the estimate. The universal mixing is a direct consequence of Theorem~\ref{thm:main-scalar-mixing}, together with the abstract Proposition~\ref{prop:data-dependent-mixing-implies-universal-mixing}. We note that below, by an interpolation argument made clear in the proof of Proposition~\ref{prop:data-dependent-mixing-implies-universal-mixing}, $H^1$ plays no special role; we could take the data be uniformly bounded in $H^\ep$ for any $\ep>0$ and get the same result with constants depending on $\varepsilon$.

\begin{corollary}[Universal $H^1$ mixing]\label{cor:main-universal-mixing}
  Let $\varphi_t$ solve~\eqref{eq:phi-transport} as in Theorem~\ref{thm:main-scalar-mixing} with initial data $\varphi_0 \in H^1(\xspace)$. Then for any $p>0$, there exists $C(p)>0$ and a random variable $K(p,u_0)>0$ with $\E K^p \leq C$ such that we have the universal mixing estimate for all $t \geq 0,$
  \begin{equation}\label{eq:cor-scalar-almost-sure-decay}
    \sup_{\|\varphi_0\|_{H^1} \leq 1} \left\|\varphi_t - \fint_{\xspace} \varphi_0\right\|_{H^{-1}} \leq K e^{- C^{-1} t} (\|u_0\|_{H^1}^{1/p} + 1) \quad \text{almost surely.}
  \end{equation}
\end{corollary}

Following~\cite{bedrossian_batchelor_2021}, Theorem~\ref{thm:main-scalar-mixing} implies that the statistics of the passive scalar obey a cumulative version of Batchelor's law~\cite{batchelor_small-scale_1959}. For a discussion of the cumulative Batchelor law compared to the usual pointwise version, see~\cite[Remark 1.9]{bedrossian_batchelor_2021}. Batchelor's law is usually stated for a passive scalar with nonzero diffusivity $\kappa$ and is only supposed to be valid on length scales well above the diffusive scale at which the diffusivity becomes dominant. As such there is usually an upper bound $\approx \kappa^{-1/2}$ to the wavenumbers for which the Batchelor scaling should hold. In our case, we take $\kappa =0$ and so the Batchelor scaling holds for all (sufficiently large) wavenumbers. Since Batchelor's law is about the effect of the advective term of the passive scalar equation, considering only the zero diffusivity case simplifies the setting without changing the essential physical content. As is made clear in~\cite{bedrossian_batchelor_2021}, in order to prove the cumulative Batchelor spectrum for nonzero diffusivity requires showing uniform-in-diffusivity mixing for the passive scalar equation, as was shown in the case of nondegenerate forcing in~\cite{bedrossian_almost-sure_2021}.

In a forthcoming work, we will treat the problem of uniform-in-diffusivity mixing and hence the positive diffusivity version of the cumulative Batchelor spectrum; see Subsection~\ref{ss:uniform-in-diffusivity} for further discussion. We additionally note that while~\cite{bedrossian_batchelor_2021} gives the cumulative Batchelor spectrum by discussing expectations with respect to the invariant measure, we give a statement about the long time behavior of expectations with respect to the stochastic forcing. These statements are essentially equivalent under a suitable ergodicity property, though the form we give allows us to avoid the somewhat technical discussion of ergodicity and invariant measures for the zero-diffusivity passive scalar problem. 

Due to the mixing of $L^2$ initial data given by Theorem~\ref{thm:main-scalar-mixing}, we give a version of the Batchelor spectrum with constants proportional to the $L^2$ norm of the forcing $g$, which is in agreement with physical predictions; see~\cite[Section 1]{bedrossian_batchelor_2021} and references therein for an overview.

\begin{corollary}[Cumulative Batchelor spectrum]
    \label{cor:Batchelor-main}
    Let $\Pi_{\leq N} : L^2_0(\T^2) \to L^2_0(\T^2)$ denote the orthogonal projection to Fourier modes of wavenumber $\leq N$. Let $g \in L^2_0(\T^2)$ and $\varphi_t$ solve
    \begin{equation}
    \label{eq:phi-forced}
         \dot{\varphi}_t + u_t \cdot \nabla \varphi_t = g\,\dif B_t,
    \end{equation}
    with $\varphi_0 \in L^2_0(\T^2)$ and $B_t$ an $\R$-valued standard Brownian motion independent of $u_t$. Then there exists $C>0$ and $N_0(g) \geq 2$ such that for any $N \geq N_0$,
    \[C^{-1} \|g\|_{L^2} \log N \leq \liminf_{t \to \infty} \E\|\Pi_{\leq N} \varphi_t\|_{L^2}^2 \leq \limsup_{t \to \infty} \E\|\Pi_{\leq N} \varphi_t\|_{L^2}^2 \leq C \|g\|_{L^2} \log N.\]
\end{corollary}

We note that $\E \|\Pi_{\leq N} \varphi_t\|_{L^2}^2$ can be replaced with $\big(\E \|\Pi_{\leq N} \varphi_t\|_{L^2}^{2p}\big)^{1/p}$ for $p \geq 1$. The lower bound for $p>1$ follows from the lower bound of Corollary~\ref{cor:Batchelor-main} and Jensen's inequality. The upper bound follows from applying the BDG inequality and using the integral representation of $\varphi_t$, as in~\cite[Section 3.1]{bedrossian_batchelor_2021}. Note also that the statement of Corollary~\ref{cor:Batchelor-main} is independent of $\varphi_0$. This is because the mixing guaranteed by Theorem~\ref{thm:main-scalar-mixing} ensures any contribution from the initial data is eliminated in the large time asymptotic. Similarly the constants are independent of $u_0$, as the Laplacian in~\eqref{eq:sns-velocity-form} erases the dependence of $u_t$ on the initial data in the large time asymptotic. Lastly, see Subsection~\ref{ss:batchelor}, in particular the discussion proceeding Proposition~\ref{prop:lower-bound-batchelor}, for a discussion of the lower bound $N_0$ on the wavenumbers for which the cumulative Batchelor spectrum holds. 

\subsection{Background}

The use of Gaussian forcing as a statistical model for turbulence has a long history, going back at least 60 years~\cite{edwards_statistical_1964,novikov_functionals_1965,bensoussan_equations_1973,vishik_mathematical_1979}. Of particular interest is when the forcing is isolated to a fixed macroscopic injection scale, allowing the nonlinearity of the Navier--Stokes equation to develop an inertial range of scales that are not directly forced. Mathematically, this corresponds to a highly degenerate forcing, involving only a finite-dimensional spatially-smooth noise process. A particularly natural case---such as when the forcing is statistically translation invariant---is where finitely many Fourier modes are directly forced, as in~\eqref{eq:sns-velocity-form}. Taking the forcing white-in-time is computationally the most convenient, though not altogether necessary, as discussed in Subsection~\ref{ss:correlated}.

While appealing as a physical model, degenerately forced stochastic Navier--Stokes brings with it substantial mathematical challenges. A question as fundamental as the uniqueness of invariant measures---which in finite dimensions is often reasonably straightforward---becomes exceptionally complex in this infinite dimensional setting. Even if one allows infinitely many modes to be forced, which is substantially less degenerate, uniqueness is decidedly non-obvious and was first shown in~\cite{flandoli_ergodicity_1995}. Following this, uniqueness was shown in the ``essentially elliptic'' setting, where a large but finite number (depending on the viscosity, $\nu$) of modes are forced, contemporaneously in the works~\cite{bricmont_ergodicity_2001, e_gibbsian_2001, kuksin_stochastic_2000}. Finally, in~\cite{hairer_ergodicity_2006}, uniqueness in the totally degenerate case of a small finite number (independent of $\nu$) of directly forced modes was proven. As a direct consequence of this uniqueness, Hairer and Mattingly established from first principles a basic tenet of physical turbulence theory---the ergodicity which equates time averages of observables with statistical ensemble averages---for the degenerately forced stochastic model of turbulence. While we will primarily work in the framework introduced by~\cite{hairer_ergodicity_2006}, we note the alternative framework of~\cite{kuksin_exponential_2020}, which is able to prove uniqueness of invariant measures for two-dimensional Navier--Stokes subject to a wide variety of Gaussian and non-Gaussian degenerate forcings, including the interesting case of bounded noise. The more subtle problem of 3D stochastic Navier--Stokes has also seen progress; ergodicity and mixing of transition probabilities is proved in~\cite{da2003ergodicity, albeverio2012exponential} when all but finitely many modes are forced.

The proof of ergodicity opened the door to rigorous treatment of other statistical properties of the stochastic fluid model. Of particular interest would be the inertial range statistics of the fluid predicted by Kraichnan~\cite{kraichnan_inertial_1967}, a two-dimensional analogue of K41 theory~\cite{kolmogorov_local_1941, kolmogorov_degeneration_1941, kolmogorov_dissipation_1941}. This involves studying a vanishing viscosity limit---though one needs to add other frictional damping for the equation to remain sensible, see~\cite[Section 8]{kupiainen_ergodicity_2011} for a lucid discussion. On the other hand, the machinery of~\cite{hairer_ergodicity_2006} is built to study the fluid at some fixed positive viscosity. As such, the questions that are most apt to study are those for which no vanishing viscosity limit is needed. Perhaps the most accessible problem for which no vanishing viscosity limit is needed is the dynamics of passive tracers advected by the fluid, rather than the dynamical properties of the fluid itself. Passive tracers have long been a foundational part of turbulence theory---for a comprehensive review from the physics perspective, see~\cite{falkovichParticlesFieldsFluid2001}.  

Understanding the dynamics of passive tracers for any nontrivial model is still quite difficult. In the deterministic setting, a simple time-periodic Lipschitz velocity field which exhibits exponential mixing is constructed in recent work~\cite{elgindi_optimal_2023}. In the random setting, there are several works (see for example~\cite{carverhill1985flows, baxendale1988large, dolgopyat2004sample}) addressing \emph{stochastic flows}, where the velocity is given by a linear combination of deterministic vector fields with independent white noise as coefficients. We follow the same basic strategy laid out in these works, though our setting is complicated by the fact that we must track the evolution of the velocity field, which is coupled to the trajectories of various features of interest (for example, points advected by the flow and accompanying tangent vectors). In~\cite{bedrossian_lagrangian_2022,bedrossian_almost-sure_2022,bedrossian_almost-sure_2021,bedrossian_batchelor_2021}, Bedrossian, Blumenthal, and Punshon-Smith establish myriad dynamical statistical properties of passive tracers for several fluid models. Most notably, they establish exponential mixing for passive scalars advected by stochastic Navier-Stokes, where all but finitely many Fourier modes are independently forced.  

Our contribution thus can be viewed as the merging of the ideas of~\cite{hairer_ergodicity_2006} (as well as the follow-up papers~\cite{hairer_theory_2011} and~\cite{hairer_asymptotic_2011}) with those of~\cite{bedrossian_lagrangian_2022,bedrossian_almost-sure_2022,bedrossian_almost-sure_2021,bedrossian_batchelor_2021} in order to prove a basic statistical property of the dynamics of passive tracers advected by the natural physical model of degenerately forced stochastic Navier--Stokes.

\subsection{Sketch of the main contributions}

The main challenge of this work is the lack of the \emph{strong Feller} property for various relevant Markov processes; that is, the transition probabilities are continuous only in a weak sense, not in total variation. Bedrossian, Blumenthal, and Punshon-Smith observe that the strong Feller assumption is critically required at two points in their work~\cite[Remark 2.20]{bedrossian_almost-sure_2022}: (i) proving continuity of a family of measures given by Furstenberg's criterion, and (ii) deducing exponential mixing for the two-point process from a positive Lyapunov exponent of the derivative cocycle. Since we deal with finite-dimensional forcing, we do not expect the strong Feller property to hold. We therefore address these problems directly and show that, in both cases, a smoothing estimate similar to~\cite[Proposition~4.15]{hairer_ergodicity_2006} is sufficient.

In the case of (i), we construct a family of measures which is unconditionally continuous; as far as we are aware, this is the first such constructed family of measures which is not merely continuous under the assumption (assumed for sake of contradiction) that the top Lyapunov exponent is zero, even in the setting of elliptic forcing. The family of measures constructed is however not identical to that considered in~\cite{bedrossian_lagrangian_2022}. While they consider the conditional measures of the stationary measure for the projective process, the family we consider can be seen as the conditional measures of the stationary measure for the projective \textit{backward} process. See Subsection~\ref{ss:lyapunov-overview} for more discussion. A key idea is to construct the measures iteratively, starting from a smooth family and using the smoothing estimate, Proposition~\ref{prop:unit-time-smoothing}, to inductively propagate regularity through the iteration. An iteration of this sort is essential when working in the asymptotic strong Feller setting, as compared to the strong Feller setting of~\cite{bedrossian_lagrangian_2022} in which one can take a measurable object and immediately deduce its continuity after pulling back under the dynamics.

To prove (ii), we show that our setting fits into the weak Feller framework of~\cite{hairer_asymptotic_2011}. The primary obstacle is the construction of a Lyapunov function for the two-point process. The form of our constructed Lyapunov function closely matches that of~\cite{bedrossian_lagrangian_2022}, but we prove the drift condition differently in order to avoid a technical problem related to unbounded semigroups observed in~\cite[Remark~2.19]{bedrossian_almost-sure_2022}. After approximating the two-point process by the projective process, we use the positivity of the top Lyapunov exponent, in addition with exponential moment bounds on the velocity field to upgrade the (additive) statement of a positive Lyapunov exponent to the (multiplicative) statement of the drift condition.

Our version of the smoothing estimate required for (i) is unsurprising; our setting is identical to that of~\cite{hairer_ergodicity_2006} but with an additional coupled ODE on a compact manifold. On the other hand, the smoothing estimate (Lemma~\ref{lem:two-point-smoothing}) required for (ii) involves an ODE on a noncompact manifold and relies on a sharp scaling near the boundary for the argument to close. In order to prove such an estimate, we approximate the dynamics of two closely separated points with the dynamics of a single point and a vector tangent to that point, which can be naturally seen as a first order approximation. We show that this approximation is close in the appropriate sense, which involves showing not just the dynamics are close but also their linearizations are close. We then use the smoothing estimate for the tangent process and the closeness of the approximation to attain the smoothing estimate for the two-point process. Since the tangent vector obeys linear dynamics, the smoothing estimate automatically scales correctly. This correct scaling is then passed to the two-point dynamics, attaining the desired sharp smoothing estimate. See Subsection~\ref{ss:two-pt-by-tangent-overview} for further discussion.

Since the addition of the coupled ODEs on manifolds to stochastic Navier--Stokes takes us out of the framework of~\cite{hairer_theory_2011}, we need to redo much of the work therein to match our setting. This is covered in Section~\ref{sec:malliavin}. Along the way, we provide some simplifications to the proof of the unit-time smoothing estimate, in particular showing that in the estimates such as Proposition~\ref{prop:malliavin-induction-base} and Proposition~\ref{prop:malliavin-inductive-step}, no space around the final time $T$ is needed for the adjoint linearization $J^*_{t,T}$ to regularize the initial data. In~\cite[Section 6]{hairer_theory_2011}, Hairer and Mattingly utilize regularizing properties of $J^*_{t,T}$, allowing them to improve the result of~\cite{mattingly_malliavin_2006} by removing the higher regularity norm in their Malliavin matrix invertibility bound. In order to use this regularization, they must work on time intervals $[T/2,T-\delta]$ away from the final time, where $\delta =\ep^r$ for some well-chosen $r$---see the discussion of~\cite[Sections 1.4, 6.0, 6.4]{hairer_theory_2011} for more detail. We avoid a similar argument by using the spatial regularity of $\omega_t$ for $t \geq T/2$ and adjointness to move derivatives from $J_{t,T}^*$ onto $\omega_t$, see the proof of Proposition~\ref{prop:malliavin-induction-base} and Proposition~\ref{prop:malliavin-inductive-step}.

Finally, we show in Corollary~\ref{cor:long-time-smoothing} how the unit-time smoothing estimate together with the super-Lyapunov\footnote{See Corollary~\ref{cor:super-lyapunov}; if $P \colon X \to \mathcal{P}(X)$ is a Markov transition kernel on a space $X$, then a function $f \colon X \to \R_{\geq 1}$ is said to be super-Lyapunov (cf.~\cite[Proposition~5.5]{hairer_asymptotic_2011}) if the pullback is bounded by $Pf \leq Cf^\alpha$ for some $C > 0$ and $\alpha \in (0, 1)$. We note that if $f$ is super-Lyapunov then $\log (f + 1)$ is a Lyapunov function by Jensen, so the super-Lyapunov property can be seen as a version of the Lyapunov property with stronger stochastic integrability.} function allow us to directly deduce the long-time smoothing estimate. This simplifies the argument of~\cite{hairer_theory_2011}, allowing one to work on finite time horizons and iterate.

\subsection{Acknowledgments}

We thank Vlad Vicol for several helpful discussions and feedback on an early draft. W.C.\ was partially supported by NSF grant DMS-2303355. K.R.\ was partially supported by NSF grants DMS\mbox{-}2350340, DMS\mbox{-}2000200, and DMS\mbox{-}1954357.

\section{Overview and notation}

\subsection{Setting and notation}
We identify $\xspace$ with $[-\pi,\pi]^2$ throughout. We fix a finite set $F \subseteq \Z^2 \setminus \{0\}$ of forced modes, nonzero coefficients ${\{c_k\}}_{k \in F}$, independent standard Brownian motions ${\{W^k_t\}}_{k \in F}$, and a viscosity $\nu >0$.

\begin{note}
  Throughout, all analytic constants $C$ depend freely on $\nu, F$ and $\max_{k \in F} |c_k| + |c_k|^{-1}$. Any other dependencies of the constants will be stated explicitly, e.g.\ ``there exists a constant $C(p,\eta,\alpha)>0$ such that\dots'' means the constant $C$ depends on $p,\eta,\alpha, \nu, F$ and $ \max_{k \in F} |c_k| + |c_k|^{-1}$.
\end{note}

We define a real-valued Fourier basis by
\begin{equation}
  \label{eq:fourier-basis}
  e_k(x) :=
  \begin{cases}
    \sin(k \cdot x) &\quad \text{if $k > (0, 0)$ lexicographically,}\\
    \cos(k \cdot x) &\quad \text{otherwise.}\\
  \end{cases}
\end{equation}
\begin{definition}
  The \emph{base process} refers to the Markov process ${\{\omega_t\}}_{t \geq 0}$ on $\wspace$ which solves the (vorticity-form) incompressible stochastic Navier--Stokes equation, that is,
  \begin{equation}
    \label{eq:sns}
    \dot{\omega}_t = \nu \Delta \omega_t -  (\nabla^\perp\Delta^{-1})\omega_t \cdot \nabla \omega_t +\sum_{k \in F} c_k e_k(x) \dif W^k_t,
  \end{equation}
  with initial data given by $\omega_0$. For the general theory of such infinite-dimensional stochastic differential equations, see~\cite{da_prato_stochastic_2014}. Throughout, we write $u_t := \nabla^\perp\Delta^{-1}\omega_t$ to refer to the velocity field induced by the vorticity $\omega_t$, and equivalently we may refer to ${\{u_t\}}_{t \geq 0}$ as the base process as well.
\end{definition}

In the following definition, we write $\vfspace_{\text{loc}}(M)$ to denote the set of $C^1_{\text{loc}}$ vector fields on a manifold $M$.

\begin{definition}
\label{defn:arbitrary-process}
  Given a geodesically complete Riemannian manifold $(M, g)$ and an associated vector field map $\Theta \colon \uspace \to \vfspace_{\text{loc}}(M)$---associating to any velocity field $u$ a vector field $\Theta_u$ on $M$---we define the associated Markov process ${\{u_t, p_t\}}_{t \geq 0}$ where $u_t$ is the base process and $p_t$ solves
  \begin{equation}
    \label{eq:manifold-ode}
    \dot{p}_t = \Theta_{u_t}(p_t),
  \end{equation}
  with initial condition given by $p_0 \in M$. We additionally require that for any $u_0 \in \uspace$, the random ODE~\eqref{eq:manifold-ode} almost surely admits a global-in-time solution for any $p_0 \in M$. There are several special processes\footnote{We call the process on $\T^2\times S^1$ the projective process in order to emphasize the similarity with~\cite{bedrossian_almost-sure_2022}, although because we work on $S^1$ it should naturally be named the ``sphere process''. We note that one could simply pass all the results through the natural projection $S^1 \to RP^1$ and consider the ``true'' projective process on $RP^1$ throughout.} that we will focus on listed in Table~\ref{table:special-processes}.
\end{definition}

\begin{table}
  \caption{Several associated Markov processes of interest}
  \begin{center}
    \begin{tabular}{l l l l l}
      \bottomrule
      process & & manifold & & vector field\\
      \midrule
      one-point & $M^1$ &$:= \xspace$ & $\Theta^1_u(x)$ &$:= u(x)$\\
      two-point & $M^2$ &$:= (\xspace \times \xspace) \setminus \{(x, x) \mid x \in \xspace\}$ & $\Theta^2_u(x, y)$ &$:= (u(x), u(y))$\\
      tangent & $M^T$ & $:= \xspace \times (\R^2 - \{0\})$ & $\Theta^T_u(x,\tau)$ &$:= (u(x), \tau \cdot \nabla u(x))$\\
      projective & $M^P$ &$:= \xspace \times \veespace$ & $\Theta^P_u(x, v)$ &$:= (u(x),  v \cdot \nabla u(x)\cdot v^\perp v^\perp)$\\
      Jacobian & $M^J$ &$:= \xspace \times \SL(2, \R)$ & $\Theta^J_u(x, A)$ &$:= (u(x), A \nabla u(x))$\\
      \toprule
    \end{tabular}
    \label{table:special-processes}
  \end{center}
\end{table}

We write $P_t \colon \uspace \times M \to \mathcal{P}(\uspace \times M)$ to denote the time $t$ transition kernel, where $\mathcal{P}(X)$ denotes the space of probability measures on $X$. We also let $P_t$ act adjointly on observables $\varphi \colon \uspace \times M \to X$ by pulling back:
\[P_t \varphi(u_0,p_0) = \int \varphi(u_t,p_t) P_t(u_0,p_0; \dif u_t, \dif p_t).\]

\subsection{\textit{A priori} estimates and a super-Lyapunov function for stochastic Navier--Stokes}

As we describe in Subsection~\ref{ss:reduction-to-two-point}, the main challenge of this work is to prove that the two-point process, whose manifold component lives on $M^2$, mixes exponentially. Since $M^2$ is not compact, there are substantial obstacles to proving mixing. Before confronting these issues, we need to build tools that are sufficient to show the substantially simpler mixing of the one-point process. A key ingredient to control the $\uspace$ component of this process is a super-Lyapunov function of $u_t$. In order to construct the super-Lyapunov function, we need the following \textit{a priori} estimates on the vector field $u_t$, which will also be used many times throughout. We defer their fairly standard proofs to Subsection~\ref{ssa:sns-bounds}.

\begin{proposition}\label{prop:omega-bounds}
  Suppose that $\omega_t$ is a solution to~\eqref{eq:sns}. Then there exists $C>0$ such that for all $\eta \leq C^{-1}$ and for all $0 \leq s \leq t,$
  \begin{equation}
    \label{eq:l2-energy-bound}
    \E \exp\Big(\eta \sup_{s \leq r \leq t}  \Big(\|\omega_r\|_{L^2}^2  + \nu \int_s^r \|\omega_a\|_{H^1}^2\,\dif a\Big) \Big) \leq Ce^{C (t-s)} \exp\big( e^{-C^{-1} s}\eta \|\omega_0\|_{L^2}^2\big).
  \end{equation}
  Additionally, there exists $C(n)>0$ such that for all $0 \leq \eta \leq C^{-1}$ and $0 \leq s \leq t,$
  \begin{equation}
    \label{eq:Hn-bound}
    \E \exp\Big( \eta \sup_{s \leq r \leq t} \Big(\|\omega_r\|_{H^n}^2 + \nu \int_s^r \|\omega_a\|_{H^{n+1}}^2\,\dif a\Big)^{\frac{1}{n+2}}\Big)\leq C e^{C t}\exp\Big(e^{-C^{-1} s} \eta \|\omega_0\|_{H^n}^{{\frac{2}{n+2}}}+C \eta \|\omega_0\|_{L^2}^2\Big),
  \end{equation}
  and
  \begin{equation}
    \label{eq:Hn-regularization}
    \E \exp\Big( \eta \sup_{s \leq r \leq t} \Big(\|\omega_r\|_{H^n}^2 + \nu \int_s^r \|\omega_a\|_{H^{n+1}}^2\,\dif a\Big)^{\frac{1}{n+2}}\Big) \leq C e^{C s^{-1}} e^{Ct}\exp\Big(C \eta \|\omega_0\|_{L^2}^2\Big).
  \end{equation}
\end{proposition}

\begin{definition}[Super-Lyapunov function for the base process]\label{def:V}
  For $\omega \in \wspace$, we define \[ V(\omega) := \sigma(\|\omega\|_{L^2}^2 + \alpha\|\omega\|_{H^4}^{1/3}), \] where $\alpha, \sigma > 0$ are some fixed small constants, chosen according to Proposition~\ref{prop:omega-bounds} so that Corollary~\ref{cor:super-lyapunov} holds for all $\eta \in (0,2)$.
\end{definition}

In the following, we will abuse notation: when $u \in \uspace$ we write $V(u)$ to mean $V(\nabla^\perp\cdot u)$, since $\omega = \nabla^\perp \cdot u$ is the vorticity associated to the velocity $u$. There is no ambiguity as $u$ is vector-valued and $\omega$ is scalar-valued.

\begin{corollary}[Super-Lyapunov property]\label{cor:super-lyapunov} There exists $\delta \in (0,1)$ and a constant $C>0$ such that for all $\eta \in (0,2)$ and $t \geq 1$,
  \[\E \exp\left(\eta V(\omega_t)\right)\leq C\exp(\delta \eta V(\omega_0)).\]
\end{corollary}

\begin{proof}
  Direct from Cauchy--Schwarz,~\eqref{eq:l2-energy-bound}, and~\eqref{eq:Hn-bound}, provided $\alpha, \sigma>0$ are chosen sufficiently small.
\end{proof}

\begin{corollary}[Exponential moments of $C^\alpha$ norms]
\label{cor:C-alpha-moments}
    There exists $\alpha>0$ such that for all $K,\eta,t>0$ there exists $C(K,\eta,t)>0$ such that
    \[\E \exp\Big(K \int_0^t \|\nabla u_s\|_{C^\alpha}\,\dif s\Big) \leq C \exp\big(\eta V(\omega_0)\big).\]
\end{corollary}
\begin{proof}
    Direct from Cauchy--Schwarz,~\eqref{eq:l2-energy-bound},~\eqref{eq:Hn-bound}, and Morrey's inequality.
\end{proof}

\subsection{Reduction to mixing of the two-point process}

\label{ss:reduction-to-two-point}

With these preliminary estimates in hand, we are ready to prove Theorem~\ref{thm:main-scalar-mixing}, under the assumption of mixing for the two-point process. This scheme of reducing scalar mixing to Markov mixing of the two-point process is the general strategy of~\cite{dolgopyat2004sample}, recently applied in~\cite{bedrossian_almost-sure_2022} to the stochastic Navier--Stokes equations. The necessary Markov mixing result is that the transition probabilities of the two-point process converge to a unique stationary measure at an exponential rate. Proposition~\ref{prop:2pt-implies-scalar} then shows that this Markov mixing result implies mixing for the passive scalar equation. At the end of Section~\ref{s:exponential-mixing}, we will conclude Theorem~\ref{thm:main-scalar-mixing} as a consequence of Proposition~\ref{prop:2pt-implies-scalar} as well as the mixing result Theorem~\ref{thm:2pt-mixing}. 

\begin{proposition}[Two-point mixing implies scalar mixing]\label{prop:2pt-implies-scalar}
  Assume there are $\alpha \in (0, 1)$ and $t_0 \geq 1$ and a distance function $d \colon {(\uspace \times M^2)}^2 \to \R_{\geq 0}$ with
  \begin{equation}\label{eq:Lipschitz-domination}
    d((u, x_1, y_1), (u, x_2, y_2)) \geq |(x_1, y_1)-(x_2, y_2)|,
  \end{equation}
    such that for any probability measure $\nu \in \mathcal{P}(\uspace \times M^2)$, we have the exponential mixing rate
  \begin{equation*}
    d(P_{nt_0}\nu, \mu^b \times \Leb) \leq \alpha^n d(\nu, \mu^b \times \Leb)
  \end{equation*}
  for $n \in \N$, where $d(\cdot, \cdot)$ denotes the Wasserstein distance associated to $d$ (see Definition~\ref{def:Wasserstein}) and $\mu^b$ is the stationary measure of the base process, given by Corollary~\ref{cor:projective-ergodicity}.

  Assume also that for the super-Lyapunov function $V$, defined above in Definition~\ref{def:V}, we have for some $C>0,$
  \begin{equation}\label{eq:d-L2}
    \int {d(\delta_{(u, x, y)}, \mu^b \times \Leb)}^2 \, \dif x \dif y \leq C\exp(V(u))
  \end{equation}
  for every $u \in \uspace$.

  Let $\varphi_0 \in L^2(\xspace)$ with $\int \varphi_0 = 0$ and for $t > 0$, let $\varphi_t$ solve the transport equation
  \begin{equation}\label{eq:scalar-transport}
      \dot{\varphi}_t + u_t \cdot \nabla \varphi_t = 0.
  \end{equation}
    Then there exists a random variable $K(\varphi_0, u_0, \alpha, t_0)\geq 0$ and constants $C(\alpha, t_0), p(\alpha, t_0) > 0$ such that for all $t \geq 0$,
    \begin{equation}
      \label{eq:scalar-uniform-in-time}
    \|\varphi_t\|_{H^{-1}}  \leq K \alpha^{t/(16t_0)} (\|u_0\|_{H^1}^p +1) \|\varphi_0\|_{L^2}^2
    \end{equation}
    where $K$ has the (uniform in $u_0,\varphi_0$) moment bound
    \[\E K \leq C.\]
\end{proposition}
\begin{proof}
  We bound the variance of Fourier coefficients of $\varphi_{nt_0}$ by, for $k \in \Z^2 \setminus \{0\}$,

  \begin{align*}
    \E\left[{\left(\int \varphi_{nt_0}(x) e_k(x) \, \dif x\right)}^2\right] &\leq \E\left[\int \varphi_{nt_0}(x) \varphi_{nt_0}(y) e_k(x) e_k(y) \, \dif x \dif y\right]\\
                                                              &\leq \int \varphi_0(x) \varphi_0(y) P_{nt_0}[e_k(\cdot)e_k(\cdot)](x, y) \, \dif x \dif y\\
                                                              &\leq \|\varphi_0\|_{L^2}^2\|P_{nt_0}[e_k(\cdot)e_k(\cdot)]\|_{L^2}\\
                                                              &\leq \|\varphi_0\|_{L^2}^2\Lip_d(e_k(\cdot)e_k(\cdot)){\left(\int d(P_{nt_0} \delta_{(u_0, x, y)}, \mu^b \times \Leb)^2 \, \dif x \dif y\right)}^{1/2}\\
                                                              &\leq C\exp\left(\frac{V(u_0)}{2}\right)\|\varphi_0\|_{L^2}^2|k|\alpha^n,
  \end{align*}
  where we use the fact that $\Lip_d(e_k(\cdot)e_k(\cdot)) \leq \Lip(e_k(\cdot)e_k(\cdot)) \leq C|k|$ since $d$ dominates Euclidean distance.
  Summing over $k$, we obtain
  \begin{align*}
    \E\left[\|\varphi_{nt_0}\|_{H^{-2}}^2\right] &= \sum_{|k| > 0} {|k|}^{-4}\E\left[{\left(\int \varphi_{nt_0}(x) e_k(x) \, \dif x\right)}^2\right]\\
                                    &\le C\exp\left(\frac{V(u_0)}{2}\right)\|\varphi_0\|_{L^2}^2\alpha^n \sum_{|k| > 0}{|k|}^{-3}\\
                                    &\leq C\exp\left(\frac{V(u_0)}{2}\right)\|\varphi_0\|_{L^2}^2\alpha^n.
  \end{align*}
  Since $\nabla \cdot u = 0$, we have $\|\varphi_{nt_0}\|_{L^2} = \|\varphi_0\|_{L^2}$ almost surely. Thus, we have from interpolation between these two bounds,
  \begin{equation}
  \label{eq:phi-mix-bad-constant}
  \E\left[\|\varphi_{nt_0}\|_{H^{-1}}^2\right] \leq C\exp\left(\frac{V(u_0)}{4}\right)\|\varphi_0\|_{L^2}^2\alpha^{n/2}.
  \end{equation}
  Using~\eqref{eq:Hn-regularization} and~\eqref{eq:l2-energy-bound}, we note that if $T \in t_0\N$ is minimal such that $T \geq C \log \|u_0\|_{H^1}$, then we have
  \begin{equation}
  \label{eq:V-small-on-log-time}
  \E\exp\left(\frac{V(u_T)}{4}\right) \leq C.
  \end{equation}
  Thus we can improve the dependence on $u_0$ by waiting until time $T$ to apply the scalar mixing estimate. That is, 
  \begin{align*}
    \E\left[\|\varphi_{T+nt_0}\|_{H^{-1}}^2\right] &\leq \E \Big[ \E \left[\|\varphi_{T+nt_0}\|_{H^{-1}}^2 \mid u_T\right] \Big]\\
                                           &\leq \E \Big[  C\exp\left(\frac{V(u_T)}{4}\right)\|\varphi_T\|_{L^2}^2\alpha^{n/2}\Big]\\
                                           &\leq C \alpha^{-T/(2t_0)}\|\varphi_0\|_{L^2}^2 \alpha^{(T+nt_0)/(2t_0)}.
  \end{align*}
  Using the definition of $T$ and the trivial bound $\|\varphi_t\|_{H^{-1}}\leq \|\varphi_t\|_{L^2} = \|\varphi_0\|_{L^2}$ for $t \leq T,$ we see that for some $C(\alpha), p(\alpha)>0$ and all $n \in \N$
  \begin{equation}
    \label{eq:phi-expectation-bound}
  \E \|\varphi_{nt_0}\|_{H^{-1}}^2 \leq C (\|u_0\|_{H^1}^p + 1) \|\varphi_0\|_{L^2}^2\alpha^{n/2}.
  \end{equation}
  
  In order to get the uniform-in-time bound~\eqref{eq:scalar-uniform-in-time}, we first get a bound for times in $t_0\N$. Let
  \[M := \max_{n \in \N} \|\varphi_{nt_0}\|_{H^{-1}}^2 \alpha^{-n/4} \|\varphi_0\|_{L^2}^{-2},\]
  so that for $n \in \N,$
  \[\|\varphi_{nt_0}\|^2_{H^{-1}} \leq M \|\varphi_0\|_{L^2}^2 \alpha^{n/4}.\]
  Then using~\eqref{eq:phi-expectation-bound},
  \[\E M\leq\sum_n  \alpha^{-n/4} \|\varphi_0\|_{L^2}^{-2}  \E \|\varphi_{nt_0}\|_{H^{-1}}^2 \leq C (\|u_0\|_{H^1}^p +1)  \sum_n \alpha^{n/4} \leq C (\|u_0\|_{H^1}^p +1).\]
    For times outside $t_0\N$, we bound how much $\|\varphi_t\|_{H^{-1}}$ can change on the unit time intervals. Using the dual formulation of the norm on $H^{-1}$ and the equation for $\varphi_s$, we see for $t \leq s$,
  \begin{equation*}
    \|\varphi_s\|_{H^{-1}} \leq \|\varphi_t\|_{H^{-1}}\exp\left(\int_t^s\|\nabla u_r\|_{L^\infty} \, \dif r\right).
  \end{equation*}
  Thus for $t \geq T$ with $\lfloor t/t_0\rfloor = n$, we have
  \[\|\varphi_t\|^2_{H^{-1}} \leq \|\varphi_{nt_0}\|_{H^{-1}}^2 \alpha^{-n/8} \max_{m \in \N, mt_0 \geq T} \alpha^{m/8} \exp\left(2\int_{mt_0}^{(m+1)t_0} \|\nabla u_r\|_{L^\infty}\,\dif r\right) \leq K^2\|\varphi_0\|_{L^2}^2 \alpha^{t/(8t_0)},\]
  where
  \[K := C M^{1/2} \max_{m \in \N, mt_0 \geq T} \alpha^{m/16} \exp\left(\int_{mt_0}^{(m+1)t_0} \|\nabla u_r\|_{L^\infty}\,\dif r\right).\]
  We note that, again using~\eqref{eq:Hn-regularization} and~\eqref{eq:l2-energy-bound}, as well as Corollary~\ref{cor:C-alpha-moments} and the definition of $T$,
  \begin{align*}
       \E \max_{n \in \N, nt_0 \geq T} \alpha^{n/8} \exp\left(2\int_{nt_0}^{(n+1)t_0} \|\nabla u_r\|_{L^\infty}\,\dif r\right)
       &\leq \sum_{n = T/t_0}^\infty  \alpha^{n/8} \E\exp\left(2\int_{nt_0}^{(n+1)t_0} \|\nabla u_r\|_{L^\infty}\,\dif r\right)
       \\&\leq C\sum_{n = T/t_0}^\infty  \alpha^{n/8} \leq C,
  \end{align*}
  so by Cauchy-Schwarz and the moment bound on $M$, we have that
  \[\E K \leq C (\|u_0\|^{p/2}_{H^1} +1).\]
  Redefining $K, p$ and using the trivial bound for $t \leq T$, we conclude~\eqref{eq:scalar-uniform-in-time}.
\end{proof}

\subsection{Asymptotic strong Feller and smoothing}

We have established that, to prove scalar mixing, it suffices to prove a mixing result for the two-point process. Before considering the two-point process, we need to establish sufficient tools to prove mixing for the one-point process, which has a compact manifold component. In this subsection, we state Proposition~\ref{prop:unit-time-smoothing} and Corollary~\ref{cor:long-time-smoothing}, which give the central smoothing estimates for the processes we consider. In particular, Corollary~\ref{cor:long-time-smoothing} is a sufficient condition for the \textit{asymptotic strong Feller} property introduced in~\cite{hairer_ergodicity_2006}. Existence of a unique stationary probability measure for the relevant processes is then essentially a direct corollary of the asymptotic strong Feller property. Proposition~\ref{prop:unit-time-smoothing} is a special case of Theorem~\ref{thm:smoothing}, which is in turn a mild generalization of~\cite[Proposition~4.15]{hairer_ergodicity_2006} to the setting of stochastic Navier--Stokes \textit{coupled to vector fields on manifolds}. We note that below we take norms of $\nabla \varphi(u,p) \in \uspace \times TM$, denoted by $\|\nabla \varphi(u,p)\|_{H^5}$. By this, we mean that we take the sum of the $H^5$ norm in the $\uspace$ component and the norm induced by the Riemannian metric in the $TM$ component.

\begin{proposition}[Unit-time smoothing]\label{prop:unit-time-smoothing}
  Consider either the base, one-point, two-point, tangent, projective, or Jacobian process.
  Then for each $\eta, \gamma \in (0, 1)$ there is $C(p, \eta, \gamma) > 0$, locally bounded in $p \in M$, such that for each Fr\'echet differentiable $\varphi \colon \uspace \times M \to \R$,
  \begin{equation} \label{eq:unit-time-smoothing}
    \|\nabla P_1\varphi(u, p)\|_{H^5} \leq \exp(\eta V(u)) \left(C\sqrt{P_1{|\varphi|}^2(u, p)} + \gamma\sqrt{P_1{\|\nabla \varphi\|_{H^5}^2}(u, p)}\right).
  \end{equation}
\end{proposition}

We will discuss the proof of this proposition in Subsection~\ref{ss:proof-of-smoothing}. When $M$ is compact, the super-Lyapunov property allows us to translate the unit-time estimate of Proposition~\ref{prop:unit-time-smoothing} into the following long-time estimate.

\begin{definition}
  For any normed space $X$ and $\tilde{V} \colon \uspace \to \R$ and $\varphi \colon \uspace \to X$, we define the weighted supremum norm
  \begin{equation}\label{eq:weighted-V-norm}
    \|\varphi\|_{\tilde{V}} := \sup_{u \in \uspace} \exp(-\tilde{V}(u))\|\varphi(u)\|_X.
  \end{equation}
\end{definition}

\begin{corollary}[Long-time smoothing]\label{cor:long-time-smoothing}
  Consider either the base, one-point, or projective process.
  For every $\eta, \gamma \in (0, 1)$ there is $C(\eta, \gamma) > 0$ such that for each Fr\'echet differentiable observable $\varphi \colon \uspace \times M \to \R$ and $t \geq 1$,
  \begin{equation}
    \label{eq:long-time-smoothing}
    \|\nabla P_t\varphi\|_{\eta V} \leq C\|\varphi\|_{\eta V} + \gamma^t\|\nabla \varphi\|_{\eta V}.
  \end{equation}
\end{corollary}
\begin{proof}
  Given $\eta, \gamma \in (0, 1)$ as in the statement, let $\tilde{\eta} := \delta \eta/2$ (where $\delta$ is from Corollary~\ref{cor:super-lyapunov}) and $\tilde{\gamma} := C^{-1}\gamma$, where $C$ is the constant in Corollary~\ref{cor:super-lyapunov}. Apply Proposition~\ref{prop:unit-time-smoothing} to obtain a constant $\tilde{C} > 0$ satisfying~\eqref{eq:unit-time-smoothing} (with $\tilde{\eta}$ and $\tilde{\gamma}$ in place of $\eta$ and $\gamma$) at every $p \in M$, using the fact that $M$ is compact.

  First, we use the super-Lyapunov property to obtain~\eqref{eq:long-time-smoothing} for $t=1$ with $\tilde{C}$ in place of $C$. Indeed, we bound
  \begin{equation*}
    \sqrt{P_1{|\varphi|}^2(u, p)} \leq \sqrt{P_1 \|\varphi\|_{\eta V}^2 \exp(2\eta V)(u, p)} \leq C\|\varphi\|_{\eta V}\exp((1-\delta/2)\eta V(u))
  \end{equation*}
  and similarly
  \begin{equation*}
    \sqrt{P_1{\|\nabla \varphi\|}^2(u, p)} \leq C\|\nabla \varphi\|_{\eta V}\exp((1-\delta/2)\eta V(u)).
  \end{equation*}
  Combining the previous two inequalities with~\eqref{eq:unit-time-smoothing} (and scaling $\tilde{C}$ by the constant from Corollary~\ref{cor:super-lyapunov}) yields~\eqref{eq:long-time-smoothing} for $t=1$.

  To conclude, we iterate to obtain~\eqref{eq:long-time-smoothing} for any $t \geq 1$. Without loss of generality (initially choosing $\gamma$ smaller if necessary) it suffices to prove~\eqref{eq:long-time-smoothing} for $t \in \N$. Inductively, we have
  \begin{equation*}
    \|\nabla P_t \varphi\|_{\eta V} \leq \left(\tilde{C} + \gamma \tilde{C} + \gamma^2 \tilde{C} + \cdots + \gamma^{t-1}\tilde{C}\right)\|\varphi\|_{\eta V} + \gamma^t\|\nabla \varphi\|_{\eta V},
  \end{equation*}
  so choosing $C := \tilde{C}{(1-\gamma)}^{-1}$ satisfies~\eqref{eq:long-time-smoothing}.
\end{proof}

\begin{corollary}\label{cor:projective-ergodicity}
  The base, one-point, and projective processes each have a unique invariant measures, denoted $\mu^b$, $\mu^1$, and $\mu^P$ respectively. Furthermore, for each initial probability measure $\nu \in \mathcal{P}(\uspace \times M)$, the averaged measures
  \[
    t^{-1}\int_0^t P_s\nu \, \mathrm{d}s
  \]
  converge weakly to the invariant measure as $t \to \infty$.
\end{corollary}
\begin{proof}
  It suffices to consider the projective process. Corollary~\ref{cor:long-time-smoothing} shows that the process is asymptotically strong Feller, and Lemma~\ref{lem:proj-control} shows that the process is approximately controllable, so~\cite[Corollary 3.17]{hairer_ergodicity_2006} proves that there is at most one invariant measure. Existence follows from tightness of the sequence of averaged measures, a consequence of Corollary~\ref{cor:super-lyapunov} and compactness of $M$. Finally, any subsequence of the averaged measures converges (along another subsequence) to an invariant measure, which is unique, therefore the original sequence of averaged measures converges as well.
\end{proof}

\subsection{Repulsion from the diagonal: positivity of the top Lyapunov exponent}
\label{ss:lyapunov-overview}

With the smoothing estimates established, we can now turn our attention to dealing with the distinct problems of proving mixing for the two-point process. As noted by~\cite{baxendale1988large, dolgopyat2004sample, bedrossian_almost-sure_2022}, the primary obstacle is that the two-point manifold $M^2$ is not compact. Indeed, since both points $(x_t, y_t) \in M^2$ follow integral curves of the same vector field, the diagonal $\{(x, x) \mid x \in \xspace\}$ is invariant under the dynamics. For transition probabilities to mix, we must see that the dynamics near the diagonal is repulsive: we hope that if $|x_0-y_0| \approx \varepsilon$, then $|x_t-y_t| \approx 1$ with high probability after time $t \approx \log \varepsilon^{-1}$.

A Lyapunov function is used to show that the dynamics near the noncompact parts of phase space (in our case, the diagonal) is repulsive. The key input to the construction of the Lyapunov function for the two-point process is the positivity of the top Lyapunov exponent of the random dynamics $u_t$ induces on $\T^2$. That is because positivity of the top Lyapunov exponent is equivalent to showing that infinitesimally close to the diagonal the dynamics are repulsive, an infinitesimal version of our desired macroscopic repulsion.

Furstenberg's criterion~\cite{furstenberg_noncommuting_1963} is a powerful tool in the theory of dynamical systems which provides a nondegeneracy condition that rules out all Lyapunov exponents being zero. We use Ledrappier's~\cite{ledrappier_positivity_1986} generalization of Furstenberg's criterion, which shows that the top Lyapunov exponent is positive unless every family $\nu \colon \uspace \times M^1 \to \mathcal{P}(S^{1})$ of probability measures on the circle which is stationary (that is, invariant in expectation) under the projective dynamics is \emph{almost surely} invariant under the same.

To contradict the existence of such a family---and thus ensure positivity of the top Lyapunov exponent---we use the fact that for fixed approximate endpoints $(u_0, x_0)$ and $(u_T, x_T)$, the set of possible trajectories is sufficiently nondegenerate that the family $(\nu_{u,x})$ cannot be deterministically invariant. Because we are only able to \emph{approximately} control the endpoints, it is crucial to obtain some continuity of the map $(u, x) \mapsto \nu_{u,x}$.

Bedrossian, Blumenthal, and Punshon-Smith~\cite{bedrossian_lagrangian_2022} apply Furstenberg's criterion to the family $(\nu_{u,x})$, specified $\mu^1$-almost everywhere by $\mu^P(\dif u, \dif x, \dif v) = \nu_{u,x}(\dif v)\mu^1(\dif u, \dif x)$. The stationarity condition for this family formally reads
\begin{equation}\label{eq:wrong-twisted-stationarity}
  \nu_{u_0,x_0} = \int {(A_1)}^{-1}_* \nu_{u_1, x_1}\, \overleftarrow{\mu}^J(u_0, x_0, \Id, \dif u_1, \dif x_1, \dif A_1),
\end{equation}
where $\overleftarrow{\mu}^J$ denotes the Markov kernel for the backward-in-time one-point process with $A_t$ solving $\dot{A}_t = -A_t\nabla u_t(x_t)$.

Ignoring the pushforward by $A_1^{-1}$, the stationarity condition~\eqref{eq:wrong-twisted-stationarity} seems to produce a smoothing effect. In particular, if we knew that the backward process satisfied a smoothing estimate like Proposition~\ref{prop:unit-time-smoothing}, then we could construct a locally Lipschitz family $(\nu_{u,x)}$  satisfying~\eqref{eq:wrong-twisted-stationarity} as in Section~\ref{sec:lyapunov}.

Unfortunately, positive viscosity for the forward process means that the equation for the backward trajectories has a negative viscosity, so the transition probabilities appear to be highly singular. On the other hand, if we are only interested in whether or not the Lyapunov exponents are identically zero, observing the process forward or backward in time makes no difference.

To circumvent this difficulty, we construct a different family of measures $(\nu_{u,x})$ satisfying~\eqref{eq:twisted-pullback}, which is a condition similar to~\eqref{eq:wrong-twisted-stationarity} but with respect to the forward process. The joint measure $\nu_{u,x}(\dif v)\mu^1(\dif u, \dif x)$ is stationary under the projective backward dynamics, and if the top Lyapunov exponent is zero, then it is almost surely invariant, and we conclude with an argument similar to that of~\cite{bedrossian_lagrangian_2022}.

The much more convenient form of~\eqref{eq:twisted-pullback} compared to~\eqref{eq:wrong-twisted-stationarity} allows for an iterative construction of the family of measures $(\nu_{u,x})$. We then use Proposition~\ref{prop:unit-time-smoothing} to propagate regularity through this iterative construction, which allows us to show that the family of measures $(\nu_{u,x})$ that we construct to satisfy~\eqref{eq:twisted-pullback} is \textit{unconditionally continuous}, that is we do not need to assume the top Lyapunov exponent is zero to deduce continuity.

\subsection{Lyapunov function for the two-point process and exponential mixing}

To obtain exponential mixing for the two-point process, we use the weak Harris framework of Hairer, Mattingly, and Scheutzow~\cite{hairer_asymptotic_2011}. The framework requires three main ingredients; the first two are similar to those required by the usual Harris theorem (see e.g.~\cite{harris1956existence, nummelin_general_1984, hairer2011yet}) and the third is unique to the weak setting.

First, we need to construct a Lyapunov function for the two-point process. This is a function $W \colon \uspace \times M^2 \to \R_{\geq 1}$ satisfying a drift condition $P_1W \leq \alpha W + C$ for some $\alpha \in (0, 1)$. This condition forces the trajectory $(u_t, x_t, y_t)$ to typically stay in a region where $W(u_t, x_t, y_t)$ is small.

Second, we need to show that sublevel sets of $W$ satisfy a \emph{small set condition}. This condition ensures that, for a fixed sublevel set $S := \{W \leq C\}$, the transition probabilities starting from any two points $z_0, z_0' \in S$ have some small overlap in a weak sense. Since the trajectories spend most of their time in $S$, the small set condition gives two independent trajectories many chances to couple.

Third, we must construct a distance $d$ on $\uspace \times M^2$ that is \emph{contracting}, which means that if $d(z_0, z_0')$ is already small (which we take to mean that $z_0, z_0'$ are already coupled), then the induced Wasserstein distance on transition probabilities contracts, in the sense that
\begin{equation*}
  d(P_1(z_0, \cdot), P_1(z_0', \cdot)) \leq \alpha d(z_0, z_0')
\end{equation*}
for $\alpha \in (0, 1)$.

If our process were to satisfy a \emph{strong} small set condition, which gives that transition probability measures overlap in the stronger sense of total variation, then trajectories which have coupled would be unable to uncouple and therefore the contraction ingredient would be unnecessary. However, since we use a weak small set condition---taking trajectories that are merely close to each other to be coupled---some quantitative information about how likely they are to uncouple, or to couple more closely, is needed. This is what contraction supplies.

Of these three ingredients, the second is immediate from a controllability argument. The third ingredient follows from a similar argument to~\cite[Proposition~5.5]{hairer_asymptotic_2011} with the help of Lemma~\ref{lem:two-point-smoothing}, a version of Proposition~\ref{prop:unit-time-smoothing} for the two-point process with a sharp scaling. The scaling in Lemma~\ref{lem:two-point-smoothing} allows us to close the argument with a Lyapunov function for the two-point process which is only super-Lyapunov in the $\uspace$ coordinate, rather than a full super-Lyapunov function as in~\cite{hairer_asymptotic_2011}, which may not exist for the two-point process.

Proving the first ingredient requires some innovation; in our weak Harris setting we are unable to use the twisted semigroup methods of~\cite{bedrossian_almost-sure_2022} (see~\cite[Remark~2.19]{bedrossian_almost-sure_2022}) and therefore choose to work directly with the two-point chain. We use the positivity of the top Lyapunov exponent as well as a weak reverse-Jensen type inequality to obtain a function that satisfies the drift condition at a bounded random time, and then argue that, after a small modification, this function satisfies the original drift condition.

\subsection{Discussion of the proof of the smoothing estimate}
\label{ss:proof-of-smoothing}
This discussion assumes substantial familiarity with~\cite{hairer_ergodicity_2006,hairer_theory_2011}. Our goal in this subsection is to explain where new arguments are needed to prove the smoothing estimate Theorem~\ref{thm:smoothing}, which is for a stochastic PDE coupled to a nonlinear flow on a manifold, compared to the proofs of the analogous smoothing estimates which do not cover manifold systems of~\cite{hairer_ergodicity_2006,hairer_theory_2011}. We first note that Theorem~\ref{thm:smoothing} is only a unit time smoothing estimate, compared to the global-in-time estimates of~\cite{hairer_ergodicity_2006,hairer_theory_2011}, which makes the proof somewhat simpler. We note also that Corollary~\ref{cor:long-time-smoothing} shows that the short time estimate implies a global-in-time estimate, at least in the case that the manifold is compact. 

The proof of a smoothing estimate like that of Theorem~\ref{thm:smoothing} proceeds in two parts. The first is proving a non-degeneracy bound on the Malliavin matrix, given here as Theorem~\ref{thm:malliavin-nondegenerate}, which one can compare with~\cite[Theorem 6.7]{hairer_theory_2011}. This is then used as the crucial ingredient to proving the smooth estimate using an approximate Malliavin control argument, as is well explained in~\cite[Section 4]{hairer_ergodicity_2006}.

The idea of~\cite[Section 6]{hairer_theory_2011} is to show that with high probability, 
\begin{equation}
\label{eq:hypo-overview-1}
\<\ap,\M_T\ap\> \leq \ep \|\ap\|_{H^4}^2  \implies \max_{|k| \leq R} \sup_{t \in [T/2,T]} |\< e_k, J^*_{t,T} \ap\>| \leq o_\ep(1) \|\ap\|_{H^4},
\end{equation}
where $\M_T$ is the Malliavin matrix, $e_k$ are the Fourier modes defined in~\eqref{eq:fourier-basis}, and $J_{t,T}^*$ is the adjoint of the linearization of process $(\omega_t,p_t)$, and ``with high probability'' means the outside of an event whose probability vanishes super-polynomially fast in $\ep$ as $\ep \to 0$. The above implication lets one immediately conclude the nondegeneracy of the Malliavin matrix given in Theorem~\ref{thm:malliavin-nondegenerate} in the case that $M = \{0\}$.  When $M$ is nontrivial, we also have to deal with directions in $TM$, not just in $\uspace$. Following the same proof scheme as~\cite[Section 6]{hairer_theory_2011}, Proposition~\ref{prop:malliavin-inducted} gives that with high probability we have both~\eqref{eq:hypo-overview-1} and
\begin{equation}
\label{eq:hypo-overview-2}\<\ap,\M_T\ap\> \leq \ep \|\ap\|_{H^4}^2  \implies \max_{|k| \leq R} |\<\Theta_{e_k}(p_T)  -  \nabla^\perp \Delta^{-1} \omega_T \cdot \nabla e_k - \nabla^{\perp} \Delta^{-1} e_k \cdot \nabla \omega_T, \ap\>| \leq o_\ep(1)\|\ap\|_{H^4}.
    \end{equation}
Assuming $\Theta_{e_k}$ uniformly span $TM$, we would be able to conclude directly analogously to~\cite[Section 6]{hairer_theory_2011} provided that instead of the second implication above, we got the implication 
\[\<\ap,\M_T\ap\> \leq \ep \|\ap\|_{H^4}^2  \implies \max_{|k| \leq R} |\<\Theta_{e_k}(p_T), \ap\>| \leq o_\ep(1) \|\ap\|_{H^4}.\]
Thus to conclude, we want to show,
\begin{equation}
\label{eq:hypo-overview-3}
\<\ap,\M_T\ap\> \leq \ep \|\ap\|_{H^4}^2  \implies \max_{|k| \leq R} |\< \nabla^\perp \Delta^{-1} \omega_T \cdot \nabla e_k+ \nabla^{\perp} \Delta^{-1} e_k \cdot \nabla \omega_T, \ap\>| \leq o_\ep(1) \|\ap\|_{H^4}.
\end{equation}
Since $\nabla^\perp \Delta^{-1} \omega_T \cdot \nabla e_k+ \nabla^{\perp} \Delta^{-1} e_k \cdot \nabla \omega_T \in H^4(\T^2),$ this implication ``should'' be a consequence of~\eqref{eq:hypo-overview-1} with a perhaps much larger $R$. The issue is that generically $\nabla^\perp \Delta^{-1} \omega_T \cdot \nabla e_k+ \nabla^{\perp} \Delta^{-1} e_k \cdot \nabla \omega_T$ has nontrivial projections onto all Fourier modes, while no matter how larger $R$ is chosen~\eqref{eq:hypo-overview-1} only gives that projections onto finitely many Fourier modes are small. 

This obstacle is overcome in Proposition~\ref{prop:H-n-2-malliavin}, which shows that if $\<\ap, \M_T\ap\> \leq \ep \|\ap\|_{H^4}^2$, then with high probability all Fourier modes are small, more particularly that $\|\Pi_{H^4}\ap\|_{H^{2}}$ is small, where $\Pi_{H^4}$ is the projection onto the $H^4$ component of $\ap$. This is done by simply using the implication~\eqref{eq:hypo-overview-1} for a well-chosen $R$ and using that $\|\Pi_{\geq R} \ap\|_{H^2} \leq R^{-2} \|\ap\|_{H^4}$ to control the high modes, where $\Pi_{\geq R}$ is the projection onto Fourier modes with wavenumber bigger than $R$. We thus get that with high probability
\[\<\ap, \M_T\ap\> \leq \ep \|\ap\|_{H^4}^2 \implies \|\Pi_{H^4} \ap\|_{H^2} = o_\ep(1).\]
Unfortunately, this does not allow us to conclude~\eqref{eq:hypo-overview-3} holds with high probability. This is because we have to control the probability $\|\omega_T\|_{H^7}$ is larger than $\frac{1}{o_\ep(1)}$. Since this is just some arbitrarily slow rate, no matter how good our bounds on moments of $\omega_T$, we only get that the probability this happens $o_\ep(1)$. This then gives that~\eqref{eq:hypo-overview-3} holds, but only with probability $1-o_\ep(1)$ instead of with high probability (in the sense defined above). This then allows us to conclude Theorem~\ref{thm:malliavin-nondegenerate}, but with the dramatically worse stochastic integrability than~\cite[Theorem 6.7]{hairer_theory_2011}, in that instead of getting $\ep^q$ (for any $q$) on the right-hand side of the bound, we just get $o_\ep(1)$. 

The takeaway from the above is that we can show a non-degeneracy bound on the Malliavin matrix in the case we consider here with coupled dynamics on a manifold that is analogous to~\cite[Theorem 6.7]{hairer_theory_2011} and follows with a similar proof, except that we get dramatically worse stochastic integrability in the bound as a consequence of needing to control infinitely many Fourier modes in order to show the nondegeneracy in the manifold directions.

The question at this point is then how we can conclude the smoothing estimate Theorem~\ref{thm:smoothing} under these worse stochastic integrability conditions. The proof of Theorem~\ref{thm:smoothing} follows from the approximate integration by parts computation~\eqref{eq:integration-by-parts} together with a bound on the cost of the Malliavin control $v^\beta$---which is a bound on a Skorokhod integral---and the size of the error $\rho^\beta$. The bound on the cost of the Malliavin control $v^\beta$ is exactly as in~\cite[Section 4.8]{hairer_ergodicity_2006}---the worse stochastic integrability and the presence of the manifold coordinates do not meaningfully alter the argument whatsoever. The difference is in the bounding of the error $\rho^\beta$. The presence of the manifold coordinates does not have any effect; it is the stochastic integrability that is the obstacle.

It turns out the stochastic integrability is not a substantial issue though, as a simple argument reduces the error bound to controlling the size of the low modes of
\[ \beta (\M_{1/2} + \beta)^{-1} J_{0,1/2}\ap\]
for $\beta >0$. But since $\|\beta (\M_{1/2} + \beta)^{-1} \| \leq 1$ \textit{deterministically} and Theorem~\ref{thm:malliavin-nondegenerate} gives that $\M_{1/2}^{-1}$ is nondegenerate on the low modes with \textit{some} stochastic integrability, interpolating these two bounds gives precisely the control of the error we need. 

\subsection{Approximating the two-point process by the tangent process}
\label{ss:two-pt-by-tangent-overview}

In this subsection we discuss Lemma~\ref{lem:two-point-smoothing}, which gives a version of the smoothing estimate of Proposition~\ref{prop:unit-time-smoothing} for the two-point process but with sharper dependence on the manifold coordinate $x,y$ as they approach the diagonal. Since this process is on a non-compact manifold, getting good constant dependence in the manifold coordinates is nontrivial. In particular what we show is a smoothing estimates with constants totally uniform in $x,y$, but with respect to a metric on $M^2$ that norms tangent vectors $(v_x,v_y) \in T_{(x,y)}M^2$ by
\begin{equation}
\label{eq:two-point-metric-overview}
|(v_x,v_y)| = |v_x+ v_y| + \frac{|v_x - v_y|}{|x-y|}.\end{equation}
This corresponds to the fact that it is hard to move points $x,y$ relatively to each other when they are close, but it never gets more difficult to move their center of mass. The heuristic behind this estimate is to imagine when $x,y$ are very close together, then $y-x$ should evolve like a tangent vector $\tau$ to $x$ in the $y-x$ direction. This suggests we can approximate the two-point process on $M^2$ by the tangent process on $M^T$. Since we are trying to prove a derivative estimate, we need this approximation not just in a $C^0$ way but also in a $C^1$ way: that is, we need to show that not only are the processes close, but their linearizations are close as well. This is what is shown in Subsection~\ref{ssa:two-point-by-tangent}.

With this closeness in hand, we can redo the proof of Theorem~\ref{thm:smoothing}. The trick now is to use the Malliavin control coming from the tangent process as a control for the two point process. We can then use the smoothing estimate for the tangent process, which will have constants independent of $x,y$, together with the closeness of the tangent process to the two-point process to give the smoothing estimate on the two-point process of Lemma~\ref{lem:two-point-smoothing}. The metric~\eqref{eq:two-point-metric-overview} appears naturally when going back and forth between the two-point process and the tangent process.

\subsection{Almost-sure mixers with uniform rates are universal mixers}
We now consider corollaries and extensions of the main result, Theorem~\ref{thm:main-scalar-mixing}. First, we see how it implies Corollary~\ref{cor:main-universal-mixing}, which is a consequence of the following abstract proposition. The idea is to apply the mixing given by Theorem~\ref{thm:main-scalar-mixing}, which gives a different random constant for each initial data, to an orthonormal basis---in this case the Fourier modes. Decomposing arbitrary data using this basis and using the triangle inequality, we get a mixing bound with a constant depending only on the $H^3$ norm. Interpolating the bound on the solution operator $H^3 \to H^{-1}$ with the fact that the solution operator is an isometry $L^2 \to L^2$, we conclude. 

\begin{proposition}
    \label{prop:data-dependent-mixing-implies-universal-mixing}
    Let $v_t$ be an arbitrary random vector field taking values in $\uspace$. Let $\Phi_t : L^2_0(\T^2) \to L^2_0(\T^2)$ be the solution operator to the passive scalar equation
    \[\dot \varphi_t + v_t \cdot \nabla \varphi_t = 0.\]
    Let $\alpha>0$ and suppose that for each Fourier mode $e_k, k \in \Z^2 \setminus \{0\}$, there is a random variable $K_k$ such that we have the mixing estimate
    \[\|\Phi_t e_k\|_{H^{-1}} \leq K_k e^{-\alpha t}.\]
    Suppose also that for some $p \geq 1,M>0$ and each $k$,
    \[\big(\E K_k^p\big)^{1/p} \leq M.\]
    Then there exists $C>0$ and a random variable $K$ such that
    \[  \sup_{\|\varphi\|_{H^1} \leq 1, \int \varphi\,\dif x =0 } \left\|\Phi_t \varphi\right\|_{H^{-1}} \leq Ke^{- \alpha t/3},\]
    and
    \[\big(\E K^{3p}\big)^{1/3p} \leq CM^{1/3}.\]
\end{proposition}

\begin{proof}
    Let $\varphi \in H^3_0$ arbitrary. Then 
    \begin{align*}\|\Phi_t \varphi\|_{H^{-1}} &\leq \sum_{k \in \Z^2 \setminus \{0\}} |\hat \varphi(k)| \|\Phi_t e_k\|_{H^{-1}}
    \\&\leq e^{-\alpha t} \sum_{k \in \Z^2 \setminus \{0\}}  |k|^3|\hat \varphi(k)|  |k|^{-3} K_k
    \\&\leq e^{-\alpha t} \|\varphi\|_{H^3} \sum_{k \in \Z^2 \setminus \{0\}} |k|^{-3} K_k. 
    \end{align*}
    We then define
    \[K := \sum_{k \in \Z^2 \setminus \{0\}} |k|^{-3} K_k\]
    and note by the triangle inequality
    \[\big(\E K^p\big)^{1/p} \leq C M.\]
    We have then shown that 
    \[ \|\Phi_t\|_{H^3_0 \to H^{-1}}:= \sup_{\|\varphi\|_{H^3} \leq 1, \int \varphi\,\dif x =0 } \left\|\Phi_t \varphi\right\|_{H^{-1}} \leq K e^{-\alpha t}.\]
    Note though that since $u_t$ is divergence-free, we have that
    \[\|\Phi_t\|_{L^2_0 \to L^2_0} = 1.\]
    Then by interpolation of linear operators on Sobolev spaces~\cite{bergh_interpolation_1976}, we have 
    \[\|\Phi_t\|_{H^1_0 \to H^{-1}} \leq \|\Phi_t\|_{H^1_0 \to H^{-1/3}} \leq 
    \|\Phi_t\|_{H^3_0 \to H^{-1}}^{1/3} \|\Phi_t\|_{L^2_0 \to L^2_0}^{2/3} \leq K^{1/3} e^{-\alpha t/3}.\]
    Redefining $K$ and using the definition of $\|\cdot\|_{H^1_0\to H^{-1}}$, we conclude.
\end{proof}

\begin{proof}[Proof of Corollary~\ref{cor:main-universal-mixing}]
    Direct from Proposition~\ref{prop:data-dependent-mixing-implies-universal-mixing} and Theorem~\ref{thm:main-scalar-mixing} applied to each Fourier mode separately.
\end{proof}

\subsection{Cumulative Batchelor spectrum}
\label{ss:batchelor}
We prove Corollary~\ref{cor:Batchelor-main} in two parts, following the general scheme of~\cite{bedrossian_batchelor_2021}, using mixing for the upper bound and regularity of $u_t$ for the lower bound.

As noted below the statement of Corollary~\ref{cor:Batchelor-main}, the mixing given by Theorem~\ref{thm:main-scalar-mixing} erases any dependence on the initial data $\varphi_0$ in the long time asymptotic considered. As such we without loss of generality suppose that $\varphi_0 =0$. As we will see in the proof, the dependence on the initial data $u_0 \in H^1_0(\T^2)$ to~\eqref{eq:sns-velocity-form} will also vanish due to the decay of the equation induced by the Laplacian. Let $\Phi_{s,t}$ be the solution operator---random through its dependence on $(u_r)_{s \leq r \leq t}$---to the equation
    \[\dot{\varphi}_t + u_t \cdot \nabla \varphi_t =0.\]
    Then the solution to~\eqref{eq:phi-forced} with $\varphi_0 =0$ is given by the stochastic integral
    \[\varphi_t = \int_0^t \Phi_{s,t} g\,\dif B_s,\]
    so by the It\^o isometry,
    \[\E \|\Pi_{\leq N} \varphi_t\|_{L^2}^2 =  \int \E \Big(\int_0^t \Pi_{\leq N} \Phi_{s,t} g\,\dif B_s\Big)^2\,\dif x =  \int \int_0^t \E \big(\Pi_{\leq N} \Phi_{s,t} g\big)^2\,\dif s\dif x = \int_0^t  \E \|\Pi_{\leq N} \Phi_{s,t} g\|_{L^2}^2\,\dif s,\]
    where we recall that $\Pi_{\leq N} : L^2_0(\T^2) \to L^2_0(\T^2)$ denotes the orthogonal projection onto the Fourier modes with wavenumber $\leq N$. We will also denote by $\Pi_{>N} : L^2_0(\T^2) \to L^2_0(\T^2)$ the orthogonal projection onto the Fourier modes of wavenumber $> N$, so that $\Pi_{\leq N} + \Pi_{> N} = 1$. Below we see the upper bound holds for any $N \geq 2$, uniform in $g$.

\begin{proposition}[Upper bound of Corollary~\ref{cor:Batchelor-main}]
    Let $\varphi_t$ solve~\eqref{eq:phi-forced} with $\varphi_0 =0$ and for $u_t$ solving~\eqref{eq:sns-velocity-form} with initial data $u_0 \in H^1_0(\T^2)$. Then there exists $C<\infty$ such that for all $N \geq 2$,
     \[\limsup_{t \to \infty} \E\|\Pi_{\leq N} \varphi_t\|_{L^2}^2 \leq C \|g\|^2_{L^2} \log N.\]
\end{proposition}

\begin{proof}
    Note that by Theorem~\ref{thm:main-scalar-mixing}
    \begin{align*}
        \E\|\Pi_{\leq N} \Phi_{s,t} g\|_{L^2}^2 &\leq N^2 \E\|\Phi_{s,t} g\|_{H^{-1}}^2
        \\&\leq CN^2(\E \|u_s\|_{H^1} +1) e^{-C^{-1}(t-s)} \|g\|_{L^2}^2
        \\&\leq CN^2(e^{-C^{-1} s} \|u_0\|_{H^1} +1) e^{-C^{-1}(t-s)} \|g\|_{L^2}^2.
    \end{align*}
    We also have the trivial bound $\E\|\Pi_{\leq N} \Phi_{s,t} g\|_{L^2}^2 \leq \|g\|_{L^2}^2$. Then, for any $\alpha>0,$
    \begin{align*}
        \limsup_{t \to \infty} \E\|\Pi_{\leq N} \varphi_t\|_{L^2}^2 &= \limsup_{t \to\infty }\int_0^t  \E \|\Pi_{\leq N} \Phi_{s,t} g\|_{L^2}^2\,\dif s
        \\&\leq  C  \|g\|_{L^2}^2 \limsup_{t \to\infty} \int_0^t 1 \land N^2(e^{-C^{-1} s}\|u_0\|_{H^1}^{1/p} +1) e^{-C^{-1}(t-s)}\,\dif s
        \\&\leq   C  \|g\|_{L^2}^2 \Big(\limsup_{t \to\infty} \int_0^{t-\alpha} N^2(e^{-C^{-1} s}\|u_0\|_{H^1}^{1/p} +1) e^{-C^{-1}(t-s)}\,\dif s+ \alpha\Big)
        \\&=  C  \|g\|_{L^2}^2 \Big( \int_\alpha^\infty N^2 e^{-C^{-1}t}\,\dif s+ \alpha\Big)
        \\&\leq C \|g\|_{L^2}^2 \big(N^2 e^{-C^{-1} \alpha} + \alpha\big).
    \end{align*}
    Then choosing $\alpha = C \log N$, we conclude.
\end{proof}

We now show the lower bound. Here is where we need $N \geq N_0(g)$, in particular we take $N_0(g)$ to be the square of the wavenumber magnitude below which a majority of the Fourier mass of $g$ lives,
\[\|\Pi_{\leq \sqrt{N_0}} g\|_{L^2} \geq \frac{3}{4} \|g\|_{L^2}.\]
The square root could be replaced with a $1-\ep$ power while only affecting the constant. It is clear that \textit{some} lower bound on $N$ is needed for the lower bound of Corollary~\ref{cor:Batchelor-main} as we only know $g \in L^2(\T^2)$, so arbitrarily much of the Fourier mass can live on arbitrarily high wavenumbers, and so will not be seen in $\|\Pi_{\leq N} \varphi_t\|_{L^2}^2$ for $N$ sufficiently small. A simple computation shows that one could take
\[N_0 = \Big(\frac{\|g\|_{H^1}}{\|g\|_{L^2}}\Big)^{1+\ep}\]
for some $\ep>0$. Since we do not assume $g \in H^1$, we rely on the (rateless) dominated convergence theorem to give that some $N_0$ exists with the property we need. 

\begin{proposition}[Lower bound of Corollary~\ref{cor:Batchelor-main}]
\label{prop:lower-bound-batchelor}
    Let $\varphi_t$ solve~\eqref{eq:phi-forced} with $\varphi_0 =0$ and for $u_t$ solving~\eqref{eq:sns-velocity-form} with initial data $u_0 \in H^1_0(\T^2)$. Then there exists $C<\infty, N_0(g) \geq 2$ such that for all $N \geq N_0$,
     \[\liminf_{t \to \infty} \E\|\Pi_{\leq N} \varphi_t\|_{L^2}^2 \leq C \|g\|_{L^2}^2 \log N.\]
\end{proposition}

\begin{proof}
    We let $R \geq 2$ be such that 
    \[\|\Pi_{\leq R} g\|_{L^2} \geq \frac{3}{4} \|g\|_{L^2}.\]
    We recall
    \[ \E\|\Pi_{\leq N} \varphi_t\|_{L^2}^2 =\int_0^t  \E \|\Pi_{\leq N} \Phi_{s,t} g\|_{L^2}^2\,\dif s.\]
    Then we note that
    \begin{align*}
        \|\Pi_{\leq N} \Phi_{s,t} g\|_{L^2} &\geq \|\Pi_{\leq N} \Phi_{s,t}  \Pi_{\leq R}g\|_{L^2} - \|\Pi_{\leq N} \Phi_{s,t} \Pi_{>R} g\|_{L^2}
        \\&\geq  \|\Pi_{\leq N} \Phi_{s,t}  \Pi_{\leq R}g\|_{L^2} - \|\Pi_{>R} g\|_{L^2}
        \\&\geq \|\Pi_{\leq N} \Phi_{s,t}  \Pi_{\leq R}g\|_{L^2} - \frac{1}{4} \|g\|_{L^2}.
    \end{align*}
    Then we also have that 
\begin{align*}
    \|\Pi_{\leq N} \Phi_{s,t}  \Pi_{\leq R}g\|_{L^2}^2 
    &= \| \Phi_{s,t}  \Pi_{\leq R}g\|_{L^2}^2  - \|\Pi_{> N} \Phi_{s,t}  \Pi_{\leq R}g\|_{L^2}^2 
    \\&\geq \|\Pi_{\leq R} g\|_{L^2}^2 - N^{-2} \|\Phi_{s,t} \Pi_{\leq R} g\|_{H^1}^2
    \\&\geq \frac{9}{16}\|g\|_{L^2}^2 - \|\Phi_{s,t}\|^2_{H^1 \to H^1} N^{-2} \|\Pi_{\leq R} g\|_{H^1}^2
    \\&\geq \frac{9}{16}\|g\|_{L^2}^2 - \|\Phi_{s,t}\|^2_{H^1 \to H^1} \frac{R^2}{N^2} \|g\|_{L^2}^2.
    \end{align*}
    Thus combining the displays,
    \[ \E \|\Pi_{\leq N} \Phi_{s,t} g\|_{L^2}^2 \geq \frac{1}{2}\E\|\Pi_{\leq N} \Phi_{s,t}  \Pi_{\leq R}g\|_{L^2}^2 - \frac{1}{16} \|g\|_{L^2}^2 \geq \frac{1}{2} \|g\|_{L^2}^2 - \frac{R^2}{N^2} \|g\|_{L^2}^2\E \|\Phi_{s,t}\|_{H^1 \to H^1}^2.\]
    Then by Gr\"onwall's inequality,
    \[\|\Phi_{s,t}\|_{H^1 \to H^1} \leq \exp\Big(\int_s^t \|\nabla u_r\|_{L^\infty}\,\dif r\Big),\]
    and so by Proposition~\ref{prop:omega-bounds} and Morrey's inequality, we have that 
    \[\E \|\Phi_{s,t}\|_{H^1 \to H^1}^2 \leq e^{C(t-s)}\exp\big(e^{-C^{-1} s} \|u_0\|_{H^1}^2\big).\]
    Thus
    \[ \E \|\Pi_{\leq N} \Phi_{s,t} g\|_{L^2}^2 \geq \Big(\frac{1}{2} - \frac{R^2}{N^2} e^{C(t-s)} \exp\big(e^{-C^{-1} s} \|u_0\|_{H^1}^2\big)\Big) \|g\|_{L^2}.\]
    Thus for $s \geq C \log \|u_0\|_{H^1}^2$ and $t \leq s + C\log\Big(\frac{N}{R}\Big)$, we have that 
    \[ \E \|\Pi_{\leq N} \Phi_{s,t} g\|_{L^2}^2 \geq \frac{1}{4} \|g\|_{L^2}^2.\]
    Thus
    \[\liminf_{t \to \infty} \E\|\Pi_{\leq N} \varphi_t\|_{L^2}^2  \geq \liminf_{t \to \infty} \int_{t- C \log(N/R)}^t  \E \|\Pi_{\leq N} \Phi_{s,t} g\|_{L^2}^2\,\dif s \geq C \log\Big(\frac{N}{R}\Big) \|g\|_{L^2}^2.\]
    Then for $N \geq N_0 := R^2$, we conclude. 
\end{proof}

\subsection{Uniform-in-diffusivity mixing}
\label{ss:uniform-in-diffusivity}

Bedrossian, Blumenthal, and Punshon-Smith~\cite{bedrossian_almost-sure_2021} prove that the conclusion of Theorem~\ref{thm:main-scalar-mixing} holds for the systems they consider even when the transport equation~\eqref{eq:phi-transport} is replaced with the advection-diffusion equation
\begin{equation*}
  \dot{\varphi}_t + u_t \cdot \nabla \varphi_t = \kappa \Delta \varphi_t,
\end{equation*}
for $\kappa > 0$, with constants independent of the diffusivity, $\kappa$.

The conclusion of Theorem~\ref{thm:main-scalar-mixing} is a long-time statement, which would appear to be sensitive to small perturbations of the equation~\eqref{eq:phi-transport}. On the other hand, we deduce scalar mixing from mixing for the two-point process, which is the unit-time statement that a certain distance on transition probabilities is a contraction. It is therefore reasonable to expect such a statement to be robust to small perturbations, and therefore we expect scalar mixing to hold uniformly in $\kappa$ as well.

In a setting where the vector field, instead of solving~\eqref{eq:sns-velocity-form}, is chosen independently at random from some family in each time interval $[t, t+1]$, Blumenthal, Coti~Zelati and Gvalani~\cite{blumenthal_exponential_2023}, construct a strong Harris framework that proves scalar mixing. Iyer, Son and the first author~\cite{cooperman_harris_2024} show that the framework is robust to small perturbations and therefore prove uniform-in-diffusivity mixing.

We expect the weak Harris framework in this work to be robust in this way as well, and we plan to investigate the question of uniform-in-diffusivity mixing in a future work.

\subsection{Correlated-in-time forcings}
\label{ss:correlated}

In this work, we only consider forcings to stochastic Navier--Stokes which are white-in-time. This is certainly the most convenient forcing to work with for the Malliavin calculus, as the extreme temporal irregularity of white noise---as well as the still low regularity of its integral, Brownian motion---create a sharp separation of time regularities of modes that are directly forced by Brownian motion, forced at first order, forced at second order, etc. This is what allows for the application of the Norris lemma~\cite[Theorem 7.1]{hairer_theory_2011}, which we make heavy use of throughout Section~\ref{sec:malliavin}. 

A clever trick that allows for more regular-in-time forcings while still working with the tools of Malliavin calculus is ``OU tower'' forcing. That is, we take stochastic Navier--Stokes $\omega_t$ with forcing
\[ \dot{\omega}_t = \nu \Delta \omega_t -  (\nabla^\perp\Delta^{-1})\omega_t \cdot \nabla \omega_t +\sum_{k \in F} c_k e_k(x) v^{k}_t,\]
where $v^k_t$ are i.i.d.\ Ornstein--Uhlenbeck processes, that is they solve
\[\dot v^k_t = - v^k_t + \dif W^k_t.\]
Then $v^k_t \in C^{1/2-}_t$ and so $\omega_t \in C^{1,1/2-}_t$. So we have constructed a version of stochastic Navier--Stokes that is at least $C^1$ in time. Further, one can convince oneself that the \textit{a priori} estimates of this equation compare favorably to~\eqref{eq:sns}, e.g.\ one could reprove Proposition~\ref{prop:omega-bounds} for this forcing. The key is that since the white noise is still there, just buried one layer deep, we can still apply the Malliavin calculus to this equation to prove smoothing estimates like those of Proposition~\ref{prop:unit-time-smoothing}. In fact, if one only considers the ``base process'' and ignores any manifold components, one can verify that the nondegeneracy condition~\cite[Assumption C.2]{hairer_theory_2011} applies to this equation. Since the results of~\cite{hairer_theory_2011} essentially allow one to take brackets as in the classical H\"ormander theorem, verifying Assumption C.2 for this system reduces to a straightforward Lie bracket computation.  

We then do not foresee any essential difficulties in extending the rest of the arguments of this work to the OU forcing. Further, by iterated the above construction---that is by allowing the first OU process $v^k_t$ to be forced by another OU process in place of white noise, then allowing the second OU process to also be forced by another OU process, and so on---one can make the forcing have any desired finite temporal regularity. The sequence of OU processes that bury the white noise under an arbitrary number of time integrals is what lends the ``OU tower'' its name. Thus, this scheme suggests a construction for a $C^k_t C^\infty_x$ version of stochastic Navier--Stokes which still has the exponential mixing properties of Theorem~\ref{thm:main-scalar-mixing} as well as the cumulative Batchelor spectrum given in Corollary~\ref{cor:Batchelor-main}. The idea for this construction was taken from~\cite[System 5]{bedrossian_almost-sure_2022}, which treats the same forcing but replaces stochastic Navier--Stokes with a Galerkin truncated version of the same. The Galerkin truncation allows for the application of the classical, finite-dimensional H\"ormander theorem, in place of the infinite dimensional analogue developed in~\cite{hairer_theory_2011}.

\section{Positivity of the top Lyapunov exponent}\label{sec:lyapunov}

\subsection{A continuous family of invariant measures}
We first formulate a condition which implies that a Markov process has a positive top Lyapunov exponent.

\begin{proposition}\label{prop.markov-furstenberg}
  Let $P(z,\mathrm{d}y)$ be a Markov kernel on a Radon space $(Z,\mathcal{M},\mu)$ where $\mu$ is stationary under $P$, that is,
  \[\int P(z,\mathrm{d}y) \; \mu(\mathrm{d}z) = \mu(\mathrm{d}y).\]
  Let $A$ be a cocycle on $Z$, identified with a measurable $A \colon Z \to \SL(d,\R)$. We also write $A_z$ to denote the induced map $A_z \colon S^{d-1} \to S^{d-1}$ given by $A_z \colon v \mapsto \frac{Av}{|Av|}$.

  Suppose $\nu \colon Z \to \mathcal{P}(S^{d-1})$ is a measurable family of probability measures on the sphere $S^{d-1}$ satisfying the twisted pullback condition
  \begin{equation}\label{eq:twisted-pullback}
    \nu_z = A(z)_*^{-1}\int \nu_y P(z, \dif y),
  \end{equation}
  where $A_* \nu$ denotes the pushforward measure of $\nu$ under the map $A$.

  Let $\lambda^+$ be the top Lyapunov exponent of the cocycle,
  \[\lambda^+ := \lim_{n \to \infty} \frac{1}{n} \E \log \| A(z_{n-1}) \cdots A(z_{1}) A(z_0)\|,\]
  where $z_0$ is distributed according to $\mu$ and the remaining $z_j$ are distributed according to the trajectory of the Markov process started from $z_0$. If $\lambda^+ = 0$, then the family is almost surely invariant:
    \[
      A(z_0)_*\nu_{z_0} = \nu_{z_1},\quad P(z_0,\mathrm{d}z_1) \mu(\mathrm{d}z_0)\text{-almost surely.}
    \]
\end{proposition}

Our goal is to reduce this result to~\cite[Proposition~2 and Theorem~3]{ledrappier_positivity_1986}, Ledrappier's powerful generalization of Furstenberg's criterion~\cite{furstenberg_noncommuting_1963}. We restate the result here for convenience.

\begin{proposition}[\cite{ledrappier_positivity_1986}]\label{prop.ledrappier-furstenberg}
  Let $(X,\mathcal{F},\rho)$ be a measure space, $T \colon X \to X$ and $B \colon X \to \SL(d,\R)$ be measurable maps such that $T_* \rho = \rho$. Define the extended dynamics $\hat T \colon X \times S^{d-1} \to X \times S^{d-1}$ by
  \[\hat T(x,v) = (Tx, B(x) v),\]
  abusing notation to let $B \in \SL(d,\R)$ act $S^{d-1} \to S^{d-1}$ as above. Let $\gamma \colon X \to \mathcal{P}(S^{d-1})$ be a measurable family of measures on $S^{d-1}$ such that $\gamma_x(\mathrm{d}v) \rho(\mathrm{d}x)$ is an invariant measure for $\hat T$, $\hat T_* \gamma_x \rho =  \gamma_x \rho$. Define
  \[\lambda^+ := \lim_{n \to \infty} \frac{1}{n} \int \log \|B(T^{n-1}x) \cdots B(x)\| \; \rho(\mathrm{d}x).\]
  Then if $\lambda^+=0$, we have
  \[B(x)_* \gamma_x = \gamma_{Tx} \quad \rho \text{-almost surely.}\]
\end{proposition}

We now argue that Proposition~\ref{prop.markov-furstenberg} is a consequence of Proposition~\ref{prop.ledrappier-furstenberg}.

\begin{proof}
  We assume the notation of Proposition~\ref{prop.markov-furstenberg} and assume $\lambda^+ = 0$.

  We need to construct $X, \rho, B, T, \gamma_x$ as in Proposition~\ref{prop.ledrappier-furstenberg}. We let
  \[X := Z^{\Z_{\leq 0}}\]
  and attach the usual product $\sigma$-algebra. We want $X$ to represent the space of trajectories of the Markov process given by $P$ from times $-\infty$ to $0$. We do this by attaching a measure $\rho$ to $X$ constructed by Kolmogorov's extension theorem. To construct $\rho$, it suffices to consistently specify marginal measures $\rho_n$ on $Z^{\{-n,\dots,0\}}$. We let $\rho_n$ be given as the law of $(z_{-n}, z_{-n+1},\dots,z_0)$, where $z_{-n}$ is distributed according to $\mu$ and $z_{-n+j}$ follows the trajectory of the Markov process given by $P$ started from $z_{-n}$. More precisely,
  \[\rho_n(\mathrm{d}z_{-n},\dots,\mathrm{d}z_0) = \mu(\mathrm{d}z_{-n}) P(z_{-n}, \mathrm{d}z_{-n+1}) \cdots P(z_{-1}, \mathrm{d}z_0).\]
  One can readily verify the $\rho_n$ form a consistent family using that $\mu$ is a stationary measure of $P$. Thus, $\rho$ is uniquely defined by the Kolmogorov extension theorem.

  We have constructed $X,\rho$; we now need to specify $B,T,
  \gamma_x$. We let
  \[B(x) = B(\dots,z_{-1},z_0) := {A(z_{-1})}^{-1}.\]
  Also, we define $T$ to be the backward shift, that is
  \[Tx = T(\dots,z_{-1},z_0) := (\dots,z_{-2}, z_{-1}).\]
  Lastly, we define
  \[\gamma_x = \gamma_{\dots,z_{-1},z_0} := \nu_{z_0}.\]

  We need to check that $T_* \rho = \rho$ and $\hat T_* \gamma_x \rho = \gamma_x \rho$. Before this computation, we first assume these conditions are verified and see that the second conclusion of Proposition~\ref{prop.markov-furstenberg} follows.

  Note that the top Lyapunov exponent in Proposition~\ref{prop.ledrappier-furstenberg} is given as
  \[\lambda^+ := \lim_{n \to \infty} \frac{1}{n} \int \log \|B(T^{n-1}x) \cdots B(x)\| \; \rho(\mathrm{d}x) =\lim_{n \to \infty} \frac{1}{n} \int \log \|A(z_{-n})^{-1} \cdots A(z_1)^{-1}\| \; \rho(\mathrm{d}x),\]
  which by the definition of $\rho$ and the stationarity of $\mu$ is the same as the definition of the top Lyapunov exponent in Proposition~\ref{prop.markov-furstenberg}. Therefore, we have $\lambda^+ =0$ and the assumptions of Proposition~\ref{prop.ledrappier-furstenberg} are fulfilled, so we conclude that
  \[B(x)_* \gamma_x = \gamma_{Tx} \quad \rho\text{-almost surely.}\]
  Unpacking notation, this says that
  \[A(z_{-1})_*^{-1} \nu_{z_0} = \gamma_{z_{-1}} \quad \mu(\mathrm{d}z_{-1})P(z_{-1}, \mathrm{d}z_0)\text{-almost surely,}\]
  which, up to re-indexing, is exactly the conclusion of Proposition~\ref{prop.markov-furstenberg}. It remains to check the two stationarity conditions above.

  First, we observe that $T_* \rho = \rho$ follows from checking the result on finite dimensional marginal laws and using the stationarity of $\mu$. Next, we check $\hat T_* \gamma_x \rho = \gamma_x \rho$; fix $\varphi \colon X \times S^{d-1} \to \R$ bounded and measurable. We get
  \begin{align*}
    \int \varphi(x,v) \; (\hat T_* \gamma_x \rho)(\mathrm{d}x,\mathrm{d}v) &= \int \varphi(\dots,z_{-2},z_{-1}, A(z_{-1})^{-1} v) \; \nu_{z_0}(\mathrm{d}v)\rho(\mathrm{d}x)
    \\&=  \int \varphi(\dots,z_{-2},z_{-1},  v) \; A(z_{-1})_*^{-1}\nu_{z_0}(\mathrm{d}v) P(z_{-1}, \mathrm{d}z_0)\rho(\dots,\mathrm{d}z_{-2}, \mathrm{d}z_{-1})
    \\&=  \int \varphi(\dots,z_{-2},z_{-1},  v) \; \nu_{z_{-1}}(\mathrm{d}v) \rho(\dots,\mathrm{d}z_{-2}, \mathrm{d}z_{-1})
    \\&= \int \varphi(x,v) \; \gamma_x(\mathrm{d}v) \rho(\mathrm{d}x),
  \end{align*}
  where we use the twisted pullback condition~\eqref{eq:twisted-pullback} for the penultimate equality.
\end{proof}

\begin{proposition}\label{prop:nu-continuity}
  There exists a locally Lipschitz family of probability measures
  \[
    \nu \colon \uspace \times M^J \to \mathcal{P}(S^{1})
  \]
  such that
  \[\nu_{u_0,x_0,A_0} = {(A_0)}_*^{-1}\int \nu_{u,x, A} \; P_1(u_0,x_0,\Id, \mathrm{d}u,\mathrm{d}x,\mathrm{d}A),\]
  where we metrize $\mathcal{P}(S^{1})$ with the $W^{-1,1}$ norm.
\end{proposition}

\begin{proof}
  Let $\nu_z^0$ be the uniform probability measure on $S^{d-1}$ for each $z = (u, x, A) \in \uspace \times M^J$, and define
  \[
    \nu_{(u, x, A)}^{n+1} := {A}_*^{-1} \int \nu_{z}^n \; P_1((u, x, \Id), \mathrm{d}z).
  \]
  We will first show that the family $z \mapsto \nu_z^n$ is locally Lipschitz in $(u, x, A)$, uniformly in $n$. Then we will apply the Arzela--Ascoli theorem to partial averages
  \[\sigma^n_z := \frac{1}{n}\sum_{j=1}^n\nu_z^j,\]
  to give a subsequence $n_j$ along which $\sigma^{n_j}_z$ converges pointwise to a locally Lipschitz limiting family of probability measures $z \mapsto \nu_z$. One can then verify that this limiting family satisfies
  \begin{equation}\label{eq:nu-invariance}
    \nu_{(u, x, A)} := {A}_*^{-1} \int \nu_{z} \; P_1((u, x, \Id), \mathrm{d}z),
  \end{equation}
  as desired.

  We now prove inductively that for each $(u, x, A)$,
  \begin{equation}\label{eq:nu-local-lipschitz}
    \|\nabla_{u,x,A} \nu_{u,x,A}^n\|_{W^{-1, 1}} \leq T(u, x, A) := Ke^{\beta V(u)}|A|^3,
  \end{equation}
  for some constants $K, \beta > 0$, where for $\nu : \uspace \times \T^2 \times \SL(2;\R) \to \mathcal{P}(S^1),$
  \[\|\nabla_{u,x,A} \nu\|_{W^{-1,1}} = \sup_{\|\varphi\|_{W^{1,\infty}(S^{1})} \leq 1} \left|\nabla_{u,x,A} \int \varphi(v) d\nu(v)\right|,\]
  and for $\varphi :  \uspace \times \T^2 \times \SL(2;\R) \to \R$,
  \[|\nabla_{u,x,A} \varphi| = \sup_{\|v\|_{H^5} + |y| + |B|\leq 1} \left| D_{v,y,B} \varphi\right|,\]
  where $D_{v,y,B} \varphi$ denotes the directional derivative of $\varphi$ in the $(v,y,B)$ direction.

  Assuming the bound~\eqref{eq:nu-local-lipschitz}, we fix a test function $\varphi \colon S^{1} \to \R$ with $\|\varphi\|_{W^{1, \infty}} \leq 1$ and write
  \begin{align*}
    &\left|\nabla_{u,x,A}\int \varphi(v) \; \nu^{n+1}_{u_0,x_0,A_0}(\mathrm{d}v)\right|\\ &\qquad = \left|\nabla_{u,x,A}\int \varphi\left(A_0(v)\right) \; \nu^n_{z}(\mathrm{d}v) P_1((u_0,x_0,\Id), \mathrm{d}z)\right|\\
    &\qquad \leq C|A_0|^3 + \left|\nabla_{u,x}\int \varphi\left(A_0(v)\right) \; \nu^n_{z}(\mathrm{d}v) P_1((u_0,x_0,\Id), \mathrm{d}z)\right|\\
                                                                       &\qquad \leq C|A_0|^3 + e^{\eta V(u_0)}\left(C(\eta, \gamma)\sqrt{P_1 |\psi|^2(u_0,x_0,\Id)} + \gamma \sqrt{P_1 |\nabla_{u,x,A} \psi|^2(u_0,x_0,\Id)}\right).
  \end{align*}
  where $\psi(z) := \int \varphi(A_0^{-1}(v)) \; \nu^n_z(\mathrm{d}v)$ and we used Proposition~\ref{prop:unit-time-smoothing} in the last inequality. Using that $|\nabla \varphi| \leq 1$, we note
  \[|\nabla_{u,x,A} \psi|(z) \leq C |A_0|^3\|\nabla_{u,x,A} \nu^n_z\|_{W^{-1,1}}.\]
  Using now that that $|\psi| \leq 1$, the inductive hypothesis, and the previous display, we have for any $\eta,\gamma>0$
  \begin{equation}\label{eq:inductive-step-1}
    \left|\nabla_{u,x,A}\int \varphi(v) \; \nu^{n+1}_{u_0,x_0,A_0}(\mathrm{d}v)\right| \leq |A_0|^3e^{\eta V(u_0)}\left(C(\eta,\gamma) + \gamma \sqrt{\int {T(z)}^2 \; \mathrm{d}P_1((u_0, x_0, \Id), z)}\right).
  \end{equation}
    Note that
    \begin{equation}
    \label{eq:A-1-moment-bound}
    \E[|A_1|^{n}] \leq \E \exp\Big(n \int_0^1 \|\nabla u_t\|_{L^\infty}\,\dif t\Big) \leq  C(n,\zeta) \exp\big(\zeta V(u_0)\big),
    \end{equation}
    where we use Corollary~\ref{cor:C-alpha-moments} for the final inequality.
  Then
    \begin{align*}
        \int {T(z)}^2 \; P_1((u_0, x_0, \Id), \mathrm{d}z)&= K^2\, \E[e^{2\beta V(u_1)} |A_1|^6]\\
        &\leq  K^2\, \E\left[\exp\big(4\beta V(u_1)\big)\right]^{1/2}\E[|A_1|^{12}]^{1/2}
        \\&\leq C(\beta) K^2\, \exp\big(2 \delta \beta V(u_0)\big),
    \end{align*}
    for some $\delta \in (0,1)$, where we use Corollary~\ref{cor:super-lyapunov} and~\eqref{eq:A-1-moment-bound} with $\zeta$ chosen sufficiently small.

  Thus by~\eqref{eq:inductive-step-1}, we have that
    \begin{align*}
     \left|\nabla_{u,x,A}\int \varphi(v) \; \nu^{n+1}_{u_0,x_0,A_0}(\mathrm{d}v)\right| &\leq|A_0|^3e^{\eta V(u_0)}\left(C(\eta,\gamma) + \gamma C(\beta)K \exp\big( \delta \beta V(u_0)\big)\right)
    \end{align*}
  We are now in a position to choose $\eta,\beta,\gamma$, and $K$. We take $\beta = \frac{1}{4},$ which ensures we could apply Corollary~\ref{cor:super-lyapunov} above. Then choose $\eta>0$ so that $\delta \beta +\eta = \beta$. Then we choose $\gamma$ so that $\gamma C(\beta) = \frac{1}{2}$. Under these choices, we get that
  \[     \left|\nabla_{u,x,A}\int \varphi(v) \; \nu^{n+1}_{u_0,x_0,A_0}(\mathrm{d}v)\right| \leq\Big(C + \frac{K}{2}\Big) |A_0|^3e^{\beta V(u_0)}.\]
  The induction then closes, choosing $K =2C$.
\end{proof}

Combining Proposition~\ref{prop.markov-furstenberg}, Proposition~\ref{prop:nu-continuity}, and the cocycle property of the Jacobian process, we obtain:

\begin{proposition}\label{prop:almost-sure-nu-invariance}
  Consider the Jacobian process $\{(u_t, x_t, A_t)\}_{t \geq 0}$ with initial data $(u_0, x_0)$ distributed according to the invariant measure of the one-point process and $A_0 = \Id$. If the top Lyapunov exponent
  \begin{equation*}
    \lambda^+ := \lim_{t \to \infty} t^{-1}\E[\log |A_t|]
  \end{equation*}
  is zero, then there is a continuous map $\nu \colon \uspace \times M^1 \to W^{-1, 1}(S^{d-1})$ satisfying the invariance formula
  \begin{equation}
    \label{eq:almost-sure-nu-invariance}
    {(A_t)}_*\nu_{u_t, x_t} = \nu_{u_0, x_0}
  \end{equation}
  $\mu^1$-almost surely for $t \in \N$.
\end{proposition}
\begin{proof}
  We apply Proposition~\ref{prop.markov-furstenberg} to the family of measures given by Proposition~\ref{prop:nu-continuity} on the space $Z := \uspace \times M^J$, with the ``restarted'' transition probabilities $P^Z(u, x, A) := P^J_1(u, x, \Id)$. Since $P^Z(u, x, A)$ does not depend on $A$, it is clear that $P^Z (\mu^1 \times \delta_{I_d})$ is the unique stationary measure for $P^Z$. The cocycle on $Z$ is the projection $(u, x, A) \mapsto A$. The hypotheses of Proposition~\ref{prop.markov-furstenberg} are therefore satisfied and the result follows, noting that the constructed family of probability measures satisfies $A_*\nu_{u, x, A} = \nu_{u, x, \Id}$, so we put $\nu_{u, x} := \nu_{u, x, \Id}$.
\end{proof}
\subsection{Contradicting Furstenberg's criterion}
This section directly follows the arguments of~\cite[Sections~4.3~and~7]{bedrossian_lagrangian_2022}, using our Corollary~\ref{eq:almost-sure-nu-invariance} in place of the strong Feller assumption. We start by quoting a useful result.

\begin{theorem}[Classification of invariant families, {{\cite[Theorem~4.7]{bedrossian_lagrangian_2022}}}]\label{bbps-classification}
  Let $(u, x) \mapsto \nu_{u,x}$ be the continuous family of probability measures satisfying the invariance formula~\ref{eq:almost-sure-nu-invariance} almost surely. Then one of the following alternatives holds.
  \begin{enumerate}
  \item There is a continuously-varying inner product $\langle \cdot, \cdot \rangle_{u,x}$ on $\R^2$ with the property that for any $(u_0, x_0) \in \uspace \times M^1$ and $A_0 := \Id$, the map $A_t \colon (\R^2, \langle \cdot, \cdot \rangle_{u_0,x_0}) \to (\R^2, \langle \cdot, \cdot \rangle_{u_t,x_t})$ is an isometry almost surely, where $(u_t, x_t, A_t)$ is distributed according to the transition probability of the Jacobian process.
  \item For some $p \in \N$, there are measurably-varying assignments from $z_0 \in \uspace \times M^1$ to one-dimensional linear subspaces $E^i_{z_0} \subsetneq \R^2$, for each $1 \leq i \leq p$, with the property that, if $z_t$ is distributed according to the transition measure $P^J_t(z_0)$, then $A_tE^i_{z_0} = E^{\pi(i)}_{z_t}$ for some permutation $\pi$. Furthermore, the collection $(E^i_z)$ is locally (on small neighborhoods) continuous up to re-labelling.
  \end{enumerate}
\end{theorem}
We conclude with the main result of this section, that the top Lyapunov exponent of the derivative cocycle is positive. Although we only need the following result in expectation (but for every initial data), a $\mu^P$-almost sure version of the  conclusion follows by the ergodic theorem.
\begin{theorem}\label{thm:lyapunov-positive}
  There is $\lambda^+ > 0$ such that, for every initial $(u_0, x_0, v_0) \in \uspace \times M^P$ and $A_0 := \Id$,
  \begin{equation}\label{eq:lyapunov-positive}
    \lambda^+ = \lim_{t \to \infty} t^{-1}\E\left[\log |A_t|\right] = \lim_{t \to \infty} t^{-1} \E\left[\int_0^t v_s \cdot \nabla u_s(x_s) v_s \, \dif s\right].
  \end{equation}
\end{theorem}
\begin{proof}
  We first note that the second equality in~\eqref{eq:lyapunov-positive} follows from the fact that
  \[
    \log |A_tv_0| = \int_0^t v_s \cdot \nabla u_s(x_s) v_s \, \dif s
  \]
and Corollary~\ref{cor:projective-ergodicity} (which shows that the limit on the right-hand side exists and is independent of $v_0$). Indeed, we can rewrite the right-hand side of~\eqref{eq:lyapunov-positive} as
  \[
    \lim_{t \to \infty} \int_{\uspace \times M^P} v \cdot \nabla u(x) v \; \mu^P_t(\dif u, \dif x, \dif v),
  \]
  where $\mu^P_t$ denotes the averaged measure on the projective process up to time $t$ starting from $\delta_{(u_0, x_0, v_0)}$. By Corollary~\ref{cor:projective-ergodicity} and the fact that $\mu^P_t(\{\|\nabla u\|_{L^\infty} \geq R\} \times M^P) \leq C\exp(-R)$, the limit exists and is equal to
  \[
    \int_{\uspace \times M^P} v \cdot \nabla u(x) v \; \mu^P(\dif u, \dif x, \dif v),
  \]
  which is visibly deterministic and independent of initial data.

  Assume for contradiction that $\lambda^+ = 0$. By Corollary~\ref{prop:almost-sure-nu-invariance}, we obtain a continuous family $\nu \colon \uspace \times M^1 \to W^{-1, 1}(S^{d-1})$ of probability measures satisfying~\eqref{eq:almost-sure-nu-invariance}. We now obtain a contradiction using an argument identical to that in~\cite{bedrossian_lagrangian_2022}. Their classification of invariant fiber measure families and approximate control results apply verbatim (in fact, our setting makes this argument slightly easier because only finitely many modes are forced, so there are no regularity concerns).

  To conclude, we apply Theorem~\ref{bbps-classification} and reach a contradiction, using controllability properties of both the projective and Jacobian process.

  Indeed, if there is a continuously-varying inner product $\langle \cdot, \cdot \rangle_{u,x}$, then the quantities
  \begin{equation*}
    M := \max \left\{ \langle v, v \rangle_{u,x} \mid \|u\|_{H^6} \leq 1, \; x \in \xspace, \; v \in S^1 \right\}
  \end{equation*}
  and
  \begin{equation*}
    m := \min \left\{ \langle v, v \rangle_{u,x} \mid \|u\|_{H^6} \leq 1, \; x \in \xspace, \; v \in S^1 \right\}
  \end{equation*}
  are positive and finite. Then Lemma~\ref{lem:Jacobian-control} shows that, for $(u_0, x_0) := (0, 0)$ the event \[ \left\{\|u_1\|_{\uspace} \le 1 \text{ and } |A_1| > Mm^{-1}\right\} \] has positive probability; in this event, for some (random) $v \in S^1$ we have $\langle v, v \rangle_{u_0, x_0} < \langle A_1 v, A_1 v \rangle$ and therefore $A_1$ is not an almost-sure isometry, contradicting Theorem~\ref{bbps-classification}.

  On the other hand, assume that there is a continuously-varying collection $(E^i_{u,x})_{1 \leq i \leq p}$ of lines which are almost surely invariant under the projective dynamics. By continuity and the fact that $p$ is finite, there is some $\tilde{v} \in S^1$ and $\varepsilon > 0$ such that, for all $\|u\|_{\uspace},|x| \leq \varepsilon$ we have $\dist(\tilde{v}, E^i_{u,x}) \geq 2\varepsilon$. Fix initial data $(u_0, x_0) := (0, 0)$ and $v_0 \in E^0_{u_0, x_0}$. Lemma~\ref{lem:proj-control} shows that we have $\|u_1\|_{\uspace}, |x_1|, |v_1-\tilde{v}| \leq \varepsilon$ with positive probability, which implies that the collection of lines is not almost surely invariant, since our choice of $\tilde{v}$ ensures $v_1 \not\in E^i_{u_1, x_1}$ for any $1 \leq i \leq p$.
\end{proof}
\section{Exponential mixing for the two-point process}
\label{s:exponential-mixing}
In this section we apply the weak Harris framework of Hairer, Mattingly, and Scheutzow~\cite{hairer_asymptotic_2011} to prove that the two-point process mixes exponentially fast, with errors measured in a Wasserstein sense.

Our first step is to build a Lyapunov function $W(u, x, y)$ for the two-point process, using the super-Lyapunov function $\exp(V(u))$ for the base process and the positivity of the top Lyapunov exponent for the derivative cocycle from Theorem~\ref{thm:lyapunov-positive}.

First, we show that a drift function with finitely many possible waiting times induces another drift function with a single waiting time.

\begin{lemma}\label{lem:different-waiting-lyapunov}
  Let $z_0, z_1, \dots$ be a Markov process on a set $Z$ with transition kernel $P_t \colon Z \to \mathcal{P}(Z)$ for $t \in \N$. Let $V \colon Z \to \R_{\geq 1}$, $\alpha \in (0, 1)$, $C > 0$, and $T \in \N$ be such that, for each $z_0 \in Z$, there is some $n \in \{1, \dots, T\}$ such that $\E[V(z_n)] \leq \alpha^n V(z_0) + C$. Then the function
  \[
    W := \min_{1 \leq n \leq T} \alpha^{-n}\E[P_nV]
  \]
  satisfies $\E[P_1W] \leq \alpha W + \alpha^{-T}C$.
\end{lemma}
\begin{proof}
  Let $z_0 \in Z$ and let $1 \leq n \leq T$ be such that $W(z_0) = \alpha^{-n}\E[P_nV(z_0)]$. If $n > 1$, then
  \begin{align*}
    \E[P_1W(z_0)] &\leq \E\left[P_1\left(\alpha^{-(n-1)}\E[P_{n-1}V]\right)(z_0)\right]\\
               &= \alpha^{-n+1}\E[P_nV(z_0)]\\
               &= \alpha W(z_0).
  \end{align*}
  On the other hand, if $n=1$ then
  \begin{align*}
    \E[P_1W(z_0)] &= \E\left[P_1\left(\min_{1 \leq n \leq T}\alpha^{-n}\E[P_nV]\right)(z_0)\right]\\
               &\leq \E[P_1V(z_0) + \alpha^{-T}C]\\
               &= \alpha W(z_0) + \alpha^{-T}C.\qedhere
  \end{align*}
\end{proof}
Next, we prove a weak bound in the opposite direction of Jensen's inequality.
\begin{lemma}\label{lem:exponent-to-drift-function}
  Let $X$ be a real-valued random variable with $\E[\exp(|X|)] \leq C$ for some $C > 0$ and $\E[X] \geq 1$. Then there is some $p_0(C) > 0$ such that $\E[\exp(-pX)] < 1-\frac{p}{2}$ for all $p \in (0, p_0)$.
\end{lemma}
\begin{proof}
  Write $\E[\exp(-pX)] = 1 - p\E[X] + \sum_{n=2}^\infty \frac{{(-p)}^{n}}{n!}\E[X^n] = 1 - p\E[X] + O(p^2)$, using the moment bound on $\exp(|X|)$.
\end{proof}

We are now ready to prove the existence of a Lyapunov function for the two-point process.
\begin{proposition}\label{prop:two-point-drift-function}
  For any $\beta > 0$, there is a continuous function $W \colon \uspace \times M^2 \to \R_{\geq 1}$, with
  \begin{equation}\label{eq:W-bound}
    W(u, x, y) \leq C|x-y|^{-\beta}\exp(\beta V(u)),
  \end{equation}
  and constants $\alpha \in (0, 1)$ and $C > 0$ such that the drift condition
  \begin{equation}
    \label{eq:W-drift-condition}
    P_1W(u, x, y) \leq \alpha W(u, x, y) + C
  \end{equation}
  is satisfied.
\end{proposition}
\begin{proof}
  Let $f(u, x, y) := |x-y|^{-1}\exp(V(u))$.

  \textbf{Step 0.}
  It suffices to find such a $W$ which satisfies~\eqref{eq:W-drift-condition} outside a sublevel set $f^{-1}([0, C'])$. for some constant $C' > 0$. Indeed, using~\eqref{eq:W-bound}, choosing the constant in~\eqref{eq:W-drift-condition} sufficiently large shows that~\eqref{eq:W-drift-condition} is trivially satisfied on $f^{-1}([0, C'])$.

  \textbf{Step 1.}
  First, we show that $\E[f(u_1, x_1, y_1)] \leq \frac{1}{2} f(u_0, x_0, y_0)$ whenever $V(u_0) \geq C_V$, for some sufficiently large $C_V > 0$. Indeed,
  \begin{align*}
    \E[f(u_1, x_1, y_1)] &= \E\left[|x_1-y_1|^{-1}\exp(V(u_1))\right]\\
                         &\leq |x_0-y_0|^{-1}\E\left[\left(\int_0^1 \|\nabla u_t\|_{L^\infty} \, \dif t\right)\exp(V(u_1))\right]\\
                         &\leq C|x_0-y_0|^{-1}\exp(\delta V(u_0))\\
                         &\leq \frac{1}{2}f(u_0, x_0, y_0),
  \end{align*}
  where the penultimate inequality follows from Corollary~\ref{cor:super-lyapunov}, Corollary~\ref{cor:C-alpha-moments}, and Cauchy--Schwarz and the final inequality holds when we choose $C_V$ large enough so that $\exp((1-\delta) V(u_0)) \geq \exp((1-\delta)C_V) \geq 2C$.

  \textbf{Step 2.}
  Next, we find a uniform bound on the time to witness a positive Lyapunov exponent, for initial velocity $u_0$ bounded. Precisely, we show that for any $\delta > 0$, there is some $T \in \N$ such that, whenever $V(u_0) \leq C_V$, there is some integer $1 \leq n \leq T$ such that
  \begin{equation}\label{eq:uniform-growth-delta}
    \E\left[V(u_0) - V(u_n) + \int_0^n v_t \cdot \nabla u_t(x_t)v_t \, \dif t\right] \geq \delta,
  \end{equation}
  where $x_0 \in \T^2$ and $v_0 \in S^1$.
  Theorem~\ref{thm:lyapunov-positive} shows that, for each $(u_0, x_0, v_0) \in \uspace \times M^P$ and $K > 0$, there is a time $\tau = \tau(u_0, x_0, v_0, K) \in \N$ such that
  \begin{equation*}
    \E\left[\int_0^\tau v_t \cdot \nabla u_t(x_t)v_t \, \mathrm{d}t\right] \geq \frac{\tau\lambda^+}{2} \geq K,
  \end{equation*}
  where the expectation is taken over trajectories $(u_t, x_t, v_t)$ of the projective process. For any $(u_0, x_0, v_0) \in \uspace \times M^P$ with $V(u_0) \leq C_V$ and other $(u_0', x_0', v_0')$ with $V(u_0') \leq C_V$, we can estimate
  \begin{align}
    &\left|\E\left[\int_0^\tau v_t \cdot \nabla u_t(x_t)v_t \, \mathrm{d}t\right] - \E\left[\int_0^\tau v_t' \cdot \nabla u_t'(x_t')v_t' \, \mathrm{d}t\right]\right|\notag\\
    &\quad\leq \E\left[\int_0^\tau |v_t-v_t'|\,\left(|\nabla u_t(x_t)| + |\nabla u_t'(x_t')|\right) + |\nabla u_t(x_t)-\nabla u_t'(x_t')|\, \mathrm{d}t\right]\notag\\
    &\quad\leq \int_0^\tau {\E\left[{|v_t-v_t'|}^2\right]}^{1/2}{\E\left[{\left(|\nabla u_t(x_t)| + |\nabla u_t'(x_t')|\right)}^2\right]}^{1/2} + \E\left[|\nabla u_t(x_t)-\nabla u_t'(x_t')|\right]\, \mathrm{d}t.\label{eq:vuv-bound}
  \end{align}
  Interpolating between the bounds from Proposition~\ref{prop:omega-bounds} and $H^1$-stability of~\eqref{eq:sns-velocity-form}, we conclude that \[ \E\left[\exp\left(C\int_0^t\|\nabla u_s-\nabla u_s'\|_{L^\infty} \, \dif s\right)\right] \to 0 \] as $\|u_0-u_0'\|_{H^1} \to 0$, so the first term in~\eqref{eq:vuv-bound} tends to zero as $|v_0-v_0'| + |x_0-x_0'| + \|u_0-u_0'\|_{H^1} \to 0$, using the equation for $v_t$ and Gr\"onwall's inequality. For the second term, we write
  \begin{equation*}
    \int_0^\tau \E\left[|\nabla u_t(x_t)-\nabla u_t'(x_t')|\right]\, \mathrm{d}t \leq \int_0^\tau \E\left[\|\nabla u_t'\|_{C^\alpha}|x_t-x_t'|^\alpha + \|\nabla u_t - \nabla u_t'\|_{L^\infty}\right]\, \mathrm{d}t
  \end{equation*}
  and observe that the right-hand side similarly tends to zero as $|x_0-x_0'| \to 0$ and $\|u_0-u_0'\|_{H^1} \to 0$, using Corollary~\ref{cor:C-alpha-moments} to bound $u_t$ in $C^\alpha$.

  We conclude that there is $\varepsilon = \varepsilon(\tau, u_0, x_0, v_0)$ such that, if $|x_0-x_0'|, |v_0-v_0'|, \|u_0-u_0'\|_{H^1} \leq \varepsilon$, then
  \begin{equation*}
    \E\left[\int_0^{\tau} v_t \cdot \nabla u_t(x_t)v_t \, \mathrm{d}t\right] \geq \frac{\tau \lambda^+}{4}.
  \end{equation*}
  On the other hand, we have $\E[V(u_0) - V(u_\tau)] \geq -C$ for some constant $C > 0$, not depending on $\tau$. Taking $K$ large enough so that $\frac{\tau\lambda^+}{8} \geq \frac{K}{4} \geq C \lor \delta$ and using relative compactness of the sublevel set $S := \{(u_0, x_0, v_0) \mid V(u_0) \leq C_V\}$ in $H^1$, we conclude that we can cover $S$ with finitely many $H^1$-balls, each of which satisfies~\eqref{eq:uniform-growth-delta}.

  \textbf{Step 3.}
  We now approximate the two-point process by the projective process.
  Let $u_0$ be such that $V(u_0) \leq C_V$ and let $x_0 \in \T^2$ and $v_0 \in S^1$. Let $y_0 = x_0 + \varepsilon v_0$, for some $\varepsilon \in (0, 1)$. We observe:
  \begin{align*}
    \log |y_n - x_n| - \log |y_0-x_0| &= \int_0^n \frac{(y_t-x_t) \cdot \left(u_t(y_t) - u_t(x_t)\right)}{|y_t-x_t|^2} \, \mathrm{d}t\\
                                      &=: \int_0^n v_t \cdot \nabla u_t(x_t)v_t \, \mathrm{d}t + R,
  \end{align*}
  where
  \begin{align*}
    \E\left[|R|\right] &\leq \E\left[\int_0^n \left|v_t-\frac{y_t-x_t}{|y_t-x_t|}\right|\frac{|u_t(y_t)-u_t(x_t)|}{|y_t-x_t|} + \left|v_t \cdot \nabla u_t(x_t) - \frac{u_t(y_t)-u_t(x_t)}{|y_t-x_t|}\right|\, \mathrm{d}t\right]\\
                       &\leq \int_0^n 2\E\left[\left|v_t-\frac{y_t-x_t}{|y_t-x_t|}\right|\|\nabla u_t\|_{L^\infty}\right] + \E\left[\|\nabla u_t\|_{C^\alpha}|y_t-x_t|^\alpha\right] \, \mathrm{d}t.
  \end{align*}

  The second term above tends to zero, uniformly in $u_0$, as $\varepsilon \to 0$. To bound the first term, we write $R_v^t := \left|v_t-\frac{y_t-x_t}{|y_t-x_t|}\right|$ and compute
  \begin{align*}
    R_v^t &\leq \int_0^t \left|\Pi_{v_s}(v_s \cdot \nabla u_s(x_s)) - \Pi_{{(y_s-x_s)}} \left(\frac{u_s(y_s)-u_s(x_s)}{|y_s-x_s|}\right)\right| \, \mathrm{d}s\\
          &\leq \int_0^t 2R_v^s\|\nabla u_s\|_{L^\infty} +  \|\nabla u_s\|_{C^\alpha}|y_s-x_s|^\alpha\, \mathrm{d}s.
  \end{align*}
  Using Gr\"onwall's inequality and Corollary~\ref{cor:C-alpha-moments} shows that (for example) we have $\E[{(R_v^t)}^2] \to 0$ as $\varepsilon \to 0$, so $E[|R|] \to 0$ as well.

  We conclude that there is $\varepsilon > 0$ sufficiently small such that, for every $u_0$ with $V(u_0) \leq C_V$ and $x_0, y_0 \in \T^2$ with $|x_0-y_0| \leq \varepsilon$, we there is some $1 \leq n \leq T$ such that
  \begin{equation*}
    \E\left[\log\left(\frac{f(u_0, x_0, y_0)}{f(u_n, x_n, y_n)}\right)\right] = \E[V(u_0) - V(u_n) + \log |y_n - x_n| - \log |y_0 - x_0|] \geq \frac{\delta}{2},
  \end{equation*}
  where $\delta > 0$ is the constant from Step 2.

  \textbf{Step 4.}
  For each $u_0, x_0, y_0$ with $V(u_0) \leq C_V$ and $|x_0-y_0| \leq \varepsilon$, we choose $\delta = 2$ and apply Lemma~\ref{lem:exponent-to-drift-function} to the random variable $X := \log\left(\frac{f(u_0, x_0, y_0)}{f(u_n, x_n, y_n)}\right)$ to find $p(n, C_V) > 0$ such that $\E[{f(u_n, x_n, y_n)}^p] \leq (1-\frac{p}{2}){f(u_0, x_0, y_0)}^p$. Taking the smallest $p$ for any $1 \leq n \leq T$, letting $\alpha := 1-\frac{p}{2}$, and applying Lemma~\ref{lem:different-waiting-lyapunov}, we conclude that the function
  \begin{equation*}
    W(u_0, x_0, y_0) := \min_{1 \leq n \leq T} \alpha^{-n}\E[{f(u_n, x_n, y_n)}^p]
  \end{equation*}
  is a Lyapunov function for the two-point process.
\end{proof}

With our Lyapunov function $W$ in hand, we introduce a distance on $\uspace$ which extends to a Wasserstein distance between probability measures on the same space.
\begin{definition}\label{def:d_V}
  For $z_1, z_2 \in \uspace \times M^2$, let
  \[
    d_V(z_1, z_2) := \inf \left\{ \int_0^1 \|\gamma'(t)\|_{H^5,2}\exp(V(\gamma(t))/2) \; \mathrm{d}t : (\gamma(0), \gamma(1)) = (z_1, z_2)\right\},
  \]
  where the infimum is taken over $\gamma \colon [0, 1] \to \uspace \times M^2$ Lipschitz.
\end{definition}
\begin{definition}\label{def:Wasserstein}
  If $X$ is a measurable space and $d \colon X \times X \to \R_\geq 0$ and $\mu_1, \mu_2 \in \mathcal{P}(X)$, we write
  \[
    d(\mu_1, \mu_2) := \inf_{\mu \in \Gamma(\mu_1, \mu_2)} \int d(x, y) \, \dif \mu(x, y),
  \]
  where $\Gamma(\mu_1, \mu_2)$ denotes the set of couplings of $\mu_1$ and $\mu_2$, that is, probability measures on $X \times X$ which have marginals $\mu_1$ and $\mu_2$. We abuse notation and identify a point $x \in X$ with the measure $\delta_x$, so (for example) $d(\mu, x)$ means $d(\mu, \delta_x)$.
\end{definition}

We will also need a smoothing estimate for the two-point process, which follows from Proposition~\ref{prop:unit-time-smoothing}. It is important to note that the estimate respects the natural scaling of the two-point process; see the definition of $\|\cdot\|_{H^5,2}$ below. In particular, the constant $C$ appearing below is \textit{independent of $(x,y)\in M^2$.} We defer the proof to Section~\ref{sec:malliavin}.
\begin{lemma}\label{lem:two-point-smoothing}
  For each $\eta, \gamma \in (0, 1)$, there is $C(\eta, \gamma) > 0$ such that, for each Fr\'echet differentiable observable $\varphi \colon \uspace \times M^2 \to \R$ and each $(u,x,y) \in \uspace \times M^2$, we have
  \begin{equation*}
    \|\nabla P_1 \varphi(u, x, y)\|_{H^5,2} \leq \exp(\eta V(u))\left(C\sqrt{P_1|\varphi|^2(u, x, y)} + \gamma\sqrt{P_1\|\nabla \varphi\|_{H^5,2}^2(u, x, y)}\right),
  \end{equation*}
  where
  \[
    \|v\|_{H^5,2} := \|v_u\|_{H^5} + |v_x + v_y| + \frac{|v_x-v_y|}{|x-y|}
  \]
  for any $v = (v_u, v_x, v_y) \in \uspace \times \R^4$.
\end{lemma}

\begin{proposition}[A contractive distance]\label{prop:contraction}
  For any $\beta > 0$, let
  \[
    d^\beta((u_1, x_1, y_1), (u_2, x_2, y_2)) := 1 \land \beta^{-1}d_V((u_1, x_1, y_1), (u_2, x_2, y_2)).
  \]
  There is a constant $\beta_0 > 0$ such that, for all $\beta \in (0, \beta_0)$, if $d^\beta((u_1, x_1, y_1), (u_2, x_2, y_2)) < 1$, then
  \[
    d^\beta(P_1^2(u_1, x_1, y_1), P_1^2(u_2, x_2, y_2)) \leq \frac{1}{2} d^\beta((u_1, x_1, y_1), (u_2, x_2, y_2)),
  \]
  where $P_t^2$ denotes the two-point transition kernel.
\end{proposition}
\begin{proof}
  This is a straightforward consequence of Lemma~\ref{lem:two-point-smoothing}; we follow the proof of Proposition 5.5 of~\cite{hairer_asymptotic_2011}. Note that, for two probability measures $\mu_1, \mu_2$ on the two-point process,
  \begin{equation*}
    d^\beta(\mu_1, \mu_2) = \sup_{\|\varphi\|_{\Lip_{d^\beta}} \leq 1} \left|\int\varphi \, \mathrm{d}\mu_1 - \int \varphi \, \mathrm{d}\mu_2\right|,
  \end{equation*}
  where
  \begin{equation*}
    \|\varphi\|_{\Lip_{d^\beta}} := \sup_{z_1, z_2 \in \uspace \times M^2} \frac{|\varphi(z_2)-\varphi(z_1)|}{d^\beta(z_1, z_2)} \lor \|\varphi\|_{L^\infty}.
  \end{equation*}
  Fix $(u_1, x_1, y_1), (u_2, x_2, y_2) \in \uspace \times M^2$ with $d^\beta((u_1, x_1, y_1), (u_2, x_2, y_2)) < 1$. Fix an arbitrary Lipschitz $\gamma \colon [0, 1] \to \uspace \times M^2$ with $(\gamma(0), \gamma(1)) = ((u_1, x_1, y_1), (u_2, x_2, y_2))$.
  For each $\varphi \colon \uspace \times M^2 \to [-1, 1]$ with $\|\varphi\|_{\Lip_{d^\beta}}\leq 1$, we use Lemma~\ref{lem:two-point-smoothing} and Corollary~\ref{cor:super-lyapunov} to compute
  \begin{align*}
    &|P_1\varphi(u_1, x_1, y_1) - P_1\varphi(u_2, x_2, y_2)| \\
                       &\qquad \leq \int_0^1 \left|\gamma'(t) \cdot \nabla P_1\varphi(\gamma(t), x_1, y_1)\right| \; \mathrm{d}t\\
                       &\qquad \leq \int_0^1 \|\gamma'(t)\|_{H^5,2}\Big| e^{\eta V(\gamma(t))}\Big(C_\zeta\sqrt{P_1|\varphi|^2(\gamma(t), x_1, y_1)} + \zeta \sqrt{P_1\|\nabla \varphi\|_{H^5,2}^2(\gamma(t), x_1, y_1)}\Big)\Big| \; \mathrm{d}t\\
                       &\qquad \leq \int_0^1 \|\gamma'(t)\|_{H^5,2} \left| e^{\eta V(\gamma(t))}\Big(C_\zeta + \zeta \beta^{-1}C\exp(\delta V(\gamma(t))/2)\Big)\right| \; \mathrm{d}t\\
                       &\qquad \leq \int_0^1 \|\gamma'(t)\|_{H^5,2} \left| e^{\eta V(\gamma(t))}\Big(C + \frac{1}{4} \beta^{-1}\exp(\delta V(\gamma(t))/2)\Big)\right| \; \mathrm{d}t
  \end{align*}
  for some constant $\delta < 1$, where we use the super-Lyapunov property, noting that $\|\nabla \varphi\|_{H^5,2}(u, x, y) \leq \exp(V(u)/2)$ by the assumption that $\|\varphi\|_{\Lip_{d^\beta}} \leq 1$ and we choose $\zeta \leq \frac{1}{4C}$. Choosing $\eta$ small enough so that $\eta + \delta/2 \leq \frac{1}{2}$ and $\beta$ small enough so that $C + \frac{1}{4}\beta^{-1} \leq \frac{1}{2}\beta^{-1}$, we take a supremum over $\varphi$ and an infimum over $\gamma$ to conclude.
\end{proof}

\begin{definition}[$d$-small sets]
  Given a distance $d \colon {(\uspace \times M^2)}^2 \to [0, 1]$, a subset $S \subseteq \uspace \times M^2$ is called $d$-small if there are $\varepsilon > 0$ and $t \in \N$ such that $d(P_t(z_1), P_t(z_2)) \leq 1-\varepsilon$ for every $z_1, z_2 \in S$.
\end{definition}

\begin{proposition}[Small sublevel sets of $W$]\label{prop:small-set}
  For any $K, \beta > 0$, the sublevel set
  \[
    S_K := \{(u, x, y) \in \uspace \times M^2 \mid W(u, x, y) \leq K\}
  \]
  is $d^\beta$-small. Here, $W$ is the Lyapunov function for the two-point process defined in Proposition~\ref{prop:two-point-drift-function} and $d^\beta$ is the distance-like function defined in Proposition~\ref{prop:contraction}.
\end{proposition}
\begin{proof}
  First, by our construction, observe that $W(u_0, x_0, y_0) \leq K$ implies $V(u_T), |x_T-y_T|^{-1} \leq K_V$ with probability at least $\frac{1}{2}$ for some constant $K_V > 0$. We therefore work in this event and assume that we are given $(u_0, x_0, y_0), (u'_0, x'_0, y'_0) \in \uspace \times M^2$ which satisfy
  \begin{equation}\label{eq:u-x-y-bounds}
    V(u_0), |x_0-y_0|^{-1}, V(u'_0), |x'_0-y'_0|^{-1} \leq K_V
  \end{equation}
  without loss of generality.

  By Lemma~\ref{lem:2pt-control}, we observe that there is $\tilde{y} \in \xspace \setminus \{0\}$ such that for every $\varepsilon > 0$ there is a time $T > 0$ and $\delta > 0$ such that
  \[
    \P[\|u_T\|_{\uspace}, |x_T|, |y_T-\tilde{y}| \leq \varepsilon] > \delta,
  \]
  and the same holds for initial data $(u'_0, x'_0, y'_0)$. Therefore, there is a coupling where
  \[
    \P[\|u_T-u'_T\|_{\uspace}, |x_T-x'_T|, |y_T-y'_T| \leq 2\varepsilon] > \delta.
  \]
  Choosing $\varepsilon > 0$ sufficiently small relative to $\beta$ (possibly making $\delta$ smaller) shows that
  \[
    \P[d^\beta((u_T, x_T, y_T), (u'_T, x'_T, y'_T)) < 1-\delta] > \delta.
  \]
  We conclude from the previous display that $d^\beta(P_T(u_0, x_0, y_0), P_T(u'_0, x'_0, y'_0)) < 1 - \delta^2$ as desired.
\end{proof}

\begin{theorem}[Exponential mixing for the two-point process]\label{thm:2pt-mixing}
  There are $\alpha \in (0, 1)$ and $\beta > 0$ and $t_0 \geq 1$ such that for any two probability measures $\mu, \nu$ on the two-point process, we have the contraction estimate $d(P_{t_0}\mu, P_{t_0}\nu) \leq \alpha d(\mu, \nu)$, where
  \[
    d((u_1, x_1, y_1), (u_2, x_2, y_2)) := \sqrt{d^\beta((u_1, x_1, y_1), (u_2, x_2, y_2))(1 + W(u_1, x_1, y_1) + W(u_2, x_2, y_2))},
  \]
  where $d^\beta$ is the distance-like function in Proposition~\ref{prop:contraction} and $W$ is the Lyapunov function for the two-point process from Proposition~\ref{prop:two-point-drift-function}.

  In particular, $\mu^b \times \Leb$ is the unique invariant measure of the two-point process, where $\mu^b$ is the invariant measure for the base process and $\Leb$ is the Lebesgue measure on $M^2$, and we have the exponential mixing rate $d(P_{nt_0} \nu, \mu^b \times \Leb) \leq \alpha^{n}d(\nu, \mu^b \times \Leb)$ for all probability measures $\nu$ and $n \in \N$.

  We also note that inequalities~\eqref{eq:Lipschitz-domination} and~\eqref{eq:d-L2} hold, and therefore our choice of $d$ satisfies the hypotheses of Proposition~\ref{prop:2pt-implies-scalar}.
\end{theorem}
\begin{proof}
  This is a direct application of~\cite[Theorem~4.8]{hairer_asymptotic_2011}; Proposition~\ref{prop:two-point-drift-function}, Proposition~\ref{prop:contraction}, and Proposition~\ref{prop:small-set} prove that the hypotheses of their result hold.

  We now prove the bounds on $d$: inequality~\eqref{eq:Lipschitz-domination} follows immediately from the fact that $\|(u, x, y)\|_{H^5,2} \geq |(x, y)|$. For inequality~\eqref{eq:d-L2}, we note first that \[ d(\delta_{u, 0, e_1}, \delta_{u, x, y}) \leq C\exp(V(u))\log|x-y|^{-1}. \] Indeed, this follows from the fact that $d((u, x, y), (u, x', y')) \leq C\exp(V(u))$ for any $x', y' \in \xspace$ with $\frac{|x-y|}{|x'-y'|} \leq 2$. A similar argument (along with the fact that $\E_{\mu^b}[\exp(V(u))] \leq C$) shows that $d(\mu^b \times \delta_{0, e_1}, \mu^b \times \Leb) \leq C$. Finally, the same exponential moment bound on $V(u)$ shows that $d(\mu^b \times \delta_{0, e_1}, \delta_{u, 0, e_1}) \leq C\exp(V(u))$. By the triangle inequality, we thus obtain
  \begin{equation*}
    d(\delta_{u, x, y}, \mu^b \times \Leb) \leq C\exp(V(u))\log|x-y|^{-1}.
  \end{equation*}
  Integrating over $(x, y) \in M^2$ completes the proof.
\end{proof}
\begin{proof}[Proof of Theorem~\ref{thm:main-scalar-mixing}]
  A direct consequence of Proposition~\ref{prop:2pt-implies-scalar} and Theorem~\ref{thm:2pt-mixing}, together with interpolating for any $p>1$, 
  \[\|\varphi_t\|_{H^{-1}} \leq \|\varphi_t\|_{H^{-1/p}} \leq \|\varphi_t\|_{H^{-1}}^{1/p} \|\varphi_t\|_{L^2}^{1-1/p} = \|\varphi_t\|_{H^{-1}}^{1/p} \|\varphi_0\|_{L^2}^{1-1/p}.\qedhere\] 
\end{proof}

\section{Malliavin calculus and a smoothing estimate}\label{sec:malliavin}

We recommend that the reader of this section be familiar with the hypoellipticity theory of~\cite{hairer_ergodicity_2006,hairer_theory_2011} prior to reading this section. The goal of this section is to adapt the techniques of~\cite{hairer_theory_2011} to our specific setting. Throughout this section, we consider an arbitrary manifold $M$ with vector field map $\Theta$ as in Definition~\ref{defn:arbitrary-process}. We recall that $\Theta$ and $\omega_t$ induces a coupled process $(\omega_t,p_t)$ where $p_t \in M$. A central idea introduced in~\cite{hairer_ergodicity_2006} is to work differentially. As such, we will be primarily concerned with the linearization of the $(\omega_t, p_t)$ process.

\begin{definition}
  \label{defn:linearization}
  We denote the derivative of the process $(\omega_t,p_t)$---viewed as a (random) function of its previous value $(\omega_s, p_s)$---by the (random) linear operator $J_{s,t} \colon H^n \times T_{p_s} M \to H^n \times T_{p_t} M$. We note that $J_{s,t}$ solves an equation in $t$, that is
  \[\begin{cases}
    \frac{\dif}{\dif t} J_{s,t} = L_t J_{s,t}\\
    J_{s,s} = \id,
  \end{cases}\]
where $L_t := L(\omega_t, p_t) \colon H^6(\T^2) \times T_{p_t} M \to H^4(\T^2) \times T_{p_t} M$ is given by
\[L_t \begin{pmatrix} \varphi \\ q \end{pmatrix} := \begin{pmatrix}
  \nu\Delta \varphi_t - \nabla^\perp \Delta^{-1} \varphi \cdot \nabla \omega_t - \nabla^\perp \Delta^{-1} \omega_t \cdot \nabla \varphi\\
  q_t \cdot \nabla \Theta_{\omega_t}(p_t) + \Theta_{\varphi}(p_t)
\end{pmatrix}.\]

We denote the second derivative of the same process $(\omega_t,p_t)$---again viewed a function of $(\omega_s,p_s)$---by the (random) bilinear form $J^2_{s,t} \colon (H^n \times T_{p_s})\otimes (H^n \times T_{p_s}) \to H^n \times T_{p_t}$. The precise value of $n$ for which these are defined is determined by Assumption~\ref{asmp:dynamic-bounds} below.
\end{definition}

We now introduce an arbitrary order version of the super-Lyapunov function $V$ given in Definition~\ref{def:V}. We ultimately will only need to apply the result with $n=4$, in which case $V^n$ is essentially the same as $V$, but we consider arbitrary order at no real increase in complexity and for a slight increase in generality.

\begin{definition}
  Let
  \[V^n(\omega) :=  \sigma_n\Big( \|\omega\|_{L^2}^2 +\|\omega\|_{H^n}^{\frac{2}{n+2}} \Big),\]
  where $\sigma_n>0$ is chosen so that we can apply Proposition~\ref{prop:omega-bounds} to
  \[e^{2\sigma_n \eta \|\omega\|_{L^2}^2} \text{ and } e^{2 \sigma_n \eta \|\omega\|_{H^n}^{\frac{2}{n+2}}}\]
  for all $\eta \leq 1$.
\end{definition}

We note the following simple relation between the $V^n$, saying the higher order versions control the lower order ones.

\begin{lemma}\label{lem:Vn-poincare}
  For any $1 \leq k \leq n$, there exists $C(k,n)$ such that
  \[ V^k \leq C V^n.\]
\end{lemma}

\begin{proof}
  Expanding the definition,
  \begin{align*}
    V^k(\omega) = \sigma_k \Big(\|\omega\|_{L^2}^2 + \|\omega\|_{H^k}^{\frac{2}{k+1}}\Big) &\leq C \Big(\|\omega\|_{L^2}^2 +  \|\omega\|_{L^2}^{\frac{2(n-k)}{n(k+1)}} \|\omega\|_{H^n}^{\frac{2k}{n(k+1)}}\Big)
    \\&\leq  C \Big(\|\omega\|_{L^2}^2 +  \|\omega\|_{H^n}^{\frac{2}{n+1}}\Big) \leq C V^n(\omega),
  \end{align*}
  where we interpolate and use Young's inequality.
\end{proof}

We now give the essential bounds on the dynamics we need to prove our hypoellipticity result; compare to~\cite[Assumption D.1]{hairer_theory_2011}. We note that below we use the notation $\|T\|_{H^n \to H^k}$ for the operator norm of $T\colon H^n \times TM \to H^k \times TM$, with the $TM$ component normed by the Riemannian metric.

\begin{assumption}\label{asmp:dynamic-bounds}
  There exists $n \in \N$ such that for all $T,q >0, \eta \in (0,1)$ there is $C(p_0,\eta,q,T)>0$, locally bounded in $p_0 \in M$, such that
  \begin{align}
    \E \sup_{0 \leq t \leq T} \sup_{\varphi \in H^4(\T^2), \|\varphi\|\leq 1} \|\Theta_{\varphi}(p_t)\|^q &\leq  C e^{\eta V^n(\omega_0)}
                                                                      \label{eq:theta-moment-bound-asmp}\\
    \E \sup_{0 \leq s \leq t \leq T} \|J_{s,t}\|_{H^n \to H^n}^q &\leq C e^{\eta V^n(\omega_0)}
                                                      \label{eq:J-moment-bound-asmp}\\
                                                     \E \|J_{T/2,3T/4}\|^q_{L^2 \to H^{n+1}} &\leq C e^{\eta V^n(\omega_0)}\label{eq:J-smoothing-bound-asmp}\\
    \E \sup_{0 \leq s \leq t \leq T} \|J^2_{s,t}\|_{H^n \otimes H^n \to H^n}^q &\leq C e^{\eta V^n(\omega_0)}
                                                              \label{eq:J2-moment-bound-asmp}\\ \E \sup_{T/2 \leq t \leq T} \|L_t\|_{H^{n+2} \to H^n}^q &\leq C e^{\eta V^n(\omega_0)} \label{eq:L-moment-bound-asmp}.
  \end{align}
\end{assumption}

We also need a spanning condition on the vector fields. In particular, we would like that at every point $p \in M$, $\{\Theta_{e_k}(p) : |k| \leq R\}$ spans $T_p M$ for some $R$. What we need is a bit stronger though, since we want a uniform quantitative bound on the spanning: both $\{(1,0),(0,1)\}$ and $\{(1,0), (1,10^{-10})\}$ span $\R^2$ but the former is much less degenerate than the latter. In the compact case, a uniform bound with some constant follows directly from a qualitative bound and continuity. We consider some noncompact manifolds as well (such as $M^2, M^T, M^J$), for which we need the following moment bound. 

\begin{assumption}
  \label{asmp:nondegen}
  There exists $n \in \N$ and $R>0$ such that for all $T,q>0, \eta\in (0,1)$, there exists $C(p_0,\eta,q,T)>0,$ locally bounded in $p_0$, such that
  \[\E \left[\sup_{v \in T_{p_T}M, \|v\|=1} \Big(\max_{|k| \leq R} \big|g\big(\Theta_{e_k}(p_T),v\big)\big|\Big)^{-q} \right] \leq C e^{\eta V^n(\omega_0)}.\]
\end{assumption}

For the case that $M$ is compact, we get a much simpler bound, that 
\[ \sup_{v \in T_{p_T}M, \|v\|=1} \Big(\max_{|k| \leq R} \big|g\big(\Theta_{e_k}(p_T),v\big)\big|\Big)^{-1} \leq C.\]
One should keep this bound in mind as the primary example. The moment bound above is just a mild extension to the case where we do not have an a.s.\ bound.

The straightforward but lengthy computational proof of the following proposition giving that the relevant processes satisfy our assumptions is deferred to the Subsections~\ref{ssa:linearization-bounds} and~\ref{ssa:nondegeneracy}.
\begin{proposition}
  \label{prop:manifold-processes-good}
  The one-point, two-point, projective, tangent, and Jacobian processes each satisfy Assumptions~\ref{asmp:dynamic-bounds} and~\ref{asmp:nondegen}.
\end{proposition}

For the remainder of this section, we fix $n$ and suppose that Assumptions~\ref{asmp:dynamic-bounds} and~\ref{asmp:nondegen} hold for this $n$. We then use a bracket to denote the inner product with respect to $H^n \times T_p M$, that is
\[\Big\<\begin{pmatrix}\varphi \\p \end{pmatrix}, \begin{pmatrix} r \\ q\end{pmatrix}\Big\> :=  g(r,q)+\<\varphi,\psi\>_{H^n}  = g(r,q) + \int_{\T^2} \nabla^n \varphi : \nabla^n \psi\,\dif x.\]
All adjoints are taken with respect to this inner product, so if $A \colon H^n \times TM \to H^n \times TM,$ $A^* \colon H^n \times TM \to H^n \times TM$, with
\[\<\mathfrak{q}, A\ap\> = \<A^*\mathfrak{q}, \ap\>,\]
where we will throughout use $\mathfrak{p}$ to denote an element $H^n \times TM$. We also use (unsubscripted) norms to refer to the norm on this space,
\[\|\ap\| := \|\ap\|_{H^n} = \sqrt{\<\ap,\ap\>}.\]

\subsection{Malliavin matrix nondegeneracy}

The brief and clear motivation for considering the following objects is provided in~\cite[Section 4.2]{hairer_ergodicity_2006}.

\begin{definition}
\label{defn:A-def}
  We define for $0 \leq s \leq t,$ the operator $A_{s,t} \colon L^2([s,t], \R^F) \to H^n(\T^2) \times T_{p_t}M$
  \[A_{s,t} v :=  \sum_{k \in F}\int_s^t c_k v^k_r J_{r,t} e_k \,\dif r.\]
  We define for $0 \leq t$ the quadratic form $H^n(\T^2) \times T_{p_t} M \to [0,\infty)$ given by
  \[\<\ap, \M_t \ap\> := \<A^*_{0,t} \ap, A^*_{0,t} \ap\> = \sum_{k \in F} c_k^2 \int_0^t |\<\ap,J_{r,t} e_k\>|^2\,\dif r = \sum_{k \in F} c_k^2 \int_0^t |\<e_k,J_{r,t}^*\ap\>|^2\, \dif r.\]
\end{definition}

Our goal for this subsection is to prove Theorem~\ref{thm:malliavin-nondegenerate} under Assumptions~\ref{asmp:dynamic-bounds} and~\ref{asmp:nondegen}, which essentially says that the Malliavin matrix is quantitatively nondegenerate on finite dimensional subspaces. It is the essential ingredient to the proof of Proposition~\ref{prop:size-of-error}, which itself comprises half of the proof of Theorem~\ref{thm:smoothing}. Comparing Theorem~\ref{thm:malliavin-nondegenerate} with \cite[Theorem 6.7]{hairer_asymptotic_2011}, we note the dramatically worse stochastic integrability, with an arbitrary rate $g(\ep)$ in place of the rate $\ep^q$ for any $q>0$. The proof below makes the origin of this difference clear; for a heuristic explanation, see the discussion of Subsection~\ref{ss:proof-of-smoothing}.

\begin{theorem}
\label{thm:malliavin-nondegenerate}
  Let $\Pi_R \colon H^n(\T^2) \times TM \to H^n(\T^2) \times TM$ be the orthogonal projection such that $\Pi_R|_{TM} = \id$, the identity map, and $\Pi_R|_{H^n}$ is the orthogonal projection onto the span of $(e_k)_{|k| \leq R}$. Then for any $T>0$, any $\eta,\delta\in (0,1)$, and any $p_0\in M$ there exists $g\colon (0,1) \to [0,1]$ such that $\lim_{\ep \to 0} g(\ep) = 0$ locally uniformly in $p_0$, and
  \begin{equation}
    \label{eq:almost-invertible-malliavin}
    \P\Big(\inf_{\substack{\|\ap\| = 1\\ \|\Pi_R \ap\| \geq \delta}} \<\ap,\M_T \ap\> <\ep\Big) \leq g(\ep) e^{\eta V^n(\omega_0)}.
  \end{equation}
\end{theorem}

This convention is inspired by the similar convention in~\cite[Section 6]{hairer_asymptotic_2011}. 

\begin{definition}
  We say a parameterized family of events $H_\ep(\omega_0,p_0) \subseteq \Omega$ is $V^n$-controlled negligible if for all $q>0, \eta\in (0,1)$, there exists $C(p_0,q,\eta)>0$, locally bounded in $p_0$, such that
  \[\P(H_\ep(\omega_0,p_0)) \leq C \ep^q e^{\eta V^n(\omega_0)}.\]

  We say an implication parameterized by $\ep, \omega_0,p_0$
  \[P_\ep(\omega_0,p_0) \implies Q_\ep(\omega_0,p_0)\]
  holds modulo a $V^n$-controlled negligible family of events if there exists a $V^n$-controlled negligible family of events also parameterized by $\ep, \omega_0,p_0$ outside of which the implication holds.
\end{definition}

The next lemma follows directly from the Chebyshev inequality.

\begin{lemma}
  \label{lem:moment-negligible}
  Suppose some parameterized family of random variables $X_{\ep,\omega_0,p_0}$ has the bound that for all $q>0, \eta \in (0,1)$ there exists $C(p_0,q,\eta)>0$, locally bounded in $p_0,$ such that
  \[\E |X|^q \leq C e^{\eta V^n(\omega_0)}.\]
  Then for any $\gamma>0$, the family of events
  \[\{|X| \geq \ep^{-\gamma}\}\]
  is $V^n$-controlled negligible.
\end{lemma}

We will essentially use the following interpolation lemma. The key idea behind the inductive argument below is that we are given an object which is small, and we then show its derivative is small using this lemma and the fact that the derivative is itself not \textit{too} rough.
\begin{lemma}[{\cite[Lemma 6.14]{hairer_theory_2011}}]
  \label{lem:holder-interpolation}
  Let $f \colon [0,T] \to \R$ with $f \in C^{1,\alpha}$. Then
  \[\|f'\|_{C^0} \leq \frac{8}{T}\|f\|_{C^0} + 8\|f\|_{C^0}^{\frac{\alpha}{1+\alpha}} \|f'\|_{C^\alpha}^{\frac{1}{1+\alpha}}.\]
\end{lemma}

We will also often use the standard fact about solution operator semigroups,
\[\frac{\dif}{\dif s} J^*_{s,T} = L^*_s J^*_{s,T}.\]
The following is the base step in an inductive argument analogous to~\cite[Lemma 6.17]{hairer_asymptotic_2011}.
\begin{proposition}
    \label{prop:malliavin-induction-base}
    The implication
    \[\<\ap,\M_T\ap\> \leq \ep \|\ap\|^2 \implies \max_{k \in F} \sup_{t \in [T/2,T]} |\<e_k,J^*_{t,T} \ap\>| \leq C\ep^{1/8} \|\ap\|,\]
    holds modulo a $V^n$-controlled negligible family of events.
\end{proposition}

\begin{proof}
     Suppose that $\<\ap,\M_T\ap\> \leq \ep \|\ap\|^2$. Then
    \[\max_{k \in F} \int_0^T |\<e_k, J_{t,T}^*\ap\>|^2\,\dif t \leq C \sum_{k \in F} c_k^2\int_0^T  |\<e_k, J_{t,T}^*\ap\>|^2\,\dif t = \<\ap, \M_T \ap\> \leq \ep \|\ap\|^2.\]
    Fix $k \in F$ and let
    \[f(t) := \int_{T/2}^t  \<e_k, J_{s,T}^*\ap\>\,ds,\]
    so that $\|f\|_{C^0([T/2,T])} \leq C\ep^{1/2} \|\ap\|$.
    Then using Lemma~\ref{lem:holder-interpolation},
    \begin{align*}
    \sup_{t \in [T/2,T]}  |\<e_k, J_{s,T}^*\ap\>| = \|f'\|_{C^0([T/2,T])}
    &\leq  C\|f\|_{C^0([T/2,T])} + C\|f\|_{C^0([T/2,T])}^{1/2} \|f''\|_{C^0([T/2,T])}^{1/2}
    \\&\leq  C\ep^{1/2} \|\ap\| + C\ep^{1/4} \|\ap\|^{1/2}\|f''\|_{C^0([T/2,T])}^{1/2} \leq C\ep^{1/8} \|\ap\|,
    \end{align*}
    unless
    \[\|f''\|_{C^0([T/2,T])} \geq \ep^{-1/4} \|\ap\|.\]
    Thus, to conclude, it suffices to see that the events $\|f''\|_{C^0([T/2,T])} \geq \ep^{-1/4}\|\ap\|$ are $V^n$-controlled negligible (with constants independent of $\ap$). To this end, we compute for $t \in [T/2,T],$
    \[
        |f''(t)|= \Big|\frac{\dif}{\dif t} \<e_k,J^*_{t,T} \ap\>\Big|
        = \big|\<L_t e_k, J^*_{t,T} \ap\>\big|
        \leq \|L_t e_k\| \|J^*_{t,T} \ap\|
        \leq \frac{1}{2}\big(C\|L_t\|_{H^{n+2} \to H^n}^2+ \|J_{t,T}\|_{H^n \to H^n}^2\big) \|\ap\|.
    \]
    Thus $\|f''\|_{C^0([T/2,T])} \geq \ep^{-1/4} \|\ap\|$ implies one of the $\ap$-independent events
    \[\sup_{t \in [T/2,T]}\|L_t\|_{H^{n+2} \to H^n} \geq C^{-1}\ep^{-1/8} \text{ or } \sup_{t \in [T/2,T]}\|J_{t,T}\|_{H^n \to H^n} \geq \ep^{-1/8}.\]
    These are then both $V^n$-controlled negligible by Assumption~\ref{asmp:dynamic-bounds} and Lemma~\ref{lem:moment-negligible}, allowing us to conclude.
\end{proof}

The following is a nonadapted version of Norris' lemma~\cite{norris_simplified_1986}. A preliminary version of such a result appeared in~\cite{mattingly_malliavin_2006}, though we will use the much more user-friendly statement below. This result is essential for the control of the Malliavin matrix, allowing us to separate out the effects of the different Brownian motions.

\begin{theorem}[{\cite[Theorem 7.1]{hairer_theory_2011}}]
    \label{thm:norris-lemma}
    There exists a negligible family of events\footnote{We say a family of events $H_\ep$ is negligible if the family $\tilde H_\ep(\omega_0,p_0):= H_\ep$ is $V^n$-controlled negligible.}, parameterized by $\ep$, outside of which for any set of processes $A^0_t, (A^k_t)_{k \in F}$ on the interval $[T/2,T]$, the implication
    \[\sup_{t \in [T/2,T]} \Big|A^0_t + \sum_{k \in F} A^k_t W^k_t \Big| \leq \ep \text{ and } \max_{k \in F \cup \{0\}}  \sup_{t \in [T/2,T]} \Big|\frac{\dif}{\dif t} A^k_t\Big| \leq \ep^{-1/9} \implies  \max_{k \in F \cup \{0\}}  \sup_{t \in [T/2,T]} |A_k^t| \leq \ep^{1/3}\]
    holds.
\end{theorem}

Following the scheme of~\cite[Section 6]{hairer_theory_2011}, we make the following definition:
\begin{equation}
\label{eq:v-def}
    v_t := \omega_t -\sum_{k \in F} c_k W^k_t.
\end{equation}
The process $v_t$ then has finite variation in time. This allows us to separate out the contributions of the Brownian motions and apply Theorem~\ref{thm:norris-lemma} in the following argument. The next proposition is the inductive step of the main argument of this subsection.

\begin{proposition}
    \label{prop:malliavin-inductive-step}
    For any $k \in \Z^2 \setminus \{0\}$ and any $\gamma > 0,$ the implications
    \begin{align*}\sup_{t \in [T/2,T]} |\<e_k, J^*_{t,T} \ap\>|  &\leq \ep^\gamma \|\ap\|\\
      \implies &\max_{j\in F} \sup_{t \in [T/2,T]} |\< \nabla^\perp \Delta^{-1} e_j\cdot \nabla e_k + \nabla^{\perp} \Delta^{-1} e_k \cdot \nabla e_j, J^*_{t,T} \ap\>| \leq \ep^{\gamma/24} \|\ap\|\\
    \sup_{t \in [T/2,T]} |\<e_k, J^*_{t,T} \ap\>|  &\leq \ep^\gamma \|\ap\|
    \\\implies  &\sup_{t \in [T/2,T]} |\<\Theta_{e_k}(p_t) + \Delta e_k -  \nabla^\perp \Delta^{-1} v_t \cdot \nabla e_k - \nabla^{\perp} \Delta^{-1} e_k \cdot \nabla v_t, J_{t,T}^* \ap\>| \leq \ep^{\gamma/24} \|\ap\|
    \end{align*}
    hold modulo a $V^n$-controlled negligible family of events, where $v_t$ is as in~\eqref{eq:v-def}.
\end{proposition}

\begin{proof}
    Let
    \[f(t) := \<e_k, J^*_{t,T} \ap\>,\]
    so by assumption
    \begin{equation}
        \label{eq:malliavin-inductive-asmpt}
        \|f\|_{C^0([T/2,T])} \leq \ep^\gamma \|\ap\|.
    \end{equation}
    Then we compute
    \begin{align*}
        f'(t) &= \<L_t e_k, J^*_{t,T} \ap\>
        \\&= \<\Theta_{e_k}(p_t)+\Delta e_k -  \nabla^\perp \Delta^{-1} \omega_t \cdot \nabla e_k - \nabla^{\perp} \Delta^{-1} e_k \cdot \nabla \omega_t, J^*_{t,T} \ap\>
        \\&= \<\Theta_{e_k}(p_t) + \Delta e_k -  \nabla^\perp \Delta^{-1} v_t \cdot \nabla e_k - \nabla^{\perp} \Delta^{-1} e_k \cdot \nabla v_t, J^*_{t,T} \ap\>
        \\&\qquad  -\sum_{j \in F} c_j W^j_t \< \nabla^\perp \Delta^{-1} e_j\cdot \nabla e_k + \nabla^{\perp} \Delta^{-1} e_k \cdot \nabla e_j, J^*_{t,T} \ap\>,
    \end{align*}
    where we use the definition of $v_t$~\eqref{eq:v-def}. Our goal is to apply Theorem~\ref{thm:norris-lemma} to $f'(t)$ in order to conclude, so we need to begin verifying the assumptions of the theorem. Our first task is to bound $\|f'\|_{C^0([T/2,T])}$. Using Lemma~\ref{lem:holder-interpolation} and~\eqref{eq:malliavin-inductive-asmpt},
    \begin{equation}
    \label{eq:malliavin-inductive-f-prime-bound}
        \|f'\|_{C^0([T/2,T])} \leq C \|f\|_{C^0([T/2,T])} + C \|f\|_{C^0([T/2,T])}^{\frac{1}{4}} \|f'\|_{C^{1/3}([T/2,T])}^{\frac{3}{4}}
        \leq  \ep^{\gamma/8} \|\ap\|,
    \end{equation}
    unless
    \[\|f'\|_{C^{1/3}([T/2,T])} \geq C^{-1} \ep^{-\gamma/6} \|\ap\|,\]
    which we now check is $V^n$-controlled negligible. Then for $T/2 \leq t \leq s \leq T,$
    \begin{align}
      &\|f'\|_{C^{1/3}([T/2,T])} \notag\\
      &\quad \leq C\Big\|\frac{\dif}{\dif t}\<\Theta_{e_k}(p_t)+\Delta e_k -  \nabla^\perp \Delta^{-1} v_t \cdot \nabla e_k - \nabla^{\perp} \Delta^{-1} e_k \cdot \nabla v_t, J^*_{t,T} \ap\>\Big\|_{C^0([T/2,T])}\label{eq:big-f'-1}
        \\&\qquad  +C \sum_{j \in F} \|W^j_t\|_{C^{1/3}([T/2,T])} \|\< \nabla^\perp \Delta^{-1} e_j\cdot \nabla e_k + \nabla^{\perp} \Delta^{-1} e_k \cdot \nabla e_j, J^*_{t,T} \ap\>\|_{C^0([T/2,T])} \label{eq:big-f'-2}
        \\&\qquad+C \sum_{j \in F} \|W^j_t\|_{C^0([T/2,T])} \Big\|\frac{\dif}{\dif t}\< \nabla^\perp \Delta^{-1} e_j\cdot \nabla e_k + \nabla^{\perp} \Delta^{-1} e_k \cdot \nabla e_j, J^*_{t,T} \ap\>\Big\|_{C^0([T/2,T])} \label{eq:big-f'-3}.
        \end{align}
We now control each term separately. For~\eqref{eq:big-f'-1}, it is controlled by
\begin{align*}
    &C\Big\|\< \Theta_{\omega_t}(p_t) \cdot \nabla \Theta_{e_k}(p_t)- \nabla^\perp \Delta^{-1} \frac{\dif}{\dif t} v_t \cdot \nabla e_k - \nabla^{\perp} \Delta^{-1} e_k \cdot \nabla \frac{\dif}{\dif t} v_t, J^*_{t,T} \ap\>\Big\|_{C^0([T/2,T])}
        \\&\qquad + C \big\|\<L_t\big(\Theta_{e_k}(p_t) +\Delta e_k -  \nabla^\perp \Delta^{-1} v_t \cdot \nabla e_k - \nabla^{\perp} \Delta^{-1} e_k \cdot \nabla v_t\big), J^*_{t,T} \ap\>\big\|_{C^0([T/2,T])}
        \\&\leq C \sup_{t \in [T/2,T]} \|\Theta_{\omega_t}(p_t)\|\sup_{t \in [T/2,T]}\|J_{t,T}\|_{H^n \to H^n} \|\ap\|
        \\&\qquad + C \sup_{t \in [T/2,T]} \Big\|\frac{\dif}{\dif t} v_t\Big\|_{H^{n+1}} \sup_{t \in [T/2,T]}\|J_{t,T}\|_{H^n \to H^n} \|\ap\|
        \\&\qquad + C\sup_{t \in [T/2,T]}\|L_t\|_{H^{n+2} \to H^n} \Big(1  +  \sup_{t \in [T/2,T]}\|v_t\|_{H^{n+3}} \Big) \sup_{t \in [T/2,T]}\|J_{t,T}\|_{H^n \to H^n} \|\ap\|
\end{align*}
While~\eqref{eq:big-f'-2} and~\eqref{eq:big-f'-3} are bounded by
        
        \begin{align*}
        &C \sum_{j \in F} \|W^j_t\|_{C^{1/3}([T/2,T])}  \sup_{t \in [T/2,T]}\| J_{t,T}\|_{H^n \to H^n} \|\ap\|
        \\&\qquad+C \sum_{j \in F} \|W^j_t\|_{C^0([T/2,T])} \big\|\<L_t \big( \nabla^\perp \Delta^{-1} e_j\cdot \nabla e_k + \nabla^{\perp} \Delta^{-1} e_k \cdot \nabla e_j\big), J^*_{t,T} \ap\>\big\|_{C^0([T/2,T])}
        \\&\quad \leq C \sum_{j \in F} \|W^j_t\|_{C^{1/3}([T/2,T])}  \sup_{t \in [T/2,T]}\| J_{t,T}\|_{H^n \to H^n} \|\ap\|
        \\&\qquad+C \sum_{j \in F} \|W^j_t\|_{C^0([T/2,T])} \sup_{t \in [T/2,T]}\|L_t\|_{H^{n+2} \to H^n} \sup_{t \in [T/2,T]} \|J_{t,T}\|_{H^n \to H^n}\| \ap\|.
    \end{align*}
    Thus to see~\eqref{eq:malliavin-inductive-f-prime-bound}, it suffices to verify the following events are $V^n$-controlled negligible:
    \begin{align}
        \sup_{t \in [T/2,T]} \|\Theta_{\omega_t}(p_t)\| &\geq C^{-1} \ep^{-\gamma/18}, \label{eq:V-negligible}
        \\
        \sup_{t \in [T/2,T]}\|J_{t,T}\|_{H^n \to H^n} &\geq C^{-1} \ep^{-\gamma/18}, \label{eq:J-negligible}
        \\\sup_{t \in [T/2,T]}\|L_t\|_{H^{n+2} \to H^n}&\geq C^{-1} \ep^{-\gamma/18}, \label{eq:L-t-H2-negligible}
        \\\|W^j_t\|_{C^{1/3}([T/2,T])}  +  \|W^j_t\|_{C^0([T/2,T])} &\geq C^{-1} \ep^{-\gamma/18}. \label{eq:brownian-negligible}
        \\ \sup_{t \in [T/2,T]} \Big\|\frac{\dif}{\dif t} v_t\Big\|_{H^{n+1}} &\geq C^{-1} \ep^{-\gamma/18}, \label{eq:d-dt-v-negligible}
        \\ \sup_{t \in [T/2,T]}\|v_t\|_{H^{n+3}}&\geq C^{-1} \ep^{-\gamma/18}. \label{eq:v-H3-negligible}
    \end{align}
    For~\eqref{eq:V-negligible}, we have that
    \begin{align*}\E\sup_{t \in [T/2,T]} \|\Theta_{\omega_t}(p_t)\|^q &\leq \E \sup_{t \in [T/2,T]} \|\omega_t\|_{H^4}^q \sup_{t \in [T/2,T]} \sup_{\|\varphi\|_{H^4} = 1} \|\Theta_{\varphi}(p_t)\|^q
    \\&\leq \E \sup_{t \in [T/2,T]} \|\omega_t\|_{H^4}^{2q} +  \E\sup_{t \in [T/2,T]} \sup_{\|\varphi\|_{H^4} = 1} \|\Theta_{\varphi}(p_t)\|^{2q} \leq C e^{\eta V^n(\omega_0)},
    \end{align*}
    where the last inequality is from Assumption~\ref{asmp:dynamic-bounds} and~\eqref{eq:Hn-bound}. Thus, we can conclude~\eqref{eq:V-negligible} by Lemma~\ref{lem:moment-negligible}.

    We note that \eqref{eq:J-negligible} and~\eqref{eq:L-t-H2-negligible} follow directly from Assumption~\ref{asmp:dynamic-bounds} and Lemma~\ref{lem:moment-negligible}. Also,~\eqref{eq:brownian-negligible} follows from Chebyshev and standard Brownian motion moment estimates. Thus, we need only show~\eqref{eq:d-dt-v-negligible} and~\eqref{eq:v-H3-negligible}. For~\eqref{eq:d-dt-v-negligible}, for any $t \in [T/2,T]$
    \begin{align*}
        \Big\|\frac{\dif}{\dif t} v_t\Big\|_{H^{n+1}} &\leq \|\Delta \omega_t\|_{H^{n+1}} + \|\nabla^{\perp} \Delta^{-1} \omega_t \cdot \nabla \omega_t\|_{H^{n+1}}
        \\&\leq \|\omega_t\|_{H^{n+3}} + C \|\omega_t\|_{W^{n+2,4}}^2
        \\&\leq C ( 1+ \|\omega_t\|_{H^{n+3}}^2).
    \end{align*}
    Also for~\eqref{eq:v-H3-negligible},
    \[\|v_t\|_{H^{n+3}} \leq \|\omega_t\|_{H^{n+3}} + C \sum_{k \in F} |W^k_t|.\]
    Thus since $\sup_{t \in [T/2,T]} \max_{k \in F} |W^k_t| \geq C^{-1} \ep^{-\gamma/18}$ is negligible, we see that for~\eqref{eq:d-dt-v-negligible} and~\eqref{eq:v-H3-negligible}, it suffices to see that
    \[\sup_{t \in [T/2,T]} \|\omega_t\|_{H^{n+3}} \geq C^{-1} \ep^{-\gamma/36}\]
    is $V^n$-controlled negligible. This is then direct from Lemma~\ref{lem:moment-negligible} and Proposition~\ref{prop:omega-bounds}.

    We have thus established~\eqref{eq:malliavin-inductive-f-prime-bound}; we now need to verify the other assumption of Theorem~\ref{thm:norris-lemma}. We note though that directly from the above computations, we have that
    \[\sup_{t \in [T/2,T]} \Big|\frac{\dif}{\dif t}\<\Theta_{e_k}(p_t) +\Delta e_k -  \nabla^\perp \Delta^{-1} v_t \cdot \nabla e_k - \nabla^{\perp} \Delta^{-1} e_k \cdot \nabla v_t, J^*_{t,T} \ap\>\Big| \leq \ep^{-\gamma/72}\]
    and
    \[\sup_{t \in [T/2,T]} \Big|\frac{\dif}{\dif t} \< \nabla^\perp \Delta^{-1} e_j\cdot \nabla e_k + \nabla^{\perp} \Delta^{-1} e_k \cdot \nabla e_j, J^*_{t,T} \ap\>\Big| \leq \ep^{-\gamma/72}.\]
    modulo a $V^n$-controlled negligible family of events. Thus, by Theorem~\ref{thm:norris-lemma} applies to $f'(t)$ and we conclude.
\end{proof}

\begin{equation}
    \label{eq:B-kj-def}
B_{k,j}:= \nabla^\perp \Delta^{-1} e_j\cdot \nabla e_k + \nabla^{\perp} \Delta^{-1} e_k \cdot \nabla e_j.
\end{equation}

The next proposition can be verified by direct computation, using from Assumption~\ref{asmp:modes} that $F = -F$.
\begin{proposition}
    \label{prop:B-kj-rep}
    Let $B_{k,j}$ be defined by~\eqref{eq:B-kj-def}. Suppose $A \subseteq \Z^2 \setminus \{0\}$ such that $A = - A$. Then
    \[\linspan\Big(\{e_k : k \in A\} \cup \{B_{k,j} : k \in A, j\in F\}\Big) = \linspan\Big(\{e_{k+j} : k \in A, j \in F, |k| \ne |j|, k \cdot j^\perp \ne 0\}\Big).\]
\end{proposition}

\begin{definition}
    Define $A_0 := F$ and
    \[A_{j+1} = A_j \cup \{k+j : k \in A, j \in F, |k| \ne |j|, k \cdot j^\perp \ne 0\}.\]
\end{definition}

The following is a direct consequence of~\cite[Corollary 4.5]{hairer_ergodicity_2006}, using Assumption~\ref{asmp:modes}.

\begin{proposition}
    \label{prop:A-j-span}
    \[\bigcup_{j=0}^\infty A_j = \Z^2 \setminus \{0\}.\]
\end{proposition}

Inducting with Propositions~\ref{prop:malliavin-induction-base} and Proposition~\ref{prop:malliavin-inductive-step} and then using Proposition~\ref{prop:B-kj-rep} and Proposition~\ref{prop:A-j-span}, we get the following.
\begin{proposition}
    \label{prop:malliavin-inducted}
    For any $R>1$, there exists $\gamma_R >0$ such that the implications
    \begin{gather*}\<\ap,\M_T\ap\> \leq \ep \|\ap\|^2  \implies \max_{|k| \leq R} \sup_{t \in [T/2,T]} |\< e_k, J^*_{t,T} \ap\>| \leq \ep^{\gamma_R} \|\ap\|
    \\\<\ap,\M_T\ap\> \leq \ep \|\ap\|^2  \implies \max_{|k| \leq R} |\<\Theta_{e_k}(p_T) + \Delta e_k -  \nabla^\perp \Delta^{-1} v_T \cdot \nabla e_k - \nabla^{\perp} \Delta^{-1} e_k \cdot \nabla v_T, \ap\>| \leq \ep^{\gamma_R} \|\ap\|
    \end{gather*}
    hold modulo a $V^n$ controlled negligible family of events, where $v_T$ is as in~\eqref{eq:v-def}.
\end{proposition}

If we did not have to consider the manifold directions, the first display of the above proposition would be sufficient to conclude Theorem~\ref{thm:malliavin-nondegenerate}. We will need to use the second display to control the manifold directions, as that is the only one in which we control overlaps of vectors which have nontrivial components in $TM$. We would like to isolate the $\Theta_{e_k}(p_T)$ part from the other contributions of the vector in the second display. Roughly, we should expect that the other terms, which all are in the $H^n(\T^2)$ direction, should be small as a consequence of the first display, which says that overlaps with vectors in the $H^n$ direction are small. The issue is that the first display only controls finitely many Fourier modes at a time, while $v_T$ is generically supported on infinitely many Fourier modes. It turns out that a soft argument together with the first display applied with infinitely many $R$ allow us to control the overlap of $\ap$ with all Fourier modes, which we now demonstrate.

\begin{lemma}\label{lem:soft-decay}
  Let $f \colon (0, 1) \times (0, 1) \to [0, 1]$ be monotone decreasing in the first variable and increasing in the second variable, with the property that $\lim_{y \to 0} f(x, y) = 0$ for every $x \in (0, 1)$. Then there is a continuous increasing function $r \colon (0, 1) \to [0, 1]$ such that $f(r(y), y) \leq r(y)$ for all $y \in (0, 1)$ and $\lim_{y \to 0} r(y) = 0$.
\end{lemma}
\begin{proof}
  Define $y_1 := 1$ and, for integer $n \geq 2$, choose $y_n \leq \frac{1}{2}y_{n-1}$ satisfying $f(n^{-1}, y_n) \leq n^{-1}$. Let
  \begin{equation*}
    r(y) :=
    \begin{cases}
      1 &\quad \text{if $y \geq y_2$}\\
      \theta n^{-1} + (1-\theta){(n+1)}^{-1} &\quad \text{if $y = \theta y_{n+1} + (1-\theta) y_{n+2}$ for $n \in \N$ and $\theta \in [0, 1]$.}
    \end{cases}
  \end{equation*}
  We observe that $r$ is continuous, defined on all of $(0, 1)$ because $y_n \to 0$ as $n \to \infty$, and $r(y) \to 0$ as $y \to 0$. For $y$ of the form $y = \theta y_{n+1} + (1-\theta) y_{n+2}$, we compute
  \[
    f(r(y), y) \leq f({(n+1)}^{-1}, y_{n+1}) \leq {(n+1)}^{-1} \leq r(y).\qedhere
  \]
\end{proof}

\begin{definition}
    We denote by $\pi_{H^n} \colon H^n \times TM \to H^n$ the projection onto the $H^n$ coordinate and $\pi_{TM} \colon H^n \times TM \to TM$ the projection onto the $TM$ coordinate.
\end{definition}

In the proposition below, we control infinitely many Fourier modes of $\ap$ at once. In particular, we show that with high probability a Sobolev norm of the $H^n$ component of $\ap$ is not too large. Note that the rate $g(\ep)$ can be arbitrarily poor, so the following bound is soft in $\varepsilon$.

\begin{proposition}
    \label{prop:H-n-2-malliavin}
    For any $q>0,$ $\eta \in (0,1),p_0 \in M$, there exists $C(p_0,q,\eta)>0$, locally bounded in $p_0$, and some $g\colon (0,1) \to [0,1]$ such that $\lim_{\ep \to 0} g(\ep) =0$ locally uniformly in $p_0$ and
    \[\P\Big( \exists \ap \in H^n \times T_{p_T}M, \|\ap\| =1 \land  \<\ap, \M_T\ap\> \leq \ep \land \|\pi_{H^n} \ap\|_{H^{n-2}} \geq g(\ep)\Big) \leq C \ep^q e^{\eta V^n(\omega_0)}.\]
\end{proposition}

\begin{proof}
    Fix $q>0, \eta \in (0,1)$.
    For $k \in \Z^2- \{0\},$ let
    \[\hat \ap_k := \<e_k, \ap\>.\]
    Then we have that
    \[\|\pi_{H^n} \ap\|_{H^{n-2}}^2 = \sum_{k \in \Z^2 \setminus \{0\}} |k|^{-4} |\hat \ap_k|^2 \leq R^{-4} + C\max_{|k| \leq R} |\hat \ap_k|^2,\]
    so
    \begin{align}
        &\P\Big( \exists \ap, \|\ap\| =1 \land  \<\ap, \M_T\ap\> \leq \ep \land \|\pi_{H^n} \ap\|_{H^{n-2}} \geq g(\ep)\Big)
        \notag\\&\qquad\leq \P\Big( \exists \ap, \|\ap\| =1 \land  \<\ap, \M_T\ap\> \leq \ep \land \max_{|k| \leq C(g(\ep))^{-1}} |\hat \ap_k|  \geq C^{-1} g(\ep)\Big).
        \label{eq:g-ep-comp-1}
    \end{align}
    From Proposition~\ref{prop:malliavin-inducted}, we have that for any $R>1$,
    \[ \P\Big( \exists \ap, \|\ap\| =1 \land  \<\ap, \M_T\ap\> \leq \ep \land \max_{|k| \leq R} |\hat \ap_k|  \geq \ep^{\gamma_R}\Big) \leq K_R \ep^{2q} e^{\eta V^n(\omega_0)},\]
    where $\gamma_R$ is without loss of generality decreasing in $R$ and $K_R$ increasing in $R$. Then let
    \[f(\delta, \ep) := \ep^{\gamma_{\delta^{-1}}},\]
    which is decreasing in the first variable, increasing in the second, and has for any fixed $\delta \in (0,1)$, $\lim_{\ep \to 0} f(\delta,\ep) = 0$. Thus, we can apply Lemma~\ref{lem:soft-decay}, to get an increasing $r(\ep)$ such that $\lim_{\ep \to 0} r(\ep) = 0$ and $f(r(\ep), \ep) \leq r(\ep)$. Then taking $g(\ep)$ so that $C^{-1} g(\ep) \geq r(\ep)$, we have
    \[\ep^{\gamma_{C (g(\ep))^{-1}}} = f(C^{-1} g(\ep), \ep) \leq f(r(\ep),\ep) \leq r(\ep) \leq C^{-1} g(\ep),\]
    and so
    \begin{align}
         &\P\Big( \exists \ap, \|\ap\| =1 \land  \<\ap, \M_T\ap\> \leq \ep \land \max_{|k| \leq C(g(\ep))^{-1}} |\hat \ap_k|  \geq C^{-1} g(\ep)\Big)
         \notag\\&\qquad\leq  \P\Big( \exists \ap, \|\ap\| =1 \land  \<\ap, \M_T\ap\> \leq \ep \land \max_{|k| \leq C(g(\ep))^{-1}} |\hat \ap_k|  \geq \ep^{\gamma_{C (g(\ep))^{-1}}}\Big)
         \notag\\&\qquad \leq K_{C (g(\ep))^{-1}} \ep^{2q} e^{\eta V^n(\omega_0)}.
         \label{eq:g-ep-comp-2}
    \end{align}
    Let $h(\ep)$ an increasing function such that $\lim_{\ep \to 0} h(\ep) = 0$ and such that
    \begin{equation}
    \label{eq:h-ep-prop}
        K_{h(\ep)^{-1}} \leq K_1 \ep^{-q}
    \end{equation}
    and then let
    \[g(\ep) = C\max(r(\ep), h(\ep)).\]
    Then by~\eqref{eq:g-ep-comp-1},~\eqref{eq:g-ep-comp-2}, and~\eqref{eq:h-ep-prop}, we have that
    \[\P\Big( \exists \ap, \|\ap\| =1 \land  \<\ap, \M_T\ap\> \leq \ep \land \|\pi_{H^n} \ap\|_{H^{n-2}} \geq g(\ep)\Big)  \leq C \ep^q e^{\eta V^n(\omega_0)},\]
    as desired.
\end{proof}

Below we leverage the above proposition, which gives control on all Fourier modes of $\ap$, together with the second display of Proposition~\ref{prop:malliavin-inducted} in order to control the overlaps of $\ap$ with the $TM$ directions. We note that the qualitative rate of Proposition~\ref{prop:H-n-2-malliavin} only allows for a qualitative rate and qualitative stochastic integrability below.

\begin{proposition}
\label{prop:tangent-space-overlap-malliavin}
    For any $\eta\in (0,1)$ and $p_0 \in M$, there exists $g\colon (0,1) \to [0,1]$ such that $\lim_{\ep \to 0} g(\ep) =0$ locally uniformly in $p_0$ and such that
    \[\P\Big( \exists \ap \in H^n \times T_{p_T}M, \|\ap\| =1 \land  \<\ap, \M_T\ap\> \leq \ep \land \|\pi_{TM} \ap\| \geq g(\ep)\Big) \leq  g(\ep) e^{\eta V^n(\omega_0)}.\]
\end{proposition}

\begin{proof}
    Fix $\eta \in (0,1)$. Let $R>1$ be such that Assumption~\ref{asmp:nondegen} applies. By Proposition~\ref{prop:malliavin-inducted}, applied with $q=1$, there exists $\gamma>0$ such that
    \begin{align*}
      &\P\Big(\exists \ap, \|\ap\| =1 \land  \<\ap, \M_T\ap\> \leq \ep \land  \max_{|k| \leq R} |\<\Theta_{e_k}(p_T) + \Delta e_k - B(v_T, e_k), \ap\>| \geq \ep^{\gamma} \Big) \leq C \ep  e^{\eta V^n(\omega_0)},
    \end{align*}
    where $B(u, v) := \nabla^\perp \Delta^{-1} u \cdot \nabla v + \nabla^\perp \Delta^{-1} v \cdot \nabla u$ denotes the symmetrized nonlinearity.

    Thus for any $g(\ep) \geq 2\ep^{\gamma}$, we have that
    \begin{align}
        &\P\Big(\exists \ap, \|\ap\| =1 \land  \<\ap, \M_T\ap\> \leq \ep \land  \max_{|k| \leq R} |\<\Theta_{e_k}(p_T),  \ap\>| \geq g(\ep) \Big) \leq  C \ep  e^{\eta V^n (\omega_0)}
        \notag\\& \qquad+ \P\Big(\exists \ap, \|\ap\| =1 \land  \<\ap, \M_T\ap\> \leq \ep \land  \max_{|k| \leq R} |\<\Delta e_k -  B(v_T, e_k),  \ap\>| \geq g(\ep) \Big)
        \notag\\&\quad\leq C \ep  e^{\eta V^n (\omega_0)}+ \P\Big(\exists \ap, \|\ap\| =1 \land  \<\ap, \M_T\ap\> \leq \ep \land  C(1+ \|v_T\|_{H^{n+3}}) \|\ap\|_{H^{n-2}} \geq g(\ep) \Big).
        \label{eq:vector-field-overlap-comp-1}
    \end{align}
    Then
    \begin{align}
        &\P\Big(\exists \ap, \|\ap\| =1 \land  \<\ap, \M_T\ap\> \leq \ep \land  C(1+ \|v_T\|_{H^{n+3}}) \|\ap\|_{H^{n-2}} \geq g(\ep) \Big)
        \notag\\&\quad\leq \P\Big(\exists \ap, \|\ap\| =1 \land  \<\ap, \M_T\ap\> \leq \ep \land \|\ap\|_{H^{n-2}} \geq g(\ep)^{1/2} \Big)
        \notag\\&\qquad+\P\Big( C(1+ \|v_T\|_{H^{n+3}})  \geq g(\ep)^{-1/2}\Big)
        \notag\\&\quad\leq C \ep  e^{\eta V^n (\omega_0)} + C g(\ep) e^{\eta V^n (\omega_0)} \leq Cg(\ep) e^{\eta V^n(\omega_0)},
        \label{eq:vector-field-overlap-comp-2}
    \end{align}
    where we choose $g(\ep)^{1/2}$ as in Proposition~\ref{prop:H-n-2-malliavin} and for the other term we use Chebyshev and Proposition~\ref{prop:omega-bounds}. Thus combining~\eqref{eq:vector-field-overlap-comp-1} and \eqref{eq:vector-field-overlap-comp-2}, we have
    \[\P\Big(\exists \ap, \|\ap\| =1 \land  \<\ap, \M_T\ap\> \leq \ep \land  \max_{|k| \leq R} |\<\Theta_{e_k}(p_T),  \ap\>| \geq g(\ep) \Big) \leq Cg(\ep) e^{\eta V^n(\omega_0)}.\]
    Then, using Assumption~\ref{asmp:nondegen}, we compute
    \begin{align*}
        &\P\Big( \exists \ap, \|\ap\| =1 \land  \<\ap, \M_T\ap\> \leq \ep \land \|\pi_{TM} \ap\| \geq g(\ep)^{1/2}\Big)
        \\&\qquad\qquad\leq \P\Big( \exists \ap, \|\ap\| =1 \land  \<\ap, \M_T\ap\> \leq \ep \land \max_{|k| \leq R} |\<\Theta_{e_k}(p_T),  \ap\>| \geq g(\ep)\Big)
        \\&\qquad\qquad \qquad + \Big( \exists \ap, \|\ap\| =1 \land  \<\ap, \M_T\ap\> \leq \ep \land \sup_{v \in T_{p_T} M, \|v\|=1} \big(\max_{|k| \leq R}  |\<\Theta_{e_k}(p_T), v\>\big)^{-1} \geq g(\ep)^{-1/2}\Big)
        \\&\leq Cg(\ep) e^{\eta V^n(\omega_0)} \leq C g(\ep)^{1/2} e^{\eta V^n(\omega_0)},
    \end{align*}
    thus we conclude, after redefining $g(\ep)$.
\end{proof}

The proof of Theorem~\ref{thm:malliavin-nondegenerate} now follows straightforwardly. 

\begin{proof}[Proof of Theorem~\ref{thm:malliavin-nondegenerate}]
      We note that
    \begin{align*}
        \P\Big(\inf_{\substack{\|\ap\| = 1\\ \|\Pi \ap\| \geq \delta}} \<\ap,\M_T \ap\> <\ep\Big) &\leq \P\Big( \exists \ap, \|\ap\| =1 \land  \<\ap, \M_T\ap\> \leq \ep \land \|\pi_{TM} \ap\| \geq \delta/2\Big)
        \\&\quad+  \P\Big( \exists \ap, \|\ap\| =1 \land  \<\ap, \M_T\ap\> \leq \ep \land \|\pi_{H^n} \ap\|_{H^{n-2}} \geq C^{-1}\delta\Big)
        \\&\leq  \P\Big( \exists \ap, \|\ap\| =1 \land  \<\ap, \M_T\ap\> \leq \ep \land \|\pi_{TM} \ap\| \geq g(\ep)\Big)
        \\&\quad + \P\Big( \exists \ap, \|\ap\| =1 \land  \<\ap, \M_T\ap\> \leq \ep \land \|\pi_{H^n} \ap\|_{H^{n-2}} \geq g(\ep)\Big)
        \\&\leq g(\ep) e^{\eta V^n_\alpha(\omega_0)} + C \ep^p e^{\eta V^n_\alpha(\omega_0)}
        \\&\leq C g(\ep) e^{\eta V^n_\alpha(\omega_0)},
    \end{align*}
    where the second inequality holds for $\ep$ small enough and for third inequality we use Proposition~\ref{prop:H-n-2-malliavin} and Proposition~\ref{prop:tangent-space-overlap-malliavin}. Redefining $g(\ep)$ we conclude~\eqref{eq:almost-invertible-malliavin}.
\end{proof}

\subsection{A unit-time smoothing estimate}

Our goal for this subsection is to prove the following smoothing estimate, under Assumptions~\ref{asmp:dynamic-bounds} and~\ref{asmp:nondegen}. This argument follows very similarly to~\cite[Section 4.6]{hairer_ergodicity_2006}, with the primary difference beginning in the proof of Proposition~\ref{prop:size-of-error}, where we need to compensate for the worse stochastic integrability we have in Theorem~\ref{thm:malliavin-nondegenerate}. The main input to this argument is the bound on the Malliavin matrix given by Theorem~\ref{thm:malliavin-nondegenerate}.

\begin{theorem}\label{thm:smoothing}
    For all $\eta>0, p_0 \in M, \gamma \in (0,1)$, there exists $C(\eta,\gamma,p_0)>0$, bounded locally uniformly in $p_0$, such that for all Fr\'echet differentiable observables $\varphi\colon H^n(\T^2) \times M \to \R,$
    \[\|\nabla P_1 \varphi(\omega_0, p_0)\|_{H^5} \leq e^{\eta V^n(\omega_0)} \Big(C \sqrt{P_1 |\varphi|^2(\omega_0,p_0)} + \gamma \sqrt{P_1 \|\nabla \varphi\|_{H^n}^2(\omega_0,p_0)}\Big).\]
\end{theorem}

We note that Proposition~\ref{prop:manifold-processes-good} and Theorem~\ref{thm:smoothing} imply Proposition~\ref{prop:unit-time-smoothing}.

Let $\ap \in H^n(\T^2) \times T_{p_0} M$ with $\|\ap\|=1$ arbitrary. Recalling Definition~\ref{defn:A-def} of $A_{0,t}$, for any $\beta>0$, for $t \in [0,1/2]$ we define
\[v^\beta_t:= A^*_{0,1/2}(\M_{1/2} + \beta)^{-1} J_{0,1/2} \ap.\]
We extend $v^\beta_t$ by $0$ on $[1/2,1]$. Note then that for any $\beta>0$, $v^\beta_t \in L^2([0,1],\R^F)$. We then define
\[\rho^\beta := J_{0,1} \ap - A_{0,1} v^\beta.\]
We then have the fundamental approximate integration by parts computation,
\begin{align}
    \<\ap, \nabla P_1\varphi(\omega_0,p_0)\> &= \E \<J_{0,t}\ap, \nabla \varphi(\omega_t,p_t)\>
    \notag\\&= \E \<A_{0,1} v^\beta, \nabla \varphi(\omega_t,p_t)\>+\E \<\rho^\beta, \nabla \varphi(\omega_t,p_t)\>
    \notag\\&= \E \big(\varphi(\omega_t,p_t) \delta(v^\beta)\big) +\E \<\rho^\beta, \nabla \varphi(\omega_t,p_t)\>
    \notag\\&\leq \big(\E \delta(v^\beta)^2\big)^{1/2} \sqrt{P_1 |\varphi|^2(\omega_0,p_0)}+  \Big(\E \|\rho^\beta\|_{H^n}^2\Big)^{1/2}  \sqrt{P_1 \|\nabla \varphi\|_{H^n}^2(\omega_0,p_0)},
    \label{eq:integration-by-parts}
\end{align}
where $\delta(v^\beta)$ is the Skorokhod integral of $v^\beta$ from $0$ to $1$. We then claim the following two propositions.

\begin{proposition}
    \label{prop:size-of-error}
    For all $\eta, \gamma>0$ there exists $\beta_0(\eta,p_0,\gamma)>0$, locally lower bounded in $p_0$, such that for all $0<\beta \leq \beta_0,$
    \[\E \|\rho^\beta\|_{H^n}^2 \leq \gamma e^{\eta V^n(\omega_0)}.\]
\end{proposition}

\begin{proposition}
    \label{prop:cost-of-control}
    For all $\beta>0$, there exists $C(\beta,\eta,p_0)>0$, locally bounded in $p_0$, such that
    \[\E \delta(v^\beta)^2 \leq Ce^{\eta V^n(\omega_0)}.\]
\end{proposition}

We note that Theorem~\ref{thm:smoothing} is a direct consequence of~\eqref{eq:integration-by-parts}, Proposition~\ref{prop:size-of-error} and Proposition~\ref{prop:cost-of-control}.

\subsubsection{Size of error, proof of Proposition~\ref{prop:size-of-error}}

For Proposition~\ref{prop:size-of-error}, we first need the following estimate on the linearization $J_{s,t}$.
\begin{definition}
    For $R  \in [1,\infty)$, let $\Pi_{\geq R} \colon H^n(\T^2) \times TM \to H^n(\T^2)$ be the orthogonal projection onto the span of Fourier modes with wavenumber $k$ such that $|k| \geq R$.
\end{definition}

\begin{lemma}
\label{lem:linearization-smoothing}
    For all $\eta,q, R> 0$ there exists $C(\eta,q,R)>0$ such that
    \[\E\|\Pi_{\geq R} J_{1/2,1}\|_{H^n \to H^n}^q + \E\|J_{1/2,1} \Pi_{\geq R}\|_{H^n \to H^n}^q \leq CR^{-q} e^{\eta V^n(\omega_0)}.\]
\end{lemma}

\begin{proof}
    We note that
    \[
        \|\Pi_{\geq R} J_{1/2,1}\|_{H^n\to H^n} \leq \frac{1}{R} \|J_{1/2,1}\|_{L^2 \to H^{n+1}}
    \]
    so the bound on $\|\Pi_{\geq R} J_{1/2,1}\|_{H^n \to H^n}$ follows from~\eqref{eq:J-smoothing-bound-asmp} and~\eqref{eq:J-moment-bound-asmp}.

    For the other bound,
    \[\|J_{1/2,1} \Pi_{\geq R}\|_{H^n \to H^n} \leq \|J_{3/4,1}\|_{L^2 \to H^n} \|J_{1/2,3/4} \Pi_{\geq R}\|_{L^2 \to L^2},\]
    so by~\eqref{eq:J-smoothing-bound-asmp},
    \[ \E\|J_{1/2,1} \Pi_{\geq R}\|_{H^n \to H^n}^q \leq C e^{\eta V^n(\omega_0)} \Big(\E \|J_{1/2,3/4} \Pi_{\geq R}\|_{L^2 \to L^2}^{2q}\Big)^{1/2},\]
    we then conclude from~\cite[Lemma 4.17]{hairer_ergodicity_2006}.\footnote{The first thing to note is that since we first project onto the large Fourier modes and so kill any directions in the tangent bundle of the manifold, the bound we need is exactly as is given in~\cite{hairer_ergodicity_2006}. In~\cite{hairer_ergodicity_2006}, they do not give the bound with that specific dependence on $R$, but it is clear from the proof they could and further we never use that specific dependence on $R$, though it is simpler to just state it as is. Finally, we note that the bound on $\|J_{1/2,1} \Pi_{\geq R}\|$ follows from a smoothing estimate for the adjoint equation, using that $\|J_{1/2,1} \Pi_{\geq R}\| = \|\Pi_{\geq R}^* J_{1/2,1}^*\| = \|\Pi_{\geq R} J_{1/2,1}^*\|$. This argument is straightforward, but we omit it as we otherwise do not have to introduce the adjoint linearization equations.}
\end{proof}

We will need the following lemma to control the low modes.

\begin{lemma}
\label{lem:low-modes-low-error}
    For all $\eta, \gamma,R>0$ there exists $\beta_0(\eta,p_0,\gamma,R)>0$, locally lower bounded in $p_0$, such that for all $0<\beta \leq \beta_0,$
    \[\E\|\beta\Pi_{R} (\M_{1/2} + \beta)^{-1} J_{0,1/2}\ap\|_{H^n}^4 \leq \gamma e^{\eta V^n(\omega_0)},\]
    where we recall the definition of $\Pi_R$ from Theorem~\ref{thm:malliavin-nondegenerate}.
\end{lemma}
 Let us first see how this lemma allows us to conclude the bound on $\rho^\beta$.

\begin{proof}[Proof of Proposition~\ref{prop:size-of-error}]
    We first note that
    \begin{align*}
        \rho^\beta &= J_{0,1} \ap - A_{0,1} v^\beta
        \\&= J_{1/2,1}\big(J_{0,1/2}  - A_{0,1/2}A^*_{0,1/2}(\M_{1/2} + \beta)^{-1} J_{0,1/2}\big)\ap
        \\&= \beta J_{1/2,1}(\M_{1/2} + \beta)^{-1} J_{0,1/2}\ap,
    \end{align*}
    where for the final equality we add and subtract $\beta$ from $A_{0,1/2}A^*_{0,1/2}$ and use $\M_{1/2} = A_{0,1/2} A^*_{0,1/2}$. We then compute that, for $R>1,$
    \begin{align*}
    \|\rho^\beta\|_{H^n} &= \|\beta J_{1/2,1}(\M_{1/2} + \beta)^{-1} J_{0,1/2}\ap\|_{H^n}
    \\&\leq  \|\beta J_{1/2,1} \Pi_{\geq R} (\M_{1/2} + \beta)^{-1} J_{0,1/2}\ap\|_{H^n} +  \|\beta J_{1/2,1} \Pi_{R} (\M_{1/2} + \beta)^{-1} J_{0,1/2}\ap\|_{H^n}
    \\&\leq  \|J_{1/2,1} \Pi_{\geq R}\|_{H^n \to H^n}  \|J_{0,1/2}\|_{H^n \to H^n} + \|J_{1/2,1}\|_{H^n \to H^n}  \|\beta\Pi_{R} (\M_{1/2} + \beta)^{-1} J_{0,1/2}\ap\|_{H^n}.\end{align*}
    Then using H\"older's inequality, Lemma~\ref{lem:linearization-smoothing} with $R= \gamma^{-1}$, Lemma~\ref{lem:low-modes-low-error}, and the bound~\eqref{eq:J-moment-bound-asmp} of Assumption~\ref{asmp:dynamic-bounds}, we see that for $\beta \leq \beta_0(p_0,\gamma, \eta)$, with $\beta_0$ locally bounded below in $p_0,$
    \[\E \|\rho^\beta\|_{H^n}^2 \leq C \gamma e^{\eta V^n(\omega_0)},\]
    where $C$ does not depend on $\gamma$. We then conclude after redefining $\gamma$.
\end{proof}

It is in the below proof that we need to slightly deviate from~\cite[Section 4.6]{hairer_ergodicity_2006} due to our worse stochastic integrability in Theorem~\ref{thm:malliavin-nondegenerate}. This is not a substantial problem, as we can use the information given by Theorem~\ref{thm:malliavin-nondegenerate} locally while using an $O(1)$ polynomial moment bound to control the probabilistic tail, as is made clear below.

\begin{proof}[Proof of Lemma~\ref{lem:low-modes-low-error}]
    Fix $R> 1$ and define
    \[\psi^\beta := \beta (\M_{1/2} + \beta)^{-1} J_{0,1/2}\ap,\]
    so that we want to show
    \[\E \|\Pi_R\psi^\beta\|_{H^n}^4 \leq \gamma e^{\eta V^n(\omega_0)}.\]
    We note that for $q>0$
    \[\E \|\psi^\beta\|_{H^n}^q \leq \E \|J_{0,1/2}\|_{H^n \to H^n}^q \leq C e^{\eta V^n(\omega_0)}.\]
    Let $A^\beta$ be the event
    \[A^\beta := \{\|\Pi_R \psi^\beta\|_{H^5} \geq \gamma^{1/4} \|\psi^\beta\|_{H^5}\}.\]
    Then
    \begin{align*}\E \|\Pi_R\psi^\beta\|_{H^n}^4 &\leq C \Big(\E 1_{A^\beta}  \|\Pi_R  \psi^\beta\|_{H^n}^4 + \E 1_{(A^\beta)^C} \|\Pi_R \psi^\beta\|_{H^n}^4\Big)
    \\&\leq C \Big(\E 1_{A^\beta}  \|\Pi_R  \psi^\beta\|_{H^n}^4 + \gamma \E \|\psi^\beta\|_{H^n}^4\Big)
    \\&\leq C \E 1_{A^\beta} \|\Pi_R \psi^\beta\|_{H^n}^4 + C \gamma e^{\eta V^n(\omega_0)}.
    \end{align*}
    Thus to conclude it suffices to show, for $C>0$ independent of $\gamma$, that
    \begin{equation}
        \label{eq:low-modes-small-on-A}
    \E 1_{A^\beta} \|\Pi_R \psi^\beta\|_{H^n}^4  \leq C \gamma e^{\eta V^n(\omega_0)},
    \end{equation}
    and redefine $\gamma$ to absorb the constant. Thus we now focus on showing~\eqref{eq:low-modes-small-on-A}. Our first bound is that
    \begin{equation}
    \label{eq:psi-beta-prob-bound-1}
    \P\big( 1_{A^\beta} \|\Pi_R \psi^\beta\|_{H^n}   \geq \lambda\big) \leq \P\big(\|\psi^\beta\|_{H^n} \geq \lambda\big) \leq C \lambda^{-8} e^{\eta V^n(\omega_0)},
    \end{equation}
    where we use the moment bound on $\|\psi^\beta\|_{H^n}$ and the Chebyshev inequality. It is worth emphasizing that the above bound did not use Theorem~\ref{thm:malliavin-nondegenerate} and is uniform in $\beta$. We will be using this bound to control the tail, when $\lambda$ is large.
    
    We then note that
    \[\<\psi^\beta, \M_{1/2} \psi^\beta\> \leq \<\psi^\beta, (\M_{1/2}+\beta) \psi^\beta\>
        =  \beta\< \beta(\mathscr{M}_{1/2} + \beta)^{-1} J_{0,1/2} \ap,  J_{0,1/2} \ap\> \leq \beta \|J_{0,1/2}\ap\|_{H^n}^2.\]
    Thus we have that
    \begin{align*}
        \P\big(\beta \|J_{0,1/2}\ap\|_{H^n}^2 < \ep 1_{A^\beta}\|\psi^\beta\|_{H^n}^2\big) & \leq \P\big(\<\psi^\beta, \M_{1/2} \psi^\beta\> < \ep 1_{A^\beta}\|\psi^\beta\|_{H^n}^2\big)
        \\&\leq   \P\Big(\inf_{\substack{\|\mathfrak{q}\| = 1\\ \|\Pi_R \mathfrak{q}\| \geq \gamma^{1/4}}} \<\mathfrak{q},\M_{1/2} \mathfrak{q}\> <\ep\Big)
        \\&\leq g(\ep) e^{\eta V^n(\omega_0)},
    \end{align*}
    where we use Theorem~\ref{thm:malliavin-nondegenerate} for the last inequality and so $g$ depends on $\eta, \gamma, p_0$ and is such that $\lim_{\ep \to 0} g(\ep) = 0$ locally uniformly in $p_0$. Thus, we have that
    \begin{align}
        \P\big( 1_{A^\beta} \|\Pi_R \psi^\beta\|_{H^n}  >\lambda\big) &\leq \P\big(\|J_{0,1/2}\ap\|_{H^n} \geq \lambda^{1/2} \beta^{-1/4}\big) + \P\big(\beta \|J_{0,1/2}\ap\|_{H^n}^2 \leq  \lambda^{-1} \beta^{1/2}1_{A^\beta}\|\psi^\beta\|_{H^n}^2\big)
       \notag \\&\leq \big(C\lambda^{-8} \beta^4   + g(\lambda^{-1} \beta^{1/2})\big)e^{\eta V^n(\omega_0)} \leq g(\lambda^{-1} \beta^{1/2}) e^{\eta V^n(\omega_0)},
       \label{eq:psi-beta-prob-bound-2}
    \end{align}
    where for the second inequality we use the moment bounds on $\|J_{0,1/2}\|_{H^n \to H^n}$~\eqref{eq:J-moment-bound-asmp} and Chebyshev and for the last inequality we take $g$ perhaps larger. We will use this bound when $\lambda$ is small.
    
    Using then~\eqref{eq:psi-beta-prob-bound-1} and~\eqref{eq:psi-beta-prob-bound-2}, we compute
    \begin{align*}
    \E 1_{A^\beta} \|\Pi_R \psi^\beta\|_{H^n}^4 &= \int_0^\infty \P(  1_{A^\beta} \|\Pi_R \psi^\beta\|_{H^n} \geq \lambda^{1/4})\,\dif \lambda
    \\&\leq \gamma +e^{\eta V^n(\omega_0)} \int_\gamma^{\gamma^{-1}} g(\lambda^{-1/4} \beta^{1/2})\,\dif \lambda +Ce^{\eta V^n(\omega_0)} \int_{\gamma^{-1}}^\infty  \lambda^{-2} \,\dif \lambda
    \\&\leq C\gamma e^{\eta V^n(\omega_0)} + 2e^{\eta V^n(\omega_0)} \gamma^{-1} \sup_{0 \leq r \leq \gamma^{-1/4} \beta^{1/2}} g(r).
    \end{align*}
    Then for the remaining integral, we can take $\beta \leq \beta_0(\gamma,\eta,R,p_0)$, with $\beta_0$ locally lower bounded in $p_0$, so that $ \sup_{0 \leq r \leq \gamma^{-1/4} \beta^{1/2}} g(r) \leq \gamma^2$, which then completes the bound
    \[\E 1_{A^\beta} \|\Pi_R \psi^\beta\|_{H^n}^4 \leq  C\gamma e^{\eta V^n(\omega_0)},\]
    which is precisely~\eqref{eq:low-modes-small-on-A}. So by the above discussion, we conclude.
\end{proof}

\subsubsection{Cost of control, proof of Proposition~\ref{prop:cost-of-control}}

This argument is highly similar to~\cite[Section 4.6]{hairer_ergodicity_2006} and requires no new ideas. In order to prove Proposition~\ref{prop:cost-of-control}, we first use the fundamental $L^2$-isometry of Skorokhod integral~\cite[Section 4.8]{hairer_ergodicity_2006} or~\cite[Proposition 1.3.1]{nualart_malliavin_2006},

\[\E \delta(v^\beta)^2 \leq \int_0^{1/2} |v_t^\beta|^2\,\dif t + \sum_{k \in F} \int_0^{1/2} \int_0^{1/2} \|\mathcal{D}^k_s v_t^\beta\|^2\, \dif s\dif t =  \|v^\beta\|_{L^2([0,1/2])}^2+ \sum_{k \in F} \int_0^{1/2}\|\mathcal{D}^k_s v^\beta\|_{L^2([0,1/2])}\, \dif s,\]
where $\mathcal{D}^k_s$ is the Malliavin derivative in the $k$th component of the noise at time $s$, that is for a function of the noise $X((W^k_t)_{k\in F})$, we have
\[\mathcal{D}^k_s X((W^\ell_t)_{\ell \in F}) = \lim_{\ep \to 0} \frac{X((W^\ell_t)_{\ell \in F} + \ep\delta_{\ell,k} 1_{t \geq s}) - X((W^\ell_t)_{\ell \in F})}{\ep},\]
that is $\mathcal{D}^k_s X$ gives the differential effect of differentially shifting the noise $W^k_t$ to $W^k_t + \ep 1_{t \geq s}$.

We first note that
\begin{equation}
\label{eq:v-beta-bound-l2}
\|v^\beta\|_{L^2([0,1/2])} =\|A^*_{0,1/2}(\M_{1/2} + \beta)^{-1} J_{0,1/2} \ap\|_{L^2([0,1/2])} \leq \|A^*_{0,1/2}\|_{H^n \to L^2([0,1/2])} \beta^{-1} \|J_{0,1/2}\|_{H^n \to H^n}.
\end{equation}
Note also that by the product rule, for each $k \in F$
\begin{align*}
  &\|\mathcal{D}^k_s v^\beta\|_{L^2([0,1/2])}\\
  &\quad =  \|\mathcal{D}^k_s  (A^*_{0,1/2}(\M_{1/2} + \beta)^{-1} J_{0,1/2} \ap)\|_{L^2([0,1/2])}
    \\&\quad \leq \|\mathcal{D}^k_s A^*_{0,1/2}\|_{H^n \to L^2([0,1/2])} \beta^{-1} \|J_{0,1/2}\|_{H^n \to H^n}
    \\&\quad \qquad+  \|A^*_{0,1/2}\|_{H^n \to L^2([0,1/2])} \beta^{-2} \|\mathcal{D}^k_sA_{0,1/2}\|_{L^2([0,1/2]) \to H^n}\| A^*_{0,1/2}\|_{H^n \to L^2([0,1/2])} \|J_{0,1/2}\|_{H^n \to H^n}
    \\&\quad \qquad + \|A^*_{0,1/2}\|_{H^n \to L^2([0,1/2])} \beta^{-2} \|A_{0,1/2}\|_{L^2([0,1/2]) \to H^n}\|\mathcal{D}^k_s A^*_{0,1/2}\|_{H^n \to L^2([0,1/2])} \|J_{0,1/2}\|_{H^n \to H^n}
    \\&\quad \qquad+  \|A^*_{0,1/2}\|_{H^n \to L^2([0,1/2])} \beta^{-1} \|\mathcal{D}^k_s J_{0,1/2}\|_{H^n \to H^n}.
\end{align*}
Thus Proposition~\ref{prop:cost-of-control} follows from~\eqref{eq:J-moment-bound-asmp} together with the following moment bounds, where $C(q,p_0, \eta)>0$ is locally bounded in $p_0$,
\begin{align}
    \E\|A^*_{0,1/2}\|_{H^n \to L^2([0,1/2])}^q+ \|A_{0,1/2}\|_{L^2([0,1/2]) \to H^n}^q &\leq C e^{\eta V_n(\omega_0)}
    \label{eq:A-op-bound}\\
        \E\sup_{0 \leq s \leq 1/2} \|\mathcal{D}^k_s A^*_{0,1/2}\|_{H^n \to L^2([0,1/2])}^q +   \|\mathcal{D}^k_sA_{0,1/2}\|_{L^2([0,1/2]) \to H^n}^q &\leq C e^{\eta V_n(\omega_0)}
    \label{eq:A-malliavin-deriv-bound}\\
    \E \sup_{0 \leq s,r \leq 1/2}  \|\mathcal{D}^k_s J_{r,1/2}\|_{H^n \to H^n} &\leq C e^{\eta V_n(\omega_0)}\label{eq:J-malliavin-deriv-bound}.
\end{align}

We recall the definition of $A_{0,1/2},$
\[A_{0,1/2} v :=  \sum_{k \in F}\int_0^{1/2} c_k v^k_r J_{r,1/2} e_k \,\dif r.\]
Then the bound~\eqref{eq:A-op-bound} follows from~\eqref{eq:J-moment-bound-asmp}, using $\|A^*_{0,1/2}\|_{H^n \to L^2([0,1/2])} = \|A_{0,1/2}\|_{L^2([0,1/2]) \to H^n}$. For~\eqref{eq:A-malliavin-deriv-bound}, we note that similarly the bound for $A^*_{0,1/2}$ follows from the bound on $A_{0,1/2}$. For $A_{0,1/2}$, we have from the definition that
\[(\mathcal{D}^k_sA_{0,1/2}) v = \sum_{\ell \in F}\int_0^{1/2} c_\ell \mathcal{D}^k_sJ_{r,{1/2}} e_\ell v_r^\ell \,\dif r,\]
thus
\[\sup_{0 \leq s \leq 1/2} \|\mathcal{D}^k_sA_{0,1/2}\|_{L^2([0,1/2]) \to H^n} \leq C \sup_{0 \leq s,r \leq 1/2} \|\mathcal{D}^k_s J_{r,1/2}\|_{H^n \to H^n}.\]
Thus~\eqref{eq:A-malliavin-deriv-bound} follows from~\eqref{eq:J-malliavin-deriv-bound}.

Then for~\eqref{eq:J-malliavin-deriv-bound}, direct computation then gives that
\[\mathcal{D}^k_s J_{r,1/2} \ap = \begin{cases}
    J^2_{s,1/2}( c_k e_k ,J_{r,s} \ap)&s \geq r\\
    J^2_{r,1/2}(c_k J_{s,r} e_k, \ap) & s \leq r,
\end{cases}\]
so that
\[\|\mathcal{D}^k_s J_{r,1/2}\|_{H^n} \leq C\sup_{0 \leq a \leq 1/2} \|J^2_{a,1/2}\|_{H^n \otimes H^n \to H^n} \sup_{0 \leq a \leq b \leq 1/2} \|J_{a,b}\|_{H^n \to H^n}.\]
Thus~\eqref{eq:J-malliavin-deriv-bound} follows from~\eqref{eq:J-moment-bound-asmp} and~\eqref{eq:J2-moment-bound-asmp}. Thus, we conclude the proof of Proposition~\ref{prop:cost-of-control}.

\subsubsection{Proof of Lemma~\ref{lem:two-point-smoothing}}

Finally, we prove the smoothing estimate given by Lemma~\ref{lem:two-point-smoothing} for the two-point process with sharper dependence on $x,y \in \T^2$ than provided by Theorem~\ref{thm:smoothing}.

\begin{definition}
\label{defn:B-xy-def}
    For $(x,y) \in M^2$, we define the linear operator $B_{x,y}$ on $H^4(\T^2) \times \R^2 \times \R^2$ given by
    \[B_{x,y}(\xi,p,q) = (\xi,p,|x-y|q + p).\]
\end{definition}

$B_{x,y}$ should be thought of as a map $H^4(\T^2) \times T_{(x,\tau)} M^T \to H^4(\T^2) \times T_{(x,y)} M^2$. Note that
\[\|B_{x,y} (\xi,p,q)\|_{H^4,2} \leq 2 \|(\xi,p,q)\|_{H^4},\]
where we take $B_{x,y} (\xi,p,q) \in  H^4(\T^2) \times T_{(x,y)} M^2$. Our goal is to approximate the two-point process by the tangent process. We consider the tangent process $(\omega_t,x_t,\tau_t)$ taken with initial data $(\omega_0, x_0, \tau_0),$
with
\[\tau_0 := \frac{y_0-x_0}{|y_0-x_0|}.\]

For $0 \leq s \leq t$,  we let $J^T_{s,t}$ denote the linearization of the tangent process about $(\omega_t, x_t,\tau_t)$. We use $J_{s,t}$ to denote the linearization of the two-point process. We note the tangent process and two-point process are naturally coupled (and hence so are the linearizations $J_{s,t}$ and $J^T_{s,t}$). We use this natural coupling throughout.

The fundamental estimates we will use are the following, the computational proofs of which we defer to Subsection~\ref{ssa:two-point-by-tangent}.
\begin{lemma}
    \label{lem:two-point-linearization-by-tangent}
    For $\eta,q>0$, there exists $C(\eta,q)>0$ such that
    \[\E\sup_{0 \leq s \leq t \leq 1}\|B_{x_0,y_0}^{-1} J_{s,t} B_{x_0,y_0} - J_{s,t}^T\|_{H^4 \to H^4}^q \leq C |x_0 - y_0|^q e^{\eta V(\omega_0)}.\]
\end{lemma}

\begin{lemma}
\label{lem:x-y-dont-separate}
        For any $\eta, q>0$, there exists $C(\eta,q)>0$ such that
    \begin{equation*}
    \E \sup_{0 \leq t \leq 1} \Big(\frac{|y_t-x_t|}{|y_0-x_0|} + \frac{|y_0 - x_0|}{|y_t-x_t|}\Big)^q \leq C e^{\eta V(\omega_0)}.
    \end{equation*}
\end{lemma}

\begin{proof}[Proof of Lemma~\ref{lem:two-point-smoothing}]
    Fix $\eta, \gamma \in (0,1), (\omega_0,x_0,y_0) \in H^4(\T^2) \times M^2, \varphi \colon H^4(\T^2) \times M^2 \to \R$. Throughout this proof, we let $K$ be a constant the changes line by line but is \textit{independent of $\gamma$} while $C$ \textit{can depend on $\gamma$}. We will prove Lemma~\ref{lem:two-point-smoothing} with $K\gamma$ in place of $\gamma$. By redefining $\gamma$, this directly implies the stated version of the lemma. We fix some $0<r_0(\gamma,\eta) \leq \gamma$ to be determined below.

    If $|x_0 - y_0| \geq r_0$, then we simply use Proposition~\ref{prop:unit-time-smoothing} with $\alpha := \gamma r_0$ as well as Lemma~\ref{lem:x-y-dont-separate} to see that\footnote{In the below two displays, the bounds on the norms may seem to go the wrong way, since $\|\cdot\|_{H^4} \lesssim \|\cdot\|_{H^4,2}$. This is because we are taking norms of derivatives $\nabla \psi$ for $\R$-valued functions $\psi$. These derivatives naturally live in the dual space, and so the order of the size of the norms is switched. This is a bit muddled by our conflation of the space with its dual and the Fr\'echet derivative with the gradient.}
    \begin{align*}
        \|\nabla P \varphi(\omega_0, x_0, y_0)\|_{H^4,2} &\leq \|\nabla P \varphi(\omega_0, x_0, y_0)\|_{H^4}
        \\&\leq e^{\eta V(\omega_0)/2} \left(C\sqrt{P_1{|\varphi|}^2(\omega_0, x_0,y_0)} + \gamma r_0\sqrt{P_1{\|\nabla \varphi\|_{H^4}^2}(\omega_0, x_0,y_0)}\right).
    \end{align*}
    Then we note that
    \begin{align*}
    P_1{\|\nabla \varphi\|_{H^4}^2}(\omega_0, x_0,y_0) &\leq |y_0-x_0|^{-2}P_1{\frac{|y_0-x_0|^2}{|y_t-x_t|^2}\|\nabla \varphi\|_{H^4,2}^2}(\omega_0, x_0,y_0)
    \\&\leq Kr_0^{-2} e^{\eta V(\omega_0)/2} \Big(P_1 \|\nabla \varphi\|_{H^4,2}^4(\omega_0,x_0,y_0)\Big)^{1/2},
    \end{align*}
    where we use Lemma~\ref{lem:x-y-dont-separate} and H\"older's inequality. Putting the displays together, we see
    \[ \|\nabla P \varphi(\omega_0, x_0, y_0)\|_{H^4,2} \leq e^{\eta V(\omega_0)} \left(C\sqrt{P_1{|\varphi|}^2(\omega_0, x_0,y_0)} + K\gamma\big(P_1{\|\nabla \varphi\|_{H^4,2}^4}(\omega_0, x_0,y_0)\big)^{1/4}\right).\]

    Thus, we conclude in the case that $|y_0 -x_0| \geq r_0$. For the remainder, we suppose $|y_0 -x_0|<r_0$. We fix $(\xi, p,q) \in H^4(\T^2) \times \R^4$ such that $\|(\xi,p,q)\|_{H^4,2} = 1$. We want to show that
\[|\<(\xi,p,q), \nabla P_1 \varphi(\omega_0,x_0, y_0)\>| \leq  e^{\eta V(\omega_0)}\left(C\sqrt{P_1|\varphi|^2(\omega_0, x_0, y_0)} + K\gamma\big(P_1\|\nabla \varphi\|_{H^4,2}^4(\omega_0, x_0, y_0)\big)^{1/4}\right).\]
Recalling the definition of $J^T_{s,t}$ as the linearization of the tangent process, we define
\[A^T_{s,t} v :=  \sum_{k \in F}\int_s^{t} c_k v^k_r J^T_{r,t} e_k \,\dif r \text{ and } \M_t := A^T_{0,t} (A^T_{0,t})^*.\]
    Then for $\beta>0$, define
\[v^\beta_t:= (A^T)^*_{0,1/2}(\M^T_{1/2} + \beta)^{-1} J^T_{0,1}B_{x_0,y_0}^{-1}(\xi,p,q).\]
Define also
\[\rho^\beta := J_{0,1}(\xi,p,q) - A_{0,1} v^\beta,\]
where the (unsuperscripted) $J_{0,1}, A_{0,1}$ denote the operators with respect to the two-point process. We note that~\eqref{eq:integration-by-parts} gives
\begin{align}
|\<(\xi,p,q), \nabla P_1 \varphi(\omega_0,x_0, y_0)\>| &\leq \big(\E \delta(v^\beta)^2\big)^{1/2} \sqrt{P_1 |\varphi|^2(\omega_0,x_0,y_0)}
\notag\\&\qquad+  \Big(\E \|\rho^\beta\|_{H^4,2}^{4/3}\Big)^{3/4}  \big(P_1 \|\nabla \varphi\|_{H^4,2}^{4}(\omega_0,x_0,y_0)\big)^{1/4}.
 \label{eq:two-point-integration-by-parts}
\end{align}
The control of the first term follows directly from (the proof of) Proposition~\ref{prop:cost-of-control}, giving
\[\big(\E \delta(v^\beta)^2\big)^{1/2} \sqrt{P_1 |\varphi|^2(\omega_0,x_0,y_0)} \leq C e^{\eta V(\omega_0)}\sqrt{P_1 |\varphi|^2(\omega_0,x_0,y_0)},\]
where the constant depends on $\beta$ but is \textit{independent of} $x_0,y_0$, since we use the tangent process when constructing $v^\beta$ and take $\tau_0$ to be unit sized, and so the constant is independent of $\tau_0$, using the local uniformity of the constants in $(x_0,\tau_0)$ and that we take $|\tau_0| = 1$.

Thus to conclude, it suffices to show that there exists $\beta_0(\eta,\gamma)>0$ such that
\[
\Big(\E \|\rho^{\beta_0}\|_{H^4,2}^{4/3}\Big)^{3/4} \leq K\gamma e^{\eta V(\omega_0)}.
\]
To this end, we note that
\begin{align}
    \Big(\E \|\rho^\beta\|_{H^4,2}^{4/3}\Big)^{3/4} &\leq \Big(\E  \big\|J_{0,1}(\xi,p,q) - B_{x_0,y_0}J_{0,1}^T B_{x_0,y_0}^{-1}(\xi,p, q)\big\|_{H^4,2}^{4/3} \Big)^{3/4}
    \notag\\&\qquad+ \Big(\E \big\|B_{x_0,y_0}J_{0,1}^TB_{x_0,y_0}^{-1}(\xi,p, q) - B_{x_0,y_0}A^T_{0,1} v^\beta\big\|_{H^4,2}^{4/3}  \Big)^{3/4}
    \notag\\&\qquad+ \Big(\E \big\|(B_{x_0,y_0}A^T_{0,1} - A_{0,1}) v^\beta\big\|_{H^4,2}^{4/3}\Big)^{3/4},
    \notag\\&\leq  K e^{\eta V(\omega_0)/2}\Big(\E  \big\| B_{x_0,y_0}^{-1} J_{0,1}(\xi,p,q) - J_{0,1}^T B_{x_0,y_0}^{-1}(\xi,p, q)\big\|_{H^4}^{2} \Big)^{1/2}
    \label{eq:rho-beta-triangle-1}\\&\qquad+ K e^{\eta V(\omega_0)/2}\Big(\E \big\|J_{0,1}^TB_{x_0,y_0}^{-1}(\xi,p, q) - A^T_{0,1} v^\beta\big\|_{H^4}^{2}  \Big)^{1/2}
    \label{eq:rho-beta-triangle-2}\\&\qquad+ K e^{\eta V(\omega_0)/2}\Big(\E \big\|(A^T_{0,1} - B_{x_0,y_0}^{-1}A_{0,1}) v^\beta\big\|_{H^4}^2\Big)^{1/2},
    \label{eq:rho-beta-triangle-3}
\end{align}
where we used that $\| \cdot \|_{H^4,2} \leq 2 \|B_{x_1,y_1}^{-1} \cdot\|_{H^4}$ and used H\"older and Lemma~\ref{lem:x-y-dont-separate} to deal with the factor of $\|B_{x_1, y_1}^{-1} B_{x_0,y_0}\| \leq \frac{|x_0 - y_0|}{|x_1-y_1|}$.

We note then that by (the proof of) Proposition~\ref{prop:size-of-error} for the tangent process, together with the construction $v^\beta$, we have that there exists $\beta_0(\eta,\gamma)>0$ such for all $0<\beta \leq \beta_0$,
\begin{equation}
\label{eq:rho-beta-triangle-2-control}
\E \big\|J_{0,1}^TB_{x_0,y_0}^{-1}(\xi,p, q) - A^T_{0,1} v^\beta\big\|_{H^4}^{2}\leq \gamma^2 e^{\eta V(\omega_0)}.
\end{equation}
For~\eqref{eq:rho-beta-triangle-1}, we note that
\begin{align}\E  \big\| B_{x_0,y_0}^{-1} J_{0,1}(\xi,p,q) - J_{0,1}^T B_{x_0,y_0}^{-1}(\xi,p, q)\big\|_{H^4}^{2} &\leq 4\E \|B_{x_0,y_0}^{-1} J_{0,1} B_{x_0,y_0} - J_{0,1}^T\|_{H^4 \to H^4}^2
\notag\\&\leq K |x_0 - y_0|^2 e^{\eta V(\omega_0)} \leq K \gamma^2 e^{\eta V(\omega_0)},
\label{eq:rho-beta-triangle-1-control}
\end{align}
where we use Lemma~\ref{lem:two-point-linearization-by-tangent} for the second inequality and that $|x_0 - y_0| \leq r_0 \leq \gamma$ for the final inequality.

For~\eqref{eq:rho-beta-triangle-3}, we note that
\begin{align}
    \E \big\|(A^T_{0,1} - B_{x_0,y_0}^{-1}A_{0,1}) v^\beta\big\|_{H^4}^2 &\leq K  \sum_{k \in F} \E\int_0^{1} |c_k|^2 |(v^\beta)^k_r|^2 \|J^T_{r,1} - B_{x_0,y_0}^{-1} J_{r,1} B_{x_0,y_0}\|^2_{H^4 \to H^4} \,\dif r
    \notag\\&\leq K \E \|v^\beta\|_{L^2([0,1])}^2 \sup_{0 \leq t \leq 1} \|J^T_{t,1} - B_{x_0,y_0}^{-1} J_{t,1} B_{x_0,y_0}\|^2_{H^4 \to H^4}
    \notag\\&\leq K \beta^{-2} |x_0 -y_0|^2 e^{\eta V(\omega_0)} \leq K \beta^{-2} r_0^2 e^{\eta V(\omega_0)},
    \label{eq:rho-beta-triangle-3-control}
\end{align}
where we use~\eqref{eq:v-beta-bound-l2} to control $\|v^\beta\|_{L^2([0,1])}$, together with H\"older and Lemma~\ref{lem:two-point-linearization-by-tangent} for the second inequality. We now let $r_0 := \beta_0 \gamma \land \gamma$, where $\beta_0$ is as in~\eqref{eq:rho-beta-triangle-2-control}. Then using~\eqref{eq:rho-beta-triangle-1-control} to bound~\eqref{eq:rho-beta-triangle-1},~\eqref{eq:rho-beta-triangle-2-control} to bound~\eqref{eq:rho-beta-triangle-2}, and~\eqref{eq:rho-beta-triangle-3-control} to bound~\eqref{eq:rho-beta-triangle-3}, we see that
\[\Big(\E \|\rho^{\beta_0}\|_{H^4,2}^{4/3}\Big)^{3/4}  \leq K \gamma e^{\eta V(\omega_0)},\]
which by the above discussion allows us to conclude.
\end{proof}

\appendix
\section{\textit{A priori} bounds}
\subsection{\textit{A priori} bounds on stochastic Navier--Stokes}\label{ssa:sns-bounds}

In this section, we prove the \textit{a priori} bounds on stochastic Navier--Stokes given in Proposition~\ref{prop:omega-bounds}. We include these arguments for completeness and the reader's convenience, making no claim of novelty. These propositions in particular have substantial overlap with the bounds given in~\cite[Proposition 4.10]{hairer_ergodicity_2006} as well as the bounds of~\cite[Appendices A-D]{mattingly_dissipative_2002}.

We first recall a standard exponential martingale estimate we will be using often.
\begin{lemma}\label{lem:exp-martingale-bound}
  Let $M_t$ a continuous $L^2$ martingale. Then for all $0 \leq \eta \leq A,$
  \[\E\exp \Big(\sup_{t \geq 0} \eta M_t - \eta A\<M,M\>_t \Big) \leq 2\E e^{\eta M_0}.\]
\end{lemma}

We first show an easier version of~\eqref{eq:l2-energy-bound}.

\begin{lemma}\label{lem:initial-energy-bound}
  There exists $C>0$ such that for all $\eta \leq C^{-1},$
  \[\E \exp\big( \eta\|\omega_t\|_{L^2}^2\big) \leq C \exp\big(\eta e^{-C^{-1} t} \|\omega_0\|_{L^2}^2 \big).\]
\end{lemma}
\begin{proof}
  By It\^o's formula, we have
  \[\frac{\dif}{\dif t} \Big(e^{\nu t} \int \omega_t^2\,\dif x\Big) = \nu e^{\nu t} \int \omega_t^2\,\dif x - 2\nu e^{\nu t}\int |\nabla \omega_t|^2\,\dif x + 2e^{\nu t}\sum_{k \in F} \Big(\int c_k e_k \omega_t\, \dif x\Big)\dif W^k_t + e^{\nu t}\sum_{k \in F} \int c_k^2 e_k^2\,\dif x.\]
  Integrating in time, using Poincar\'e's inequality, and dividing by $e^{\nu t}$, we see that
  \begin{align*}
    \|\omega_t\|_{L^2}^2 - e^{-\nu t}\|\omega_0\|_{L^2}^2 &\leq - \nu \int_0^te^{-\nu (t-s)}\|\omega_s\|_{L^2}^2\,\dif s + 2 \int_0^t e^{-\nu (t-s)}\sum_{k \in F} \Big(\int c_k e_k \omega_s\, \dif x\Big)\dif W^k_s + C
    \\&\leq -  e^{-2 \nu t} C^{-1}\<M,M\>_t+ e^{-\nu t} M_t  + C,
  \end{align*}
  where
  \[M_t := 2 \int_0^t e^{\nu s}\sum_{k \in F} \Big(\int c_k e_k \omega_s\, \dif x\Big)\dif W^k_s,\]
  and we compute
  \[e^{-2 \nu t} \<M,M\>_t = 4\int_0^t e^{-2\nu (t-s)} \sum_{k \in F} \Big(\int c_k e_k \omega_s\,\dif x\Big)^2 ds \leq C\int_0^t e^{-\nu (t-s)} \|\omega_s\|_{L^2}^2ds.\]
  Thus by Lemma~\ref{lem:exp-martingale-bound} and using that $M_0 =0$, we get that for $\eta \leq  C^{-1},$
  \begin{align*}
    \E \exp\big(\eta\|\omega_t\|_{L^2}^2 - e^{- \nu t}\eta \|\omega_0\|_{L^2}^2 - \eta C\big) \leq 2.
  \end{align*}
  Rearranging and using the bound on $\eta$, we get the result.
\end{proof}

\begin{proof}[Proof of~\eqref{eq:l2-energy-bound}]
  By It\^o's formula,
  \begin{align*} \|\omega_t\|_{L^2}^2  + \nu \int_s^t \|\omega_r\|_{H^1}^2\,\dif r - (t-s)\|c\|^2 &= \|\omega_s\|_{L^2}^2+  2 \sum_{k \in F} \int_s^t\Big(\int c_k e_k \omega_r\, \dif x\Big)\dif W^k_r - \nu \int_s^t \|\omega_r\|_{H^1}^2\,\dif r
    \\& \leq  M^s_t - C^{-1} \<M^s, M^s\>_t,
  \end{align*}
  where
  \[M^s_t := \|\omega_s\|_{L^2}^2+  2 \sum_{k \in F} \int_s^t\Big(\int c_k e_k \omega_r\, \dif x\Big)\dif W^k_r\]
  and we compute
  \[\<M^s, M^s\>_t \leq C \int_s^t \|\omega_r\|_{L^2}^2\,\dif r \leq C \nu \int_s^t \|\omega_r\|_{H^1}^2\,\dif r.\]
  Thus Lemma~\ref{lem:exp-martingale-bound} gives
  \[\E \exp\Big(\eta   \Big(\sup_{t \geq s} \|\omega_t\|_{L^2}^2  +\nu\int_s^t \|\omega_r\|_{H^1}^2\,\dif r - C(t-s)\Big)\Big) \leq 2 \E \exp\big(\eta \|\omega_s\|_{L^2}^2\big).\]
  Using Lemma~\ref{lem:initial-energy-bound} to bound the right-hand side and rearranging somewhat, we conclude.
\end{proof}

We now consider controlling higher derivatives of $\omega_t$. Differentiating the equation for $\omega_t$~\eqref{eq:sns}, we see
\begin{equation}
  \label{eq:nabla-n-omega}
  \frac{\dif}{\dif t}\nabla^n \omega_t =   \nu \Delta \nabla^n \omega_t -\Delta^{-1} \nabla^\perp \omega_t \cdot \nabla \nabla^n \omega_t - \sum_{j=0}^{n-1} \binom{n}{j} \Delta^{-1} \nabla^\perp \nabla^{n-j} \omega_t \cdot \nabla \nabla^j \omega_t+ \sum_{k \in F} c_k \nabla^n e_k \dif W^k_t.
\end{equation}

Before continuing, we note the following bound.
\begin{lemma}
  \label{lem:cross-term-interpolation}
  Let $f \in C^\infty(\T^2)$, then there exists $C(n)>0$ such that
  \[2\sum_{j=0}^{n-1} \binom{n}{j} \int |\nabla^n f| |\Delta^{-1} \nabla^\perp \nabla^{n-j} f | |\nabla \nabla^j f|\,\dif x \leq \frac{\nu}{2}\|f\|_{H^{n+1}} + C\|f\|_{L^2}^{2n+4}.\]
\end{lemma}

\begin{proof}
  Fix $n$ and $1 \leq k \leq n$, then by H\"older and Calder\'on-Zygmund estimates to cancel the inverse Laplacian,
  \[\int |\nabla^n f| |\Delta^{-1} \nabla^\perp \nabla^{n-k+1} f | |\nabla \nabla^{k-1} f|\,\dif x \leq C\|f\|_{W^{n,4}} \|f\|_{W^{n-k,p}} \|f\|_{W^{k,q}},\]
  where
  \[q =  \frac{2n-k}{4n} \text{ and } p = \frac{n+k}{4n},\]
  so that
  \[\frac{1}{4}+ \frac{1}{p} + \frac{1}{q} = 1.\]
  We then interpolate $\|f\|_{W^{n-k,p}}$ and $\|f\|_{W^{k,q}}$ between $L^2$ and $W^{n,4}$ using Gagliardo-Nirenberg interpolation. That gives that
  \[ \|f\|_{W^{n-k,p}} \leq C \|f\|_{W^{n,4}}^{\theta}\|f\|_{L^2}^{1-\theta} \text{ and } C \|f\|_{W^{n,4}}^{\theta'}\|f\|_{L^2}^{1-\theta'}\]
  where
  \[\theta = \frac{k}{n} \text{ and } \theta' = \frac{n-k}{n}.\]
  One can readily verify that these parameter values satisfy the hypotheses of the Gagliardo-Nirenberg interpolation inequality. Thus we see
  \[\int |\nabla^n f| |\Delta^{-1} \nabla^\perp \nabla^{n-k+1} f | |\nabla \nabla^{k-1} f|\,\dif x \leq C\|f\|_{W^{n,4}}^2 \|f\|_{L^2}.\]
  To conclude, we interpolate $\|f\|_{W^{n,4}}$ between $L^2$ and $H^{n+1}$. One can readily check that this gives
  \[\|f\|_{W^{n,4}} \leq C\|f\|_{H^{n+1}}^\theta\|f\|_{L^2}^{1-\theta} \text{ with } \theta = \frac{2n+1}{2n+2}.\]
  Thus
  \begin{align*}
    &2\sum_{j=0}^{n-1} \binom{n}{j} \int |\nabla^n f| |\Delta^{-1} \nabla^\perp \nabla^{n-j} f | |\nabla \nabla^j f|\,\dif x \leq C\|f\|_{H^{n+1}}^{2\theta} \|f\|_{L^2}^{3-2\theta}
    \\&\qquad\leq \frac{1}{2}\|f\|_{H^{n+1}}^{2} + C\|f\|_{L^2}^{\frac{3-2\theta}{1-\theta}}=  \frac{1}{2}\|f\|_{H^{n+1}}^{2} + C\|f\|_{L^2}^{2n+4},
  \end{align*}
  as claimed.
\end{proof}

We now show an easier version of~\eqref{eq:Hn-bound}, analogous to Lemma~\ref{lem:initial-energy-bound}.

\begin{lemma}
  \label{lem:Hn-initial}
  There exists $C(n, \|c\|, F,\nu) >0$ such that for all $0 \leq \eta \leq C^{-1}$ and all $t \geq 0$,
  \[ \E \exp\Big(\eta \|\omega_t\|_{H^n}^{\frac{2}{n+2}}\Big)  \leq  C\exp\Big(e^{-C^{-1}t} \eta \|\omega_0\|_{H^n}^{\frac{2}{n+2}}+C \eta \|\omega_0\|_{L^2}^2\Big).\]
\end{lemma}

\begin{proof}
  By It\^o's formula,~\eqref{eq:nabla-n-omega}, and Lemma~\ref{lem:cross-term-interpolation}
  \begin{align*}\frac{\dif}{\dif t} \big(e^{\nu t}\|\omega_t\|_{H^n}^2\big) &\leq  \nu e^{\nu t}\|\omega_t\|_{H^n}^2- \frac{3}{2} \nu e^{\nu t}\|\omega_t\|_{H^{n+1}}^2 + Ce^{\nu t}\|\omega_t\|_{L^2}^{2n+4}\\
    &\qquad + e^{\nu t}\sum_{k \in  F} c_k \Big(\int \nabla^n\omega_t : \nabla^n e_k\,\dif x\Big)\dif W^k_t + Ce^{\nu t}
    \\&\leq -\frac{1}{2} e^{\nu t}\|\omega_t\|_{H^{n}}^2 + Ce^{\nu t}\|\omega_t\|_{L^2}^{2n+4} + e^{\nu t}\sum_{k \in F} c_k \Big(\int \nabla^n\omega_t : \nabla^n e_k\,\dif x\Big)\dif W^k_t + Ce^{\nu t}.
  \end{align*}
  Integrating and dividing by $e^{\nu t}$, we get
  \[\|\omega_t\|_{H^n}^2 - e^{-\nu t} \|\omega_0\|_{H^n}^2 - C\int_0^te^{-\nu (t-s)} \|\omega_s\|_{L^2}^{2n+4}\,\dif s - C \leq e^{-\nu t} M_t - C^{-1} e^{-2 \nu t} \<M,M\>_t,\]
  where
  \[M_t := \int_0^t  e^{\nu s}\sum_{k \in  F} c_k \Big(\int \nabla^n\omega_s : \nabla^n e_k\,\dif x\Big)\dif W^k_s,\]
  and we see the quadratic variation term as in the proof of Lemma~\ref{lem:initial-energy-bound}. Then by Lemma~\ref{lem:exp-martingale-bound}, we see that for $0 \leq \eta \leq C^{-1},$
  \[\E \exp\Big( \eta \|\omega_t\|_{H^n}^2 - e^{-\nu t} \eta \|\omega_0\|_{H^n}^2 - C\eta \int_0^te^{-\nu (t-s)} \|\omega_s\|_{L^2}^{2n+4}\,\dif s\Big) \leq 6.\]
  Note that this immediately implies
  \begin{equation}
    \label{eq:omega-Hn-initial-bad-bound}
    \E \exp\Big( \eta \|\omega_t\|_{H^n}^2 - e^{-\nu t} \eta \|\omega_0\|_{H^n}^2 - C\eta \int_0^te^{-\nu (t-s)} \|\omega_s\|_{L^2}^{2n+4}\,\dif s\Big)_+ \leq 7,
  \end{equation}
  where $x_+ := x \lor 0$. Our goal is now to use~\eqref{eq:l2-energy-bound} to remove the term involving $\|\omega_s\|_{L^2}$. To that end, we compute
  \begin{align}
    &\E \exp\Big(\eta \|\omega_t\|_{H^n}^{\frac{2}{n+2}}\Big)\nonumber\\
    &\quad \leq \E \exp\Big(\Big(\eta^{n+2}\|\omega_t\|_{H^n}^2  - e^{-\nu t} \eta^{n+2} \|\omega_0\|_{H^n}^2 - C\eta^{n+2} \int_0^te^{-\nu (t-s)}\|\omega_s\|_{L^2}^{2n+4}\,\dif s \Big)_+^{\frac{1}{n+2}}\Big)
                                                     \notag\\&\qquad\qquad\times \exp\Big(e^{-C^{-1}t} \eta \|\omega_0\|_{H^n}^{\frac{2}{n+2}} + C \eta \Big(\int_0^t e^{-\nu (t-s)} \|\omega_s\|_{L^2}^{2n+4}\,\dif s\Big)^{\frac{1}{n+2}} \Big)
    \notag\\&\quad \leq C\bigg(\E \exp\Big(2\eta^{n+2}\|\omega_t\|_{H^n}^2  - 2e^{-\nu t} \eta^{n+2} \|\omega_0\|_{H^n}^2 - C\eta^{n+2} \int_0^te^{-\nu (t-s)}\|\omega_s\|_{L^2}^{2n+4}\,\dif s \Big)_+\bigg)^{1/2}
    \notag\\&\qquad\qquad\times \exp\Big(e^{-C^{-1} t} \eta \|\omega_0\|_{H^n}^{\frac{2}{n+2}}\Big) \bigg(\E\exp\Big( C \eta \Big(\int_0^t e^{-\nu (t-s)} \|\omega_s\|_{L^2}^{2n+4}\,\dif s\Big)^{\frac{1}{n+2}} \Big)\bigg)^{1/2}
    \notag\\&\quad \leq C\exp\Big(e^{-C^{-1} t} \eta \|\omega_0\|_{H^n}^{\frac{2}{n+2}}\Big) \bigg(\E\exp\Big( C \eta \Big(\int_0^t e^{-\nu (t-s)} \|\omega_s\|_{L^2}^{2n+4}\,\dif s\Big)^{\frac{1}{n+2}} \Big)\bigg)^{1/2},
    \label{eq:Hn-initial-bound-0}
  \end{align}
  where we use~\eqref{eq:omega-Hn-initial-bad-bound}. Our goal is now to control the last term using~\eqref{eq:l2-energy-bound}.
  \begin{align}
    \E\exp\Big( C \eta \Big(\int_0^t e^{-\nu (t-s)} \|\omega_s\|_{L^2}^{2n+4}\,\dif s\Big)^{\frac{1}{n+2}} \Big) &\leq \E \exp\Big(C\eta  \sup_{0\leq s \leq t} \|\omega_s\|_{L^2}^2\Big)
                                                                                                    \notag\\&\leq C e^{Ct} \exp\big(C \eta \|\omega_0\|_{L^2}^2\big).
    \label{eq:Hn-initial-bound-1}
  \end{align}
  Combining~\eqref{eq:Hn-initial-bound-0} and~\eqref{eq:Hn-initial-bound-1},  we conclude.
\end{proof}

\begin{proof}[Proof of~\eqref{eq:Hn-bound}]
  By It\^o's formula,~\eqref{eq:nabla-n-omega}, and Lemma~\ref{lem:cross-term-interpolation}
  \begin{equation}
    \label{eq:Hn-ito-derivative}
    \frac{\dif}{\dif t} \|\omega_t\|_{H^n}^2 \leq - \frac{3}{2} \nu \|\omega_t\|_{H^{n+1}}^2 + C\|\omega_t\|_{L^2}^{2n+4} + \sum_{k \in  F} c_k \Big(\int \nabla^n\omega_t : \nabla^n e_k\,\dif x\Big)\dif W^k_t + C.
  \end{equation}
  Integrating, we see that
  \begin{equation}
    \label{eq:H-n-by-martingale}
    \|\omega_t\|_{H^n}^2 + \nu \int_s^t \|\omega_r\|_{H^{n+1}}^2\,\dif r -C\int_s^t \|\omega_r\|_{L^2}^{2n+4}\, \dif r - C(t-s) -\|\omega_s\|_{H^n}^2\leq  M^s_t - \frac{1}{C} \<M^s,M^s\>_t,
  \end{equation}
  where
  \[M^s_t :=  \int_s^t \sum_{k \in  F} c_k \Big(\int \nabla^n\omega_r : \nabla^n e_k\,\dif x\Big)\dif W^k_r.\]
  Thus we have that
  \begin{align*}
    &\E \exp\Big( \eta \sup_{s \leq t \leq T} \Big(\|\omega_t\|_{H^n}^2 + \nu \int_s^t \|\omega_r\|_{H^{n+1}}^2\,\dif r -C\int_s^t \|\omega_r\|_{L^2}^{2n+4}\, \dif r\Big)_+^\alpha\Big)
    \\&\quad\leq \E \exp\Big( \eta \sup_{s \leq t \leq T} \Big(\|\omega_t\|_{H^n}^2 + \nu \int_s^t \|\omega_r\|_{H^{n+1}}^2\,\dif r -C\int_s^t \|\omega_r\|_{L^2}^{2n+4}\, \dif r - \|\omega_s\|_{H^n}^2\Big)_+^\alpha + \eta \|\omega_s\|_{H^n}^{2\alpha}\Big)
    \\&\quad\leq  \bigg(\E \exp\Big( 2\eta \sup_{s \leq t \leq T} \Big(\|\omega_t\|_{H^n}^2 + \nu \int_s^t \|\omega_r\|_{H^{n+1}}^2\,\dif r -C\int_s^t \|\omega_r\|_{L^2}^{2n+4}\, \dif r - \|\omega_s\|_{H^n}^2\Big)_+^\alpha\Big)\bigg)^{1/2}
    \\&\quad\qquad\times \bigg(\E \exp\big( 2\eta \|\omega_s\|_{H^n}^{2\alpha}\big)\bigg)^{1/2}
    \\&\quad\leq  C e^{C T}\exp\Big(e^{-C^{-1} s} \eta \|\omega_0\|_{H^n}^{2\alpha}+C\eta \|\omega_0\|_{L^2}^2\Big),
  \end{align*}
  where $\alpha:= \frac{1}{n+2}$ and we use Lemma~\ref{lem:exp-martingale-bound},~\eqref{eq:H-n-by-martingale}, and Lemma~\ref{lem:Hn-initial} for the final line, taking $\eta \leq C^{-1}$. Again, we want to remove the term involving $\|\omega\|_{L^2}$ using~\eqref{eq:l2-energy-bound}. Arguing somewhat similarly to the proof of Lemma~\ref{lem:Hn-initial}, we see
  \begin{align*}
    &\E \exp\Big( \eta \sup_{s \leq t \leq T} \Big(\|\omega_t\|_{H^n}^2 + \nu \int_s^t \|\omega_r\|_{H^{n+1}}^2\,\dif r\Big)^\alpha\Big)
    \\&\quad\leq \bigg(\E \exp\Big( 2\eta \sup_{s \leq t \leq T} \Big(\|\omega_t\|_{H^n}^2 + \nu \int_s^t \|\omega_r\|_{H^{n+1}}^2\,\dif r -C\int_s^t \|\omega_r\|_{L^2}^{2n+4}\, \dif r\Big)_+^\alpha\Big)\bigg)^{1/2}
    \\&\quad\qquad\times \bigg(\E \exp\Big( C\eta\Big(\int_s^T \|\omega_r\|_{L^2}^{2n+4}\, \dif r\Big)^\alpha\Big)\bigg)^{1/2}
    \\&\quad\leq C e^{CT}\exp\Big(e^{-C^{-1} s} \eta \|\omega_0\|_{H^n}^{2\alpha}+C\eta  \|\omega_0\|_{L^2}^2\Big) \bigg(\E \exp\Big( C\eta\Big(\int_s^T \|\omega_r\|_{L^2}^{2n+4}\, \dif r\Big)^\alpha\Big)\bigg)^{1/2}.
  \end{align*}
  Thus to conclude it suffices to see that
  \[ \E \exp\Big( C\eta\Big(\int_s^T \|\omega_r\|_{L^2}^{2n+4}\, \dif r\Big)^\alpha\Big)\leq C e^{CT}\exp\big(C \eta \|\omega_0\|_{L^2}^2\big).\]
  To that end, we note
  \begin{align*}
    \E \exp\Big( C\eta\Big(\int_s^T \|\omega_r\|_{L^2}^{2n+4}\, \dif r\Big)^\alpha\Big)& \leq \E \exp\Big(C \eta \sup_{s \leq r \leq T} \|\omega_r\|_{L^2}^{\frac{4n+4}{2n+4}} \Big(\int_s^T \|\omega_r\|_{L^2}^2\,\dif r \Big)^{\frac{2}{2n+4}}\Big)
    \\&\leq \E \exp\Big(C \eta \sup_{s \leq r \leq T} \|\omega_r\|_{L^2}^{2} + C \eta \int_s^T \|\omega_r\|_{L^2}^2\,\dif r \Big)
    \\&\leq \E \exp\Big(C \eta \sup_{s \leq r \leq T}\Big(\|\omega_r\|_{L^2}^{2} + \nu \int_s^r \|\omega_r\|_{H^1}^2\,\dif r\Big) \Big)
    \\&\leq C e^{CT}\exp\big(C \eta \|\omega_0\|_{L^2}^2\big),
  \end{align*}
  by~\eqref{eq:l2-energy-bound}.
  Thus, we have shown the claim and so conclude.
\end{proof}

\begin{lemma}
  \label{lem:regularization-initial}
  There exists $C(n,\|c\|, F,\nu)>0$ such that for any $0\leq t \leq 1$ and any $\eta \leq C^{-1},$
  \[\E \exp\big(\eta \|\omega_t\|_{H^{n+1}}^{\frac{2}{n+3}}\big) \leq e^{Ct^{-1}} + C\exp\Big( C\eta \|\omega_0\|_{H^n}^{\frac{2}{n+2}} + C \eta \|\omega_0\|_{L^2}^2\Big).\]
\end{lemma}

\begin{proof}
  Fix $t \leq 1$. We first note that for any $K >0$ and $\eta \leq C^{-1},$
  \begin{align*}
    \P\Big( \inf_{0 \leq r \leq t} \|\omega_r\|_{H^{n+1}} \geq K\Big) &\leq \P\Big(\frac{1}{t}\int_0^t \|\omega_r\|^2_{H^{n+1}}\,\dif r\geq K^2\Big)
    \\&\leq \P\Big( \eta \Big(\nu \int_0^t \|\omega_r\|^2_{H^{n+1}}\,\dif r \Big)^{\frac{1}{n+2}} \geq \eta K^{\frac{2}{n+2}} (\nu t)^{\frac{1}{n+2}}\Big)
    \\&\leq \exp\big( {-\eta} K^{\frac{2}{n+2}} (\nu t)^{\frac{1}{n+2}}\big) \E \exp\Big(\eta \Big(\nu \int_0^t \|\omega_r\|_{H^{n+1}}^2\,\dif r\Big)^{\frac{1}{n+2}}\Big)
    \\&\leq C  \exp\big( {-\eta} K^{\frac{2}{n+2}} (\nu t)^{\frac{1}{n+2}}\big) \exp\Big( \eta \|\omega_0\|_{H^n}^{\frac{2}{n+2}} + C \eta \|\omega_0\|_{L^2}^2\Big),
  \end{align*}
  where we use Chebyshev then~\eqref{eq:Hn-bound}.

  Then we define the stopping time
  \[\tau_{K/2} := \inf \{ t :  \|\omega_t\|_{H^{n+1}} \leq K/2\} \land t.\]
  Then by the strong Markov property, for $\eta \leq C^{-1}$
  \begin{align*}
    \P(\|\omega_t\|_{H^{n+1}} \geq K) &\leq \P\big(\sup_{s \in [0,t]} \|\omega_{\tau_{K/2} + s}\|_{H^{n+1}} \geq K\big)
    \\&= \E \P_{\omega_{\tau_{K/2}}}\Big(\sup_{s \in [0,t]} \|\omega_s\|_{H^{n+1}} \geq K\Big)
    \\&\leq \E 1 \land \bigg(C \exp\Big( \eta \|\omega_{\tau_{K/2}}\|_{H^{n+1}}^{{\frac{2}{n+3}}}+C\eta \|\omega_{\tau_{K/2}}\|_{L^2}^2\Big) \exp\big({- \eta K^{\frac{2}{n+3}}}\big)\bigg)
  \end{align*}
  where $\P_v$ denotes the probability measure for the Markov process governing $\omega_t$ started from the initial condition $\omega_0 = v$ and we use Chebyshev and~\eqref{eq:Hn-bound} for the final line. Continuing the computation, we see
  \begin{align*}
    \P(\|\omega_t\|_{H^{n+1}} \geq K) &\leq \P(\tau_{K/2} = t) + C\E \exp\Big( \eta (K/2)^{{\frac{2}{n+3}}}+C\eta \|\omega_{\tau_{K/2}}\|_{L^2}^2\Big) \exp\big({- \eta K^{\frac{2}{n+3}}}\big)
    \\&\leq \P\big(\inf_{0 \leq r \leq t} \|\omega\|_{H^{n+1}} \geq K/2\big) + C \exp\big({- C^{-1}\eta K^{\frac{2}{n+3}}}\big)\E \exp\Big(C\eta \sup_{t \in [0,2t]}\|\omega_t\|_{L^2}^2\Big)
    \\&\leq C  \exp\big( {-\eta} C^{-1} K^{\frac{2}{n+2}} t^{\frac{1}{n+2}}\big) \exp\Big( \eta \|\omega_0\|_{H^n}^{\frac{2}{n+2}} + C \eta \|\omega_0\|_{L^2}^2\Big)
    \\&\qquad+ C \exp\big({- C^{-1}\eta K^{\frac{2}{n+3}}}\big) \exp\big(C\eta \|\omega_0\|_{L^2}^2\big).
  \end{align*}
  Thus
  \begin{align*}
    \E \exp\big(\alpha \|\omega_t\|_{H^{n+1}}^{\frac{2}{n+3}}\big) &= \int_0^\infty \P\Big(\|\omega_t\|_{H^{n+1}} \geq \alpha^{-\frac{n+3}{2}} \big(\log \lambda\big)^{\frac{n+3}{2}}\Big)\,\dif \lambda
    \\&\leq a + \int_a^\infty C  \exp\big( {-\eta} C^{-1} \alpha^{-\frac{n+3}{n+2}} \big(\log \lambda\big)^{\frac{n+3}{n+2}} t^{\frac{1}{n+2}}\big) \exp\Big( \eta \|\omega_0\|_{H^n}^{\frac{2}{n+2}} + C \eta \|\omega_0\|_{L^2}^2\Big)
    \\&\qquad\qquad\qquad+ C \exp\big({- C^{-1}}\eta \alpha^{-1}\log \lambda\big) \exp\big(C\eta \|\omega_0\|_{L^2}^2\big)\,\dif \lambda
    \\&\leq  a + C\exp\Big( \eta \|\omega_0\|_{H^n}^{\frac{2}{n+2}} + C \eta \|\omega_0\|_{L^2}^2\Big)\int_a^\infty   \lambda^{-  (\log \lambda t/\alpha)^{\frac{1}{n+2}}} \,\dif \lambda
    \\&\qquad\qquad\qquad+ C  \exp\big(C\eta \|\omega_0\|_{L^2}^2\big)
  \end{align*}
  provided $a\geq 1$ and $\eta \geq C\alpha$. Then, if we take
  \[a:=  e^{2^{n+2}\frac{\alpha}{t}},\]
  the remaining integral is uniformly bounded, and so taking $\eta = C\alpha,$ we get for $\alpha\leq C^{-1}$
  \[\E \exp\big(\alpha \|\omega_t\|_{H^{n+1}}^{\frac{2}{n+3}}\big) \leq e^{C\alpha t^{-1}} + C\exp\Big( C\alpha \|\omega_0\|_{H^n}^{\frac{2}{n+2}} + C \alpha \|\omega_0\|_{L^2}^2\Big).\]
  Relabeling parameters, we conclude.
\end{proof}

Conditioning on times $jt/n$ for $j = 0,\dots,n-1$ and iteratively using the Markov property, Lemma~\ref{lem:regularization-initial}, and~\eqref{eq:l2-energy-bound}, we see

\begin{lemma}
  \label{lem:Hn-regularization-initial-2}
  There exists $C(n,\|c\|, F,\nu)>0$ such that for any $0\leq t \leq 1$ and any $\eta \leq C^{-1},$
  \[\E \exp\big(\eta \|\omega_t\|_{H^{n}}^{\frac{2}{n+2}}\big) \leq e^{Ct^{-1}} + C\exp\Big(C \eta \|\omega_0\|_{L^2}^2\Big).\]
\end{lemma}

Using the Markov property and~\eqref{eq:Hn-bound} conditioned on time $s/2 \land 1$, Lemma~\ref{lem:Hn-regularization-initial-2} implies~\eqref{eq:Hn-regularization}.

\subsection{\textit{A priori} bounds on the linearization and second derivative of the coupled systems}
\label{ssa:linearization-bounds}

Our goal in this section is to prove half of Proposition~\ref{prop:manifold-processes-good} which gives that the one-point, two-point, tangent, projective, and the matrix processes satisfy Assumption~\ref{asmp:dynamic-bounds}. We note that it suffices to verify the bounds for the total system $(\omega_t, x_t,y_t,\tau_t,  v_t, A_t) \in H^4(\T^2) \times \T^2 \times \T^2 \times \R^2 \times S^1 \times \SL(2,\R)$. We emphasize though that this total system is not hypoelliptic, so when applying the hypoellipticity theory, we consider only a subset of these coordinates.

The bounds for the two-point process follow very similarly to the bounds for the one-point process, as $y_t$ can be bounded exactly symmetrically to $x_t$. The bounds for the projective process $v_t$ follow directly from the bounds on the tangent process $\tau_t$, as $v_t = \frac{\tau_t}{\|\tau_t\|}$. Thus, we actually only consider the somewhat simpler system $(\omega_t, x_t, \tau_t, A_t)$, removing the second point $y_t$ and the normalized tangent vector $v_t$.

We fix some initial data $(\omega_0,x_0, \tau_0, A_0)$, and times $0 \leq s \leq t \leq T$, as well as a (random) direction $(\varphi_s, y_s, p_s, B_s) \in H^4(\T^2) \times \R^2 \times \R^2 \times T_{A_s} \SL(2, \R)$ such that (surely) $\|(\varphi_s, y_s, u_s, B_s)\| = 1$. We take the natural embedding $\SL(2,\R) \subseteq \R^{2 \times 2}$, which induces the embedding $T_{A_s} \SL(2, \R) \subseteq \R^{2\times 2}$. We use these embeddings when writing equations. We then let $(\varphi_t, y_t,p_t,B_t)$ be the directional derivative of $(\omega_t, x_t, \tau_t,A_t)$---viewed as (random) functions of $\omega_s, x_s, \tau_s, A_s$---in the $(\varphi_s, y_s, p_s, B_s)$ direction.\footnote{Note that this $y$ is a different $y$ from in the two-point process.} We then let $(\psi_t, z_t,  q_t, C_t)$ be the second directional derivative of the same function in the same direction.

In the language of Definition~\ref{defn:linearization}, we have that
\[(\varphi_t, y_t,p_t,B_t) = J_{s,t} (\varphi_s,y_s,p_s,B_s)\]
and
\[(\psi_t,z_t,q_t,C_t) = J^2_{s,t}((\varphi_s, y_s,q_s,B_s), (\varphi_s, y_s,q_s,B_s)).\]
We note that to control the symmetric bilinear form $J^2_{s,t}$, it suffices to control the diagonal.
Direct computation verifies that these quantities are well-defined and further they solve the equations
\begin{equation}
  \label{eq:linearization}
  \begin{cases}
    \dot \varphi_t = \nu\Delta \varphi_t - \nabla^\perp \Delta^{-1} \varphi_t \cdot \nabla \omega_t - \nabla^\perp \Delta^{-1} \omega_t \cdot \nabla \varphi_t\\
    \dot y_t = y_t \cdot \nabla u_t(x_t) + \nabla^\perp \Delta^{-1}\varphi(x_t)\\
    \dot p_t = p_t \cdot \nabla u_t(x_t) + \tau_t \otimes y_t : \nabla^2 u_t(x_t) + \tau_t \cdot \nabla \nabla^\perp \Delta^{-1} \varphi_t(x_t)\\
    \dot B_t = B_t \nabla u_t(x_t) + A_t y_t \cdot \nabla^2 u_t(x_t) + A_t \nabla \nabla^\perp \Delta^{-1} \varphi_t(x_t),
  \end{cases}
\end{equation}
where these equations have initial data given by $(\varphi_s, y_s,p_s,B_s)$. For the second derivative, we have
\begin{equation}
  \label{eq:second-derivative}
  \begin{cases}
    \dot \psi_t =& \nu \Delta \psi_t - \nabla^\perp \Delta^{-1} \psi_t \cdot \nabla \omega_t - \nabla^\perp \Delta^{-1} \omega_t \cdot \nabla \psi_t - 2\nabla^\perp \Delta^{-1} \varphi_t \cdot \nabla \varphi_t\\
    \dot z_t =& z_t \cdot \nabla u_t(x_t) + y_t \otimes y_t : \nabla^2 u_t(x_t)  +  2 y_t \cdot \nabla \nabla^\perp \Delta^{-1} \varphi_t(x_t)+ \nabla^{\perp} \Delta^{-1} \psi_t(x_t)\\
    \dot q_t =& q_t \cdot \nabla u_t(x_t) + \tau_t \nabla \nabla^\perp \Delta^{-1} \psi_t(x_t) + \tau_t \otimes z_t : \nabla^2 u_t(x_t) +\tau_t \otimes y_t \otimes y_t : \nabla^3 u_t(x_t)
    \\&\qquad+2  \tau_t \otimes y_t : \nabla^2 \nabla^\perp \Delta^{-1} \varphi_t(x_t)  + 2 p_t \otimes y_t : \nabla^2 u_t(x_t) + 2 p_t \cdot \nabla \nabla^\perp \Delta^{-1} \varphi_t(x_t) \\
    \dot C_t =&C_t \nabla u_t(x_t) + A_t \nabla \nabla^\perp \Delta^{-1} \psi_t(x_t) + A_t z_t \cdot \nabla^2 u_t(x_t) + A_t y_t \otimes y_t : \nabla^3 u_t(x_t)
    \\&\qquad+ 2 B_t \nabla \nabla^\perp \Delta^{-1} \varphi_t(x_t) + 2A_t y_t \cdot \nabla^2\nabla^\perp \Delta^{-1} \varphi_t(x_t)  +2  B_t y_t \cdot \nabla^2 u_t(x_t),
  \end{cases}
\end{equation}
where these equations are given initial data $0$.

We first need to bound $\varphi_t$ in $L^2$.

\begin{lemma}
  \label{lem:phi-l2-bound}
  There exists $C>0$ such that for $t \geq s,$
  \[\|\varphi_t\|_{L^2} \leq \|\varphi_s\|_{L^2} \exp\Big(C\int_s^t \|\omega_r\|_{H^1}^{4/3}\,\dif r\Big).\]
\end{lemma}

\begin{proof}
  We compute
  \begin{align*}\frac{\dif}{\dif t}\frac{1}{2} \|\varphi_t\|_{L^2}^2 &\leq - \nu \|\varphi_t\|_{H^1} + C\|\omega_t\|_{H^1} \|\varphi_t\|_{W^{-1,4}} \|\varphi_t\|_{L^4}
    \\&\leq  - \nu \|\varphi_t\|_{H^1} + C\|\omega_t\|_{H^1} \|\varphi_t\|_{L^2}^{3/2} \|\varphi_t\|_{H^1}^{1/2} \leq C \|\omega_t\|_{H^1}^{4/3} \|\varphi_t\|_{L^2}^2,
  \end{align*}
  where we interpolate for the second inequality and use Young's for the final inequality. Using Gr\"onwall, we conclude.
\end{proof}

We need to bound $\|\varphi_t\|_{H^4}$. Taking derivatives of the equation for $\varphi_t$, we see,
\begin{equation}
  \label{eq:linearization-derivative}
  \frac{\dif}{\dif t} \nabla^n \varphi_t=\nu \Delta \nabla^n \varphi_t  - \sum_{j=0}^n \binom{n}{j} \big(  \nabla^{n-j} \nabla^{\perp} \Delta^{-1} \omega_t \cdot \nabla \nabla^j\varphi_t + \nabla^{\perp} \Delta^{-1}  \nabla^j \varphi_t \cdot \nabla \nabla^{n-j} \omega_t\big).
\end{equation}

We will need the following control on the second term on the right-hand side.
\begin{lemma}
  For $n \geq 1,$ there exists $C(n)>0$ such that
  \begin{align}
    &\Big|\int \nabla^n \varphi_t : \sum_{j=0}^n \binom{n}{j} \big(  \nabla^{n-j} \nabla^{\perp} \Delta^{-1} \omega_t \cdot \nabla \nabla^j\varphi_t + \nabla^{\perp} \Delta^{-1}  \nabla^j \varphi_t \cdot \nabla \nabla^{n-j} \omega_t\big)\,\dif x\Big|
      \notag\\&\qquad\qquad\qquad\leq  \frac{\nu}{2} \|\varphi_t\|_{H^{n+1}}^2 + C\|\omega_t\|_{H^n}^{2n+2} \|\varphi_t\|_{L^2}^2.
    \label{eq:Hn-linearization-crossterms}
  \end{align}
\end{lemma}
\begin{proof}
  First we compute
  \begin{align*}
    &\Big|\int \nabla^n \varphi_t : \sum_{j=0}^n \binom{n}{j} \big(  \nabla^{n-j} \nabla^{\perp} \Delta^{-1} \omega_t \cdot \nabla \nabla^j\varphi_t + \nabla^{\perp} \Delta^{-1}  \nabla^j \varphi_t \cdot \nabla \nabla^{n-j} \omega_t\big)\,\dif x\Big|
    \\&\quad \leq \Big|\int \nabla^n \varphi_t : \nabla^{\perp} \Delta^{-1}  \varphi_t \cdot \nabla \nabla^n \omega_t\,\dif x \Big|+\Big|\int \nabla^n \varphi_t : \nabla^{\perp} \Delta^{-1} \omega_t \cdot \nabla \nabla^n\varphi_t \,\dif x\Big| + C\|\omega_t\|_{H^n} \|\varphi_t\|_{W^{n,4}}^2
    \\&\quad\leq \Big|\int \nabla^{n+1} \varphi_t : \nabla^{\perp} \Delta^{-1}  \varphi_t \nabla^n \omega_t\,\dif x \Big|+\Big|\int \nabla^n \varphi_t : \nabla^{\perp} \Delta^{-1} \omega_t \cdot \nabla \nabla^n\varphi_t \,\dif x\Big| + C\|\omega_t\|_{H^n} \|\varphi_t\|_{W^{n,4}}^2
    \\&\quad\leq \|\varphi_t\|_{H^{n+1}}  \|\nabla^{\perp} \Delta^{-1}  \varphi_t\|_{L^\infty} \| \omega_t\|_{H^n} + \|\varphi_t\|_{H^n} \|\nabla^{\perp} \Delta^{-1} \omega_t\|_{L^\infty} \|\varphi_t\|_{H^{n+1}}  + C\|\omega_t\|_{H^n} \|\varphi_t\|_{W^{n,4}}^2
    \\&\quad\leq \|\varphi_t\|_{H^{n+1}}  \|\varphi_t\|_{H^1} \| \omega_t\|_{H^n} + \|\varphi_t\|_{H^n} \| \omega_t\|_{H^1} \|\varphi_t\|_{H^{n+1}}  + C\|\omega_t\|_{H^n} \|\varphi_t\|_{W^{n,4}}^2.
  \end{align*}
  Then note
  \[\|\varphi_t\|_{W^{n,4}}^2 \leq \|\varphi_t\|_{H^{n+1/2}}^2 \leq \|\varphi_t\|_{H^n} \|\varphi_t\|_{H^{n+1}} \leq \|\varphi_t\|_{L^2}^{\frac{1}{n+1}} \|\varphi_t\|_{H^{n+1}}^{\frac{2n+1}{n+1}}.\]
  Thus putting it together
  \begin{align*}
    &\Big|\int \nabla^n \varphi_t : \sum_{j=0}^n \binom{n}{j} \big(  \nabla^{n-j} \nabla^{\perp} \Delta^{-1} \omega_t \cdot \nabla \nabla^j\varphi_t + \nabla^{\perp} \Delta^{-1}  \nabla^j \varphi_t \cdot \nabla \nabla^{n-j} \omega_t\big)\,\dif x\Big|
    \\&\qquad\qquad \leq C\|\omega_t\|_{H^n}\|\varphi_t\|_{L^2}^{\frac{1}{n+1}} \|\varphi_t\|_{H^{n+1}}^{\frac{2n+1}{n+1}} \leq \frac{\nu}{2} \|\varphi_t\|_{H^{n+1}}^2 + C\|\omega_t\|_{H^n}^{2n+2} \|\varphi_t\|_{L^2}^2.
  \end{align*}
\end{proof}

\begin{lemma}
  \label{lem:phi-Hn-bound}
  For all $n \in \N$, there exists $C(n) >0$ such that
  \begin{equation}
    \label{eq:linearization-Hn-bound}
    \|\varphi_t\|_{H^n} + \frac{\nu}{2}\int_s^t \|\varphi_r\|_{H^{n+1}}\,\dif r\leq  \|\varphi_s\|_{H^n} + C (t-s) \sup_{s \leq r \leq t}\big( \|\omega_r\|_{H^n}^{2n+2} +  \|\varphi_r\|_{L^2}^{2}\big).
  \end{equation}
\end{lemma}

\begin{proof}
  By~\eqref{eq:linearization-derivative} and~\eqref{eq:Hn-linearization-crossterms}, we have that
  \begin{equation*}
    \frac{\dif}{\dif t} \|\varphi_t\|_{H^n}^2 + \frac{\nu}{2} \|\varphi_t\|_{H^{n+1}}^2 \leq  C\|\omega_t\|_{H^n}^{2n+2} \|\varphi_t\|_{L^2}^2 \leq C \|\omega_t\|_{H^n}^{4n+4} +  C\|\varphi_r\|_{L^2}^{4}
  \end{equation*}
  Integrating this bound on $[s,t]$, we conclude.
\end{proof}

We now get control on $A_t$ and $\tau_t$.

\begin{lemma}
  \label{lem:A-bound}
  There exists $C>0$ such that
  \[|A_t| \leq \exp\Big(C\int_s^t \|\omega_r\|_{H^{5/4}}\,\dif r\Big) |A_s|,\]
  and
  \[|\tau_t| \leq \exp\Big(C\int_s^t \|\omega_r\|_{H^{5/4}}\,\dif r\Big) |\tau_s|.\]
\end{lemma}
\begin{proof}
  Using the equation for $A_t$,
  \[
    \frac{\dif}{\dif t} \frac{1}{2} |A_t|^2 \leq |A_t|^2 \|\nabla u\|_{L^\infty} \leq C |A_t|^2 \|\omega\|_{H^{5/4}},\]
  so we conclude by Gr\"onwall. The computation for $\tau$ follows similarly.
\end{proof}

We can now control the first derivatives $y_t,p_t,B_t$.

\begin{lemma}
  There exists $C>0$ such that
  \begin{align}
    |y_t| &\leq Ce^{t-s}\exp\Big(C\int_s^t \|\omega_r\|_{H^{5/4}}\,\dif r\Big) \big(|y_s| + \sup_{s \leq r \leq t} \|\varphi_r\|_{H^1}\big), \label{eq:y-bound}\\
    |p_t| &\leq C e^{t-s} \exp\Big(C\int_s^t \|\omega_r\|_{H^{5/4}}\,\dif r\Big)\big(|p_s| + \sup_{s \leq r \leq t} \|\varphi_r\|_{H^2}^{3/2}+ \|\omega_r\|_{H^3}^3 + |y_r|^3\big)
            \label{eq:p-bound}\\
    |B_t| &\leq C e^{t-s} \exp\Big(C\int_s^t \|\omega_r\|_{H^{5/4}}\,\dif r\Big)\big(|B_s| + \sup_{s \leq r \leq t} |A_r|^3 + |y_r|^3 + \|\omega_r\|_{H^{3}}^3 +\|\varphi_r\|_{H^2}^{3/2} \big).
            \label{eq:B-bound}
  \end{align}
\end{lemma}

\begin{proof}
  We first show~\eqref{eq:y-bound}. Using the equation for $y_t,$
  \[
    \frac{\dif}{\dif t} \frac{1}{2}|y_t|^2 \leq |y_t|^2 (\|\nabla u_t\|_{L^\infty} + 1) + \|\nabla^\perp \Delta^{-1} \varphi_t\|_{L^\infty}^2 \leq |y_t|^2 (C\| \omega_t\|_{H^{5/4}} + 1) + C\|\varphi_t\|_{H^1}^2.
  \]
  So we conclude by Gr\"onwall.

  For~\eqref{eq:p-bound}, using the equation for $p_t$, we compute
  \begin{align*}
    \frac{\dif}{\dif t} \frac{1}{2}|p_t|^2& \leq |p_t|^2 \|\nabla u_t\|_{L^\infty} + |y_t| |p_t| |\tau_t| \|\nabla^2 u_t\|_{L^\infty} + |p_t| |\tau_t| \|\nabla \nabla^\perp \Delta^{-1} \varphi_t\|_{L^\infty}
    \\&\leq  |p_t|^2 (C\|\omega_t\|_{H^{5/4}} + 1) + C|y_t|^6 + C \|\omega_t\|_{H^3}^6 + C\|\varphi_t\|_{H^2}^3 + C|\tau_t|^6,
  \end{align*}
  so we conclude by Gr\"onwall.

  For~\eqref{eq:B-bound}, using the equation for $B_t$, we compute
  \begin{align*}
    \frac{\dif}{\dif t} |B_t|^2 &\leq |B_t|^2 \|\nabla u\|_{L^\infty} + |B_t||A_t| |y_t| \|\nabla^2 u\|_{L^\infty} + |B_t| |A_t| \|\nabla^2 \Delta^{-1} \varphi_t\|_{L^\infty}
    \\&\leq |B_t|^2 \big(C \|\omega_t\|_{H^{5/4}} + 1\big) + C |A_t|^6 + C |y_t|^6 + C\|\omega_t\|_{H^3} + C\|\varphi_t\|_{H^2}^3,
  \end{align*}
  so we conclude by Gr\"onwall.
\end{proof}

\begin{lemma}
  \label{lem:psi-l2-bound}
  There exists $C>0$ such that
  \[ \|\psi_t\|_{L^2} \leq C(t-s)\exp\Big(C\int_s^t \|\omega_r\|_{H^1}^{4/3}\,\dif r\Big)  \sup_{s \leq r \leq t} \|\varphi_r\|_{H^1}^2.\]
\end{lemma}

\begin{proof}
  Using the equation for $\psi_t$, we compute as in Lemma~\ref{lem:phi-l2-bound},
  \[
    \frac{\dif}{\dif t} \frac{1}{2} \|\psi_t\|_{L^2}^2 \leq C \|\omega_t\|_{H^1}^{4/3} \|\psi_t\|_{L^2}^2 + C\|\varphi_t\|_{H^1}^4,\]
  and so conclude by Gr\"onwall, using that $\psi_s =0$.
\end{proof}

\begin{lemma}
  \label{lem:psi-Hn-bound}
  For all $n \geq 1$, there exists $C(n) >0$ such that
  \[\|\psi_t\|_{H^n} \leq C(t-s) \big(\sup_{s \leq r \leq t} \|\omega_r\|_{H^n}^{4n+4} +  \|\psi_r\|_{L^2}^2  +  \|\varphi_r\|_{H^{n}}^2\big)\]
\end{lemma}

\begin{proof}
  Differentiating the equation for $\psi_t,$ we see that
  \begin{align*}
    \frac{\dif}{\dif t} \nabla^n \psi_t&=\nu \Delta \nabla^n \psi_t  - \sum_{j=0}^n \binom{n}{j} \big(  \nabla^{n-j} \nabla^{\perp} \Delta^{-1} \omega_t \cdot \nabla \nabla^j\psi_t + \nabla^{\perp} \Delta^{-1}  \nabla^j \psi_t \cdot \nabla \nabla^{n-j} \omega_t\big)
    \\&\qquad\qquad\qquad\qquad-2 \sum_{j=0}^n \binom{n}{j}  \nabla^\perp \Delta^{-1} \nabla^{n-j} \varphi_t \cdot \nabla \nabla^j \varphi_t.
  \end{align*}
  For the final term, we note for $n\geq 1,$
  \begin{align*}
    &\Big| 2 \sum_{j=0}^n \binom{n}{j} \int \nabla^n \psi_t:  \nabla^\perp \Delta^{-1} \nabla^{n-j} \varphi_t \cdot \nabla \nabla^j \varphi_t\,\dif x\Big|
    \\&\qquad \leq C  \|\psi_t\|_{H^{n+1}} \|\nabla^\perp \Delta^{-1} \varphi_t\|_{L^2} \|\varphi_t\|_{H^n} +C \sum_{j=0}^{n-1} \|\psi_t\|_{W^{n,4}} \|\varphi_t\|_{H^{j+1}} \|\varphi_t\|_{W^{n-j-1,4}}\\ &\qquad \leq C\|\psi_t\|_{H^{n+1}} \|\varphi_t\|_{H^{n}}^2.
  \end{align*}

  Dealing with the other term as in Lemma~\ref{lem:phi-Hn-bound}, we see
  \begin{align*}
    \frac{\dif}{\dif t} \|\psi_t\|_{H^n}^2 \leq  C\|\omega_t\|_{H^n}^{4n+4} +  \|\psi_t\|_{L^2}^4  + C \|\varphi_t\|_{H^{n}}^4.
  \end{align*}
  Integrating this bound on $[s,t],$ we conclude.
\end{proof}

\begin{lemma}
  There exists $C>0$ such that
  \begin{align}
    |z_t|&\leq Ce^{t-s}\exp\Big(C\int_s^t \|\omega_r\|_{H^{5/4}}\,\dif r\Big)\big(\sup_{s \leq r \leq t} |y_r|^4 +  \|\omega_r\|_{H^3}^2 +  \|\varphi_r\|_{H^2}^{4/3} +  \|\psi_r\|_{H^1}\big)
           \label{eq:z-bound}
    \\|q_t| &\leq Ce^{t-s}\exp\Big(C\int_s^t \|\omega_r\|_{H^{5/4}}\,\dif r\Big) \nonumber \\
    &\qquad \times\big( \sup_{s \leq r \leq t} \|\psi_r\|_{H^2}^2 + |z_r|^3  + |y_r|^6 +  \|\omega_r\|_{H^4}^3+ \|\varphi_r\|_{H^3}^3  + |p_r|^3 + |\tau_r|^3 +1 \big)
             \label{eq:q-bound}
    \\|C_t| &\leq Ce^{t-s}\exp\Big(C\int_s^t \|\omega_r\|_{H^{5/4}}\,\dif r\Big)\nonumber \\
    &\qquad \times \big(\sup_{s \leq r \leq t}  \|\psi_r\|_{H^2}^2+ |A_r|^3 +|B_r|^3+ |z_r|^3 +|y_r|^{6}+ \|\omega_r\|_{H^4}^3  +  \|\varphi_r\|_{H^3}^2 +1 \big)
             \label{eq:C-bound}
  \end{align}
\end{lemma}

\begin{proof}
  For~\eqref{eq:z-bound}, we compute using the equation for $z_t$,
  \begin{align*}
    \frac{\dif}{\dif t} \frac{1}{2} |z_t|^2 &\leq |z_t|^2 \|\nabla u\|_{L^\infty} + |z_t||y_t|^2 \|\nabla^2 u\|_{L^\infty} + 2 |z_t| |y_t| \|\nabla^2 \Delta^{-1} \varphi_t\|_{L^\infty} + |z_t| \|\nabla \Delta^{-1} \psi_t\|_{L^\infty}
    \\&\leq |z_t|^2 \big( C\|\omega\|_{H^{5/4}} +1\big) + C|y_t|^8 + C \|\omega_t\|_{H^3}^4 + C \|\varphi_t\|_{H^2}^{8/3} + C \|\psi_t\|_{H^1}^2,
  \end{align*}
  so we conclude by Gr\"onwall.

  For~\eqref{eq:q-bound}, we compute using the equation for $q_t$,
  \begin{align*}
    \frac{\dif}{\dif t} \frac{1}{2} |q_t|^2 &\leq |q_t|^2 \|\nabla u\|_{L^\infty} + |q_t| |\tau_t| \|\nabla^2 \Delta^{-1} \psi_t\|_{L^\infty} + |q_t||\tau_t||z_t| \|\nabla^2 u_t\|_{L^\infty} + |q_t| |y_t|^2 |\tau_t| \|\nabla^3 u\|_{L^\infty}
    \\&\qquad+2|q_t||y_t| |\tau_t|\|\nabla^3 \Delta^{-1} \varphi_t\|_{L^\infty}+ 2 |q_t||p_t||y_t| \|\nabla^2 u_t\|_{L^\infty} + 2 |q_t| |p_t| \|\nabla^2 \Delta^{-1} \varphi_t\|_{L^\infty}
    \\&\leq |q_t|^2 \big(C\|\omega_t\|_{H^{5/4}} +1\big) + C\|\psi_t\|_{H^2}^4 + |\tau_t|^6 + C|z_t|^6  + C|y_t|^{12} + C \|\omega_t\|_{H^4}^6\\
    &\qquad + C\|\varphi_t\|_{H^3}^6  + C |p_t|^6 +C,
  \end{align*}
  so we conclude by Gr\"onwall.

  Finally, for~\eqref{eq:C-bound}, we compute using the equation for $C_t,$
  \begin{align*}
    \frac{\dif}{\dif t} \frac{1}{2} |C_t|^2 &\leq |C_t|^2 \|\nabla u\|_{L^\infty} + |C_t||A_t| \|\nabla^2 \Delta^{-1} \psi_t\|_{L^\infty} + |C_t||A_t| |z_t| \|\nabla^2 u_t\|_{L^\infty} + |C_t| |A_t| |y_t|^2 \|\nabla^3 u\|_{L^\infty}
    \\&\qquad+ 2|C_t||B_t| \|\nabla^2 \Delta^{-1} \varphi_t\|_{L^\infty} + 2 |C_t||A_t| |y_t| \|\nabla^3 \Delta^{-1} \varphi_t\|_{L^\infty} + 2|C_t| |B_t| |y_t| \|\nabla^2 u\|_{L^\infty}
    \\&\leq |C_t|^2 \big( C \|\omega_t\|_{H^{5/4}} + 1\big) + C \|\psi_t\|_{H^2}^4+ C |A_t|^6 + C |B_t|^6+ C |z_t|^6
    \\&\qquad+C|y_t|^{12}+ C \|\omega_t\|_{H^4}^6  + C \|\varphi_t\|_{H^3}^4 +C,
  \end{align*}
  so we conclude by Gr\"onwall.
\end{proof}

Our goal now is to conclude the bounds of all the bounds of Assumption~\ref{asmp:dynamic-bounds} except~\eqref{eq:J-smoothing-bound-asmp}, which are directly implied by the following bounds, where $C(\tau_0,A_0,q,\eta,t)>0$ is locally bounded in $(\tau_0,A_0)$. In particular, \eqref{eq:theta-moment-bound-asmp} follows from~\eqref{eq:theta-moment-bound}:
\begin{align}
    \E \sup_{0 \leq r \leq t}  \sup_{\|\gamma\|_{H^4} = 1} |\nabla^\perp \Delta^{-1} \gamma(x_r)|^q +  |\tau_t \cdot \nabla \nabla^\perp \Delta^{-1}\gamma(x_t)|^q\qquad &
 \notag\\+| A_t\nabla \nabla^\perp \Delta^{-1} \gamma(x_t))|^q&\leq C\exp(\eta V^4(\omega_0))
   \label{eq:theta-moment-bound}
\end{align}
 \eqref{eq:J-moment-bound-asmp} follows from \eqref{eq:phi-moment-bound-H-4},~\eqref{eq:y-moment-bound}, and~\eqref{eq:p-B-moment-bound}:
\begin{align}
     \E\sup_{0\leq s \leq r \leq t} \sup_{\|\varphi_s\|_{L^2} = 1}  \|\varphi_r\|_{L^2}^q &\leq C \exp(\eta V^4(\omega_0))
   \label{eq:phi-moment-bound-L-2}\\
  \E\sup_{0\leq s \leq r \leq t}\sup_{\|\varphi_s\|_{H^4} = 1}  \|\varphi_r\|_{H^4}^q &\leq C \exp(\eta V^4(\omega_0))
   \label{eq:phi-moment-bound-H-4}\\
   \E\sup_{0\leq s \leq r \leq t}\sup_{\|\varphi_s\|_{H^4} + |y_s| = 1}  |y_r|^q &\leq C \exp(\eta V^4(\omega_0))
   \label{eq:y-moment-bound}\\
     \E \sup_{0\leq s \leq r \leq t}\sup_{\|\varphi_s\|_{H^4} + |y_s| + |p_s| + |B_s| = 1}  |B_r|^q + |p_r|^q &\leq C \exp(\eta V^4(\omega_0))
   \label{eq:p-B-moment-bound}
\end{align}
   \eqref{eq:J2-moment-bound-asmp} follows from~\eqref{eq:psi-moment-bound-H-4}, \eqref{eq:z-moment-bound}, and~\eqref{eq:q-C-moment-bound}:
\begin{align}
  \E\sup_{0\leq s \leq r \leq t}\sup_{\|\varphi_s\|_{L^2} = 1}  \|\psi_r\|_{L^2}^q &\leq C \exp(\eta V^4(\omega_0))
   \label{eq:psi-moment-bound-L-2}\\
  \E\sup_{0\leq s \leq r \leq t}\sup_{\|\varphi_s\|_{H^4} = 1}  \|\psi_r\|_{H^4}^q &\leq C \exp(\eta V^4(\omega_0))
   \label{eq:psi-moment-bound-H-4}\\
  \E \sup_{0\leq s \leq r \leq t}\sup_{\|\varphi_s\|_{H^4} + |y_s| = 1} |z_r|^q &\leq C \exp(\eta V^4(\omega_0))
   \label{eq:z-moment-bound}\\
  \E \sup_{0\leq s \leq r \leq t}\sup_{\|\varphi_s\|_{H^4} + |y_s| + |p_s| + |B_s| = 1} |C_r|^q + |q_r|^q &\leq C \exp(\eta V^4(\omega_0))
   \label{eq:q-C-moment-bound}
   \end{align}
Finally, \eqref{eq:L-moment-bound-asmp} follows from~\eqref{eq:L-moment-bound-vorticity},~\eqref{eq:L-moment-bound-torus},~\eqref{eq:L-moment-bound-tangent}, and~\eqref{eq:L-moment-bound-matrix}:
   \begin{align}
  \E \sup_{t/2 \leq r \leq t} \sup_{\|\gamma\|_{H^6} = 1} \|\nu\Delta \gamma - \nabla^\perp \Delta^{-1} \gamma \cdot \nabla \omega_r - \nabla^\perp \Delta^{-1} \omega_r \cdot \nabla \gamma\|_{H^4} &\leq C  \exp(\eta V^4(\omega_0))
\label{eq:L-moment-bound-vorticity}\\
  \E \sup_{0 \leq r \leq t} \sup_{\|\gamma\|_{H^4} + |a| + |b| + |D|= 1} |y_t \cdot \nabla u_t(x_t) + \nabla^\perp \Delta^{-1}\varphi(x_t)|^q &\leq C \exp(\eta V^4(\omega_0))
   \label{eq:L-moment-bound-torus}\\
  \E \sup_{0 \leq r \leq t} \sup_{\|\gamma\|_{H^4} + |a| + |b|= 1} \Big|b \cdot \nabla u_t(x_t) +
  + \tau_t \otimes a : \nabla^2 u_t(x_t) \qquad&
     \notag\\ + \tau_t \cdot \nabla \nabla^\perp \Delta^{-1} \gamma(x_t)\Big|^q  & \leq C \exp(\eta V^4(\omega_0))
   \label{eq:L-moment-bound-tangent}\\
  \E \sup_{0 \leq r \leq t} \sup_{\|\gamma\|_{H^4} + |a|+ |D|= 1} \Big|D \nabla u_t(x_t) + A_t a \cdot \nabla^2 u_t(x_t)\qquad&
   \notag\\+ A_t \nabla \nabla^\perp \Delta^{-1} \gamma(x_t)\Big|^q &\leq  C \exp(\eta V^4(\omega_0)).
   \label{eq:L-moment-bound-matrix}
\end{align}

Before we show these bounds, we note additionally
\begin{align*}\exp\Big(C \int_s^t \|\omega_r\|_{H^{5/4}}\,\dif r\Big) &\leq \exp\Big(C \int_s^t \|\omega_r\|_{H^1}^{3/4} \|\omega_r\|_{H^2}^{1/4}\,\dif r\Big)
  \\&\leq \exp\Big(C\Big(\int_s^t \|\omega_r\|_{H^1}^{2}\,\dif r\Big)^{3/4} + C \Big(\int_s^t \|\omega_r\|_{H^2}^{2}\Big)^{1/4}\Big).
\end{align*}
Thus by~\eqref{eq:l2-energy-bound} and~\eqref{eq:Hn-bound} with $n=1$,
\begin{equation}
  \label{eq:exp-omega-H-5/4-moment-bound}
  \E \Big(\exp\Big(C \int_s^t \|\omega_r\|_{H^{5/4}}\,\dif r\Big)\Big)^q \leq C \exp(c\eta V^1(\omega_0))\leq C \exp(\eta V^4(\omega_0)),
\end{equation}
where we use Lemma~\ref{lem:Vn-poincare} for the final inequality.

We now sketch how to prove the bounds (\ref{eq:theta-moment-bound}-\ref{eq:L-moment-bound-matrix}). For~\eqref{eq:theta-moment-bound}, we use Sobolev embeddings together Lemma~\ref{lem:A-bound} and~\eqref{eq:exp-omega-H-5/4-moment-bound}. For \eqref{eq:phi-moment-bound-L-2}, we use Lemma~\ref{lem:phi-l2-bound} with \eqref{eq:l2-energy-bound}. For \eqref{eq:phi-moment-bound-H-4}, we use Lemma~\ref{lem:phi-Hn-bound} together with \eqref{eq:Hn-bound}---both with $n=4$---together with \eqref{eq:phi-moment-bound-L-2}. For \eqref{eq:psi-moment-bound-L-2}, we use Lemma~\ref{lem:psi-l2-bound} with \eqref{eq:l2-energy-bound} and \eqref{eq:phi-moment-bound-H-4}. For \eqref{eq:psi-moment-bound-H-4}, we use Lemma~\ref{lem:psi-Hn-bound} with \eqref{eq:Hn-bound}, \eqref{eq:psi-moment-bound-L-2}, and \eqref{eq:phi-moment-bound-H-4}. For \eqref{eq:y-moment-bound}, we use \eqref{eq:y-bound}, \eqref{eq:exp-omega-H-5/4-moment-bound}, and \eqref{eq:phi-moment-bound-H-4}. For \eqref{eq:z-moment-bound}, we use \eqref{eq:z-bound} with \eqref{eq:exp-omega-H-5/4-moment-bound}, \eqref{eq:y-moment-bound}, \eqref{eq:Hn-bound}, \eqref{eq:phi-moment-bound-H-4}, and \eqref{eq:psi-moment-bound-H-4}. For \eqref{eq:p-B-moment-bound}, we use \eqref{eq:p-bound} and \eqref{eq:B-bound}, together with \eqref{eq:exp-omega-H-5/4-moment-bound}, \eqref{eq:phi-moment-bound-H-4}, \eqref{eq:Hn-bound}, \eqref{eq:y-moment-bound}, and Lemma~\ref{lem:A-bound}. For \eqref{eq:q-C-moment-bound}, we use \eqref{eq:q-bound} and \eqref{eq:C-bound}, together with \eqref{eq:exp-omega-H-5/4-moment-bound}, \eqref{eq:psi-moment-bound-H-4}, \eqref{eq:z-moment-bound}, \eqref{eq:y-moment-bound}, \eqref{eq:Hn-bound}, \eqref{eq:p-B-moment-bound}, and Lemma~\ref{lem:A-bound}. For~\eqref{eq:L-moment-bound-vorticity}, we use~\eqref{eq:Hn-regularization}. For~\eqref{eq:L-moment-bound-torus}, we use~\eqref{eq:Hn-bound},~\eqref{eq:y-moment-bound}, and~\eqref{eq:phi-moment-bound-H-4}. For~\eqref{eq:L-moment-bound-tangent}, we use~\eqref{eq:Hn-bound}, Lemma~\ref{lem:A-bound}, and~\eqref{eq:exp-omega-H-5/4-moment-bound}. Finally, for~\eqref{eq:L-moment-bound-matrix}, we use~\eqref{eq:Hn-bound} and Lemma~\ref{lem:A-bound}.

Finally, in order to conclude, we show~\eqref{eq:J-smoothing-bound-asmp}. We now suppose that $(\varphi_t)_{T/2 \leq t \leq T}$ is such that
\[ \dot \varphi_t = \nu\Delta \varphi_t - \nabla^\perp \Delta^{-1} \varphi_t \cdot \nabla \omega_t - \nabla^\perp \Delta^{-1} \omega_t \cdot \nabla \varphi_t.\]
Given the bounds shown above, in order to conclude~\eqref{eq:J-smoothing-bound-asmp}, it suffices to prove the following.
\begin{lemma}
    \label{lem:J-smoothing-bound}
    For all $\eta, q,T>0, n\in \N$ there exists $C(\eta,q,T,n)>0$ such that
    \[\E \sup_{\|\varphi_{T/2}\|_{L^2} \leq 1} \|\varphi_{3T/4}\|_{H^{n+1}}^q \leq C \exp(\eta V^n(\omega_0)).\]
\end{lemma}

First, we prove the following.
\begin{lemma}
    \label{lem:one-step-J-smoothing}
    For all $n\geq 0, 0 \leq s \leq t \leq T$, there exists $C(n, t-s, T)>0$ such that
    \[\|\varphi_t\|_{H^{n+1}} \leq C \|\varphi_s\|_{H^n} + C \sup_{s \leq r \leq t} \big(\|\omega_r\|_{H^{n+1}}^{2n+4} + \|\varphi_r\|_{L^2}^2 + 1\big).\]
\end{lemma}

\begin{proof}
    By~\eqref{eq:linearization-Hn-bound},
    \[ \frac{\nu}{2}\int_s^t \|\varphi_r\|_{H^{n+1}}\,\dif r\leq  \|\varphi_s\|_{H^n} + C\sup_{s \leq r \leq t}\big( \|\omega_r\|_{H^n}^{2n+2} +  \|\varphi_r\|_{L^2}^{2}\big).\]
    Thus there exists $t_0 \in [s, (t+s)/2]$ such that
    \[\|\varphi_{t_0}\|_{H^{n+1}} \leq   C\|\varphi_s\|_{H^n} + C\sup_{s \leq r \leq t}\big( \|\omega_r\|_{H^n}^{2n+2} +  \|\varphi_r\|_{L^2}^{2}\big).\]
    Applying again~\eqref{eq:linearization-Hn-bound} now for $n+1$ and on $[t_0,t],$ we see that
    \[\|\varphi_t\|_{H^{n+1}} \leq  \|\varphi_{t_0}\|_{H^{n+1}} + C\sup_{t_0 \leq r \leq t}\big( \|\omega_r\|_{H^{n+1}}^{2n+4} +  \|\varphi_r\|_{L^2}^{2}\big).\]
    Combining the two displays, we conclude.
\end{proof}

\begin{proof}[Proof of Lemma~\ref{lem:J-smoothing-bound}]
    The result is direct from iterating Lemma~\ref{lem:one-step-J-smoothing} and using the bounds~\eqref{eq:phi-moment-bound-L-2} and~\eqref{eq:Hn-regularization} to control the remaining terms.
\end{proof}

\subsection{Nondegeneracy of the vector fields for the projective and matrix processes}
\label{ssa:nondegeneracy}

Our goal for this subsection is to prove the second half of Proposition~\ref{prop:manifold-processes-good} that says that the one-point, two-point, tangent, projective, and Jacobian processes satisfy Assumption~\ref{asmp:nondegen}. Together with the arguments of Subsection~\ref{ssa:linearization-bounds}, this concludes the proof of Proposition~\ref{prop:manifold-processes-good}. Assumption~\ref{asmp:nondegen} for the one-point process follows from Assumption~\ref{asmp:nondegen} for any of the other four processes. Assumption~\ref{asmp:nondegen} for the projective process also follows from Assumption~\ref{asmp:nondegen} (alternatively the proof follows exactly analogously). Thus, we only focus on the two-point, tangent, and Jacobian processes. Unlike in Subsection~\ref{ssa:linearization-bounds}, these do have to be treated separately rather than as part of one larger process as the $(\omega_t, x_t,y_t,\tau_t,A_t)$ process \textit{does not satisfy} Assumption~\ref{asmp:nondegen}. Lemma~\ref{lem:two-point-nondegen}, Lemma~\ref{lem:tangent-nondegen}, and Lemma~\ref{lem:jacobian-nondegen} together with the above discussion show Assumption~\ref{asmp:nondegen} for the processes under consideration and thus show the second half of Proposition~\ref{prop:manifold-processes-good}. Combining this with Subsection~\ref{ssa:linearization-bounds}, we conclude the proof of Proposition~\ref{prop:manifold-processes-good}.

\subsubsection{The two-point process}

\begin{lemma}
\label{lem:two-point-nondegen}
    There exists $C,R>0$ such that for all $x',y' \in \T^2\times \T^2$ and all $a,b \in \R^2\times \R^2$,
    \[|a| + |b| \leq C |x'-y'|^{-1} \max_{|k| \leq R} |\<\Theta_{e_k}^2(x',y'),(a,b)\>|.\]
\end{lemma}

\begin{proof}
    Without loss of generality, we suppose that $|x'_1 - y'_1| \geq |x_2'-y_2'|$, $|a|\leq |b|$, $x'=0$.

    Let $S \subseteq H^5(\T^2)$ be defined by
    \[S := \linspan\{\nabla^\perp \Delta^{-1} e_k : |k| \leq R\},\]
    for some $R \geq 2$ to be determined.
    We then suppose,
    \[ \max_{|k| \leq R} |\<\Theta^2_{e_k}(x',y'), (a,b)\>| \leq 1,\]
    so that for any $u \in S$,
    \begin{equation}
    \label{eq:S-small}
        |\<\Theta^2_u(x',y'), (a,b)\>| \leq C \|u\|_{H^5}.
    \end{equation}
    To conclude, it suffices then to show
    \[|b| \leq C |y'|^{-1}.\]
    For $|y'| \geq \frac{3\pi}{4}$, it is clear one can choose $R$ sufficiently large so that
    \[|b| \leq C \max_{k \leq R} |\<\Theta_{e_k}^2(x',y'),(a,b)\>| \leq  C |y'|^{-1}\max_{k \leq R} |\<\Theta_{e_k}^2(x',y'),(a,b)\>|,\]
    so we only consider $|y'| \leq \frac{3\pi}{4}$.

    Note that
    \begin{align*}
        u_1(x) &:= (\sin(x_2),0)\\
        u_2(x) &:= (0,\sin(x_1))
    \end{align*}
    are such that $u_1,u_2\in S$. Thus using~\eqref{eq:S-small}, we see
    \[| \sin(y'_2) b_1|+|\sin(y_1') b_2|\leq C.\]
    Thus $\frac{3\pi}{4} \geq |y'_1| \geq |y'|/2$, so that
    \[|\sin(y'_1)| \geq C^{-1} |y'|,\]
    giving that
    \[|b_2| \leq C|y'|^{-1}.\]
    Then if $|y'_2| \geq |y'|/10,$ we also see symmetrically that
    \[|b_1| \leq C |y'|^{-1}.\]
    On the other hand, if $|y'_2| \leq |y'|/10,$ defining
    \[u_3(x) := (\sin(x_1 + x_2), - \sin(x_1+x_2)) \in S,\]
     we get that
     \[|\sin(y_1'+ y_2') (b_1 -b_2)| \leq C.\]
    Then since we've already seen that $|b_2| \leq C |y'|^{-1},$ we either have $|b_1| \leq C |y'|^{-1}$ or
    \[|b_1| \leq C |\sin(y_1'+y_2')|^{-1} \leq C |y'|^{-1},\]
    where we use that on the set $\{(y_1,y_2) : |y| \leq \frac{3\pi}{4}, |y_2| \leq  |y|/10\}$, we have that $|\sin(y_1 + y_2)| \geq C^{-1} |y|$. Thus, we have seen that in either case $|b_1| \leq C|y'|^{-1}$, allowing us to conclude for $|y'| \leq \frac{3\pi}{4}$.

    Putting the cases $|y'| \geq \frac{3\pi}{4}$ and $|y'| \leq \frac{3\pi}{4}$, we conclude for all $y'$.
\end{proof}

\subsubsection{The tangent process}

\begin{lemma}
    \label{lem:tangent-nondegen}
    There exists $C>0$ such that for all $(x',\tau') \in \T^2 \times \R^2$ and all $y,p \in  \R^2 \times \R^2$,
    \[|(y,p)| \leq C (1+|\tau'|^{-1})\max_{|k| \leq 2} |\<\Theta^T_{e_k}(x',\tau'), (y,p)\>|.\]
\end{lemma}

\begin{proof}
    We assume without loss of generality $x' = 0$. Let $S \subseteq H^5(\T^2)$ be defined by
    \[S := \linspan\{\nabla^\perp \Delta^{-1} e_k : |k| \leq 2\}\]
    Suppose
    \[ \max_{|k| \leq 1} |\<\Theta^T_{e_k}(x',\tau'), (y,p)\>| \leq 1,\]
    so that for any $u \in S$,
    \begin{equation}
        \label{eq:tangent-overlap-small}
    \<\Theta^T_u(x',\tau'), (y,p)\>| \leq C\|u\|_{H^5}.
    \end{equation}
    Note that
    \begin{align*}
        u_1(x) &:= (\cos(x_2),0)\\
        u_2(x) &:= (0,\cos(x_1))\\
        u_3(x) &:= (0,\sin(x_1))\\
        u_4(x) &:= (\sin(x_2),0)\\
        u_5(x) &:= (\sin(x_1 + x_2), -\sin(x_1+x_2))\\
        u_6(x) &:= (\sin(x_1 - x_2), \sin(x_1-x_2))
    \end{align*}
    are such that $u_i \in S$. Thus by~\eqref{eq:tangent-overlap-small}, we have that
    \[|y_1| + |y_2| + |\tau'_1| |p_2| + |\tau'_2| |p_1| + |(\tau'_1+\tau'_2) (p_1 -p_2)| + |(\tau'_1-\tau'_2)(p_1+p_2)| \leq C.\]
    In particular,
    \[|y| \leq C \leq C(|\tau'|^{-1} + 1).\]
    Then without loss of generality $|\tau'_1| \geq |\tau'|/2$, giving that
    \[|p_2| \leq C |\tau'|^{-1} \leq C (|\tau'|^{-1} + 1).\]
    Then either
    \[|p_1| \leq 2|p_2| \leq C (|\tau'|^{-1} + 1) \text{ or } |p_1| \leq  2|p_1+p_2| \land 2|p_1 - p_2|.\]
    But then we also have that either $|\tau'_1+ \tau'_2| \geq |\tau'|$ or $|\tau'_1 - \tau'_2| \geq |\tau'|$ and so either
    \[|p_1 - p_2| \leq C|\tau'|^{-1} \text{ or } |p_1+ p_2| \leq C |\tau'|^{-1}.\]
    Thus combining all the cases, we see that also
    \[|p_1| \leq C(\tau'|^{-1} + 1).\]
    Combining all the bounds, we conclude.
\end{proof}

\subsubsection{The Jacobian process}

We now consider the Jacobian process. We first want to reduce the problem to understanding the behavior for $A' = I$. First, we use a linear algebra fact.

\begin{lemma}
    \label{lem:inner-product-nondegen-comp}
    Let $V$ a $d$-dimensional vector space with two inner-products $g_1, g_2$. Suppose that $g_1,g_2$ induce norms with the bound
    \[ \sqrt{g_1(v,v)} \leq \Lambda \sqrt{g_2(v,v)}.\]
    Let $v_1,\dots,v_n \in V$ a finite collection of vectors and suppose $K>0$ is such that for all $v \in V$,
    \[\sqrt{g_1(v,v)} \leq K \max_{1 \leq j \leq n} |g_1(v_j,v)|.\]
    Then
    \[\sqrt{g_2(v,v)} \leq \Lambda K \max_{1 \leq j \leq n} |g_2(v_j,v)|.\]
\end{lemma}
\begin{proof}
    By standard theory of quadratic forms, there exists an invertible linear map $A \colon V \to V$ such that for all $v,w \in V$,
    \[g_2(v,w) = g_1(v,Aw).\]
    Further we have that for any $v \in V$ with $g_1(v,v) =1$,
    \[g_1(A^{-1} v,v) \leq \sqrt{g_1(A^{-1}v, A^{-1} v)}\leq \Lambda \sqrt{g_2(A^{-1}v,A^{-1}v)} = \Lambda \sqrt{g_1(A^{-1}v ,v)}.\]
    Thus dividing both sides by $\sqrt{g_1(A^{-1}v,v)}$ and using the homogeneity, we see that for any $v \in V$,
    \[\sqrt{g_1(A^{-1}v,v)} \leq \Lambda \sqrt{g_1(v,v)}.\]
    Then for any $v \in V,$
    \begin{align*}
        \sqrt{g_2(v,v)}&= \sqrt{g_2(A^{-1}Av, v)}
        \\&= \sqrt{g_1(A^{-1} A v, Av)}
        \\&\leq \Lambda\sqrt{g_1(Av,Av)}
        \\&\leq \Lambda K \max_{1 \leq j \leq n} |g_1(v_j,Av)|
        \\&= \Lambda K \max_{1 \leq j \leq n} |g_2(v_j,v)|.\qedhere
    \end{align*}

\end{proof}

\begin{lemma}
    \label{lem:jacobian-nondegen-from-ident}
    Suppose that for some $C,R>0$ with $R \geq 1$, we have for all $x' \in \T^2$ and $(y,B) \in \R^2 \times T_I \SL(2,\R)$ that
    \[|(y,B)| \leq C \max_{|k|\leq R} \big|\<\Theta^J_{e_k}(x', I), (y,B)\>\big|.\]
    Then for all $(x',A') \in \T^2 \times \SL(2,\R)$ and $(y,B) \in \R^2 \times T_{A'} \SL(2,\R)$ we have that
    \[|(y,B)| \leq C |A'|\max_{|k|\leq R} \big|\<\Theta^J_{e_k}(x', A'), (y,B)\>\big|.\]
\end{lemma}

\begin{proof}
    We without loss of generality always take $x' = 0$.

    Suppose $C,R>0$, $R \geq 1,$ are such that for all $(y,B) \in \R^2 \times T_I \SL(2,\R),$
    \begin{equation}
        \label{eq:control-at-ident}
    |(y,B)| \leq C \max_{|k|\leq R} \big|\<\Theta^J_{e_k}(0, I), (y,B)\>\big|.
    \end{equation}

    Let $A' \in \SL(2,\R)$ and $(y,B) \in \R^2 \times T_{A'} \SL(2,\R)$. We then introduce an additional inner product on $\R^2\times T_{A'} \SL(2,\R)$. Let
    \[g_\ell((z,C), (y,B)) := z \cdot y + \<(A')^{-1} C, (A')^{-1}B\> = z \cdot y + \text{Tr}(C^T (A')^{-T} A^{-1}B).\]
    Then we note that $(y,(A')^{-1} B) \in T_I \SL(2,\R)$, so
    \begin{align*}
      \sqrt{g_\ell((y,B),(y,B))} &= \|(y, (A')^{-1} B)\|\\
                              &\leq C \max_{|k| \leq R} |\<\Theta_{e_k}^J(0,I), (y,(A')^{-1}B)\>|\\
                              &= C \max_{|k|\leq R} |g_\ell(\Theta_{e_k}^J(0,A'), (y,B))|.
    \end{align*}
    Note also that
    \[g_\ell((z,M),(z,M)) = |z|^2 + \text{Tr}( (A')^{-T} (A')^{-1} MM^T) \leq |z|^2 + C|A'|^2|M|^2 \leq C|A'|^2 \<(z,M),(z,M)\>.\]
    Thus by Lemma~\ref{lem:inner-product-nondegen-comp}, we have that
    \[|(y,B)| \leq C |A'|\max_{|k|\leq R} \big|\<\Theta^J_{e_k}(0, A'), (y,B)\>\big|.\qedhere\]
\end{proof}

\begin{lemma}
    \label{lem:jacobian-nondegen}
     Then for all $(x',A') \in \T^2 \times \SL(2,\R)$ and $(y,B) \in \R^2 \times T_{A'} \SL(2,\R)$ we have that
    \[|(y,B)| \leq C |A'|\max_{|k|\leq R} \big|\<\Theta^J_{e_k}(x', A'), (y,B)\>\big|.\]
\end{lemma}
\begin{proof}
    By Lemma~\ref{lem:jacobian-nondegen-from-ident}, it suffices to consider $A' = I$. We also without loss of generality take $x'=0$. We take $(y, B) \in \R^2 \times T_I \SL(2,\R)$ and $R\geq 1$ to be fixed later. We suppose that
    \begin{equation*}
    \max_{|k|\leq R} \big|\<\Theta^J_{e_k}(0, I), (y,B)\>\big| \leq 1.
    \end{equation*}
    To conclude, we need to show that
    \[|y| + |B| \leq C.\]
    Verifying $|y|\leq 2 \leq C$ follows exactly as in Lemma~\ref{lem:two-point-nondegen} and Lemma~\ref{lem:tangent-nondegen}, so we only consider showing $|B| \leq C$. Note that $\text{Tr}(B) = 0$ from $B \in T_I \SL(2,\R),$ thus we can write
    \[B = \begin{pmatrix} a & b \\ c & -a \end{pmatrix}.\]
    Let $S \subseteq H^5(\T^2)$ be defined by
    \[S := \linspan\{\nabla^\perp \Delta^{-1} e_k : |k| \leq R\},\]
    so that for $u \in S,$
    \begin{equation}
        \label{eq:jacobian-overlap-small}
    \big|\<\Theta^J_{u}(0, I), (y,B)\>\big| \leq C \|u\|_{H^5}.
    \end{equation}
    Then note that
    \begin{align*}
        u_1(x) &:= (\sin(x_2),0)\\
        u_2(x) &:= (0,\sin(x_1))\\
        u_3(x)&:= (\sin(x_1+x_2),-\sin(x_1+x_2))
    \end{align*}
    are all such that $u_j \in S$. Thus from~\eqref{eq:jacobian-overlap-small}, we have that
    \[|c| + |b| + |2a-b +c| \leq C,\]
    which immediately implies that $|B| \leq C,$ allowing us to conclude.
\end{proof}

\subsection{Approximating the two-point process by the tangent process}
\label{ssa:two-point-by-tangent}

The goal of this subsection is to prove Lemma~\ref{lem:two-point-linearization-by-tangent}. Rephrasing the lemma, we consider an arbitrary $(\omega_0, x_0, y_0) \in H^4(\T^2) \times M^2$. We then let
\[\tau_0 := \frac{y_0 - x_0}{|y_0-x_0|}\]
and suppose $\omega_t$ solves~\eqref{eq:sns} and
\begin{equation*}
\begin{cases}
    \dot x_t = u_t(x_t),\\
    \dot y_t = u_t(y_t),\\
    \dot \tau_t = \tau_t \cdot \nabla u_t(x_t),
    \end{cases}
\end{equation*}
where $u_t$ and $\omega_t$ are related in the usual way. Then we consider an arbitrary random time $0 \leq s \leq 1$ and random $(\varphi_s, p_s, q_s) \in H^4(\T^2) \times \R^4$ such that
\[\|\varphi_s\| + |p_s| + |q_s| \leq 1.\]
We then let $a_s := p_s + q_s |y_0-x_0|$ and let $(\varphi_t, p_t,q_t,a_t)_{t \geq s}$ solve
  \begin{equation*}
  \begin{cases}
   \dot \varphi_t = \nu\Delta \varphi_t - \nabla^\perp \Delta^{-1} \varphi_t \cdot \nabla \omega_t - \nabla^\perp \Delta^{-1} \omega_t \cdot \nabla \varphi_t\\
    \dot p_t = p_t \cdot \nabla u_t(x_t) + \nabla^\perp \Delta^{-1} \varphi_t(x_t)\\
    \dot q_t = q_t \cdot \nabla u_t(x_t) + \tau_t \otimes p_t : \nabla^2 u_t(x_t) + \tau_t \cdot \nabla \nabla^\perp \Delta^{-1} \varphi_t(x_t)\\
    \dot a_t = a_t \cdot\nabla u_t(y_t) + \nabla^\perp \Delta^{-1} \varphi_t(y_t).
\end{cases}
    \end{equation*}
We then let
\[r_t := \frac{a_t - p_t}{|y_0-x_0|}.\]
We note that then, in the language of Lemma~\ref{lem:two-point-linearization-by-tangent} that
\[(\varphi_t,p_t, q_t) = J^T_{s,t} (\varphi_s, p_s,q_s) \text{ and } (\varphi_t, p_t, r_t) = B_{x_0,y_0}^{-1} J_{s,t} B_{x_0,y_0} (\varphi_s, p_s,q_s).\]
Thus to conclude Lemma~\ref{lem:two-point-linearization-by-tangent}, it suffices for any $q,\eta >0$, there exists $C(q,\eta)$ such that
\[\E \sup_{s \leq t \leq 1} |q_t - r_t|^q \leq C |y_0-x_0|^q e^{\eta V(\omega_0)}.\]

First, we need to control the difference between\footnote{We note that $y_t -x_t$ is only defined mod $2\pi$. At time $s$, we take $y_s - x_s \in \R^2$ to be the minimal length representative of the equivalence class. For future times we take the unique lift of $y_t-x_t$ to $\R^2$ that makes the path $t \mapsto y_t - x_t$ continuous.}
\[\sigma_t := \frac{y_t -x_t}{|y_0-x_0|}\]
and $\tau_t$.

\begin{lemma}
    For any $\eta, q>0$, there exists $C(\eta,q)>0$ such that
    \begin{equation}
    \label{eq:sigma-tau-difference}
    \E \sup_{0 \leq t \leq 1} |\tau_t - \sigma_t|^q \leq C |y_0-x_0|^q e^{\eta V(\omega_0)},
    \end{equation}
    and
    \begin{equation}
    \label{eq:x-y-dont-separate}
    \E \sup_{0 \leq t \leq 1} \Big(\frac{|y_t-x_t|}{|y_0-x_0|} + \frac{|y_0 - x_0|}{|y_t-x_t|}\Big)^q \leq C e^{\eta V(\omega_0)}.
    \end{equation}
\end{lemma}

Note that~\eqref{eq:x-y-dont-separate} implies Lemma~\ref{lem:x-y-dont-separate}.

\begin{proof}
    We note that by construction $\sigma_0 = \tau_0$ and that $\sigma_t$ solves
    \[\dot \sigma_t = \frac{1}{|y_0 - x_0|} \big(u_t(y_t) - u_t(x_t)\big) = \sigma_t \cdot \nabla u_t(x_t) + \frac{1}{|y_0 - x_0|} \big(u_t(y_t) - u_t(x_t) - (y_t-x_t) \cdot \nabla u_t(x_t)\big).\]
    Then using the first expression for $\dot \sigma_t$ and the Lipschitz bound on $u_t,$
\[\frac{\dif }{\dif t} |\sigma_t|^2 \leq 2 \|\nabla u_t\|_{L^\infty} |\sigma_t|^2,\]
    so that
    \[|\sigma_t| \leq \exp\Big(\int_0^t \|\nabla u_r\|_{L^\infty}\dif r\Big) \leq \exp\Big(C\int_0^t \|\omega_r\|_{H^{5/4}}\dif r\Big).\]
    Using~\eqref{eq:exp-omega-H-5/4-moment-bound} we conclude one half of~\eqref{eq:x-y-dont-separate}. The other half of~\eqref{eq:x-y-dont-separate} follows from running the ODE for $\sigma_t$ backward in time and using the same Gr\"onwall bound, which gives that
    \[1=|\sigma_0| \leq |\sigma_t|\exp\Big(C\int_0^t \|\omega_r\|_{H^{5/4}}\dif r\Big),\]
    which then again by~\eqref{eq:exp-omega-H-5/4-moment-bound} gives the second half of~\eqref{eq:x-y-dont-separate}.

    For~\eqref{eq:sigma-tau-difference}, using the other expression for the $\dot \sigma_t$, we see that
    \begin{align*}
    \frac{\dif }{\dif t} \frac{1}{2}|\sigma_t - \tau_t|^2& =  |\sigma_t - \tau_t|^2 + (\sigma_t -\tau_t) \cdot  \frac{1}{|y_0 - x_0|} \big(u_t(y_t) - u_t(x_t) - (y_t-x_t) \cdot \nabla u_t(x_t)\big)
    \\&\leq |\sigma_t -\tau_t|^2 + |\sigma_t-\tau_t| |\sigma_t|^2 |y_0-x_0| \|\nabla^2u_t\|_{L^\infty}
    \\&\leq 2 |\sigma_t - \tau_t|^2 + \frac{1}{2}|\sigma_t|^4 |y_0-x_0|^2 \|\nabla^2 u_t\|_{L^\infty}^2,
    \end{align*}
    thus
    \[|\sigma_t - \tau_t| \leq e^{4t} |y_0 - x_0|\sup_{0 \leq r \leq t} |\sigma_r|^2 \|\nabla^2 u_r\|_{L^\infty} \leq e^{4t} |y_0-x_0| \exp\Big(C\int_0^t \|\omega_r\|_{H^{5/4}} \dif r\Big) \sup_{0 \leq r \leq t} \|\nabla^2 u_r\|_{L^\infty}.\]
    Using then~\eqref{eq:exp-omega-H-5/4-moment-bound} and~\eqref{eq:Hn-bound}, we conclude~\eqref{eq:sigma-tau-difference}.
\end{proof}

Before proceeding, we write the equation for $r_t$ in a suggestive way.
    \begin{align*}
    \dot r_t &= r_t \cdot \nabla u_t(y_t) + \tau_t\otimes p_t : \nabla^2 u_t(x_t) + \tau_t \cdot\nabla \nabla^\perp \Delta^{-1} \varphi_t(x_t)
    \\&\qquad+ (\sigma_t-\tau_t)\otimes p_t : \nabla^2 u_t(x_t) + (\sigma_t-\tau_t) \cdot\nabla \nabla^\perp \Delta^{-1} \varphi_t(x_t)
    \\&\qquad+\frac{1}{|y_0 -x_0|} \Big( p_t \cdot \big(\nabla u_t(y_t) - \nabla u_t(x_t) - (y_t -x_t) \cdot \nabla^2 u_t(x_t)\big)\Big)
    \\&\qquad+\frac{1}{|y_0-x_0|}\Big(\nabla^\perp \Delta^{-1} \varphi_t(y_t) - \nabla^\perp \Delta^{-1} \varphi_t(x_t) - (y_t -x_t) \cdot \nabla \nabla^\perp \Delta^{-1} \varphi_t\Big)
    \\&=:  r_t \cdot \nabla u_t(y_t) + \tau_t\otimes p_t : \nabla^2 u_t(x_t) + \tau_t \cdot\nabla \nabla^\perp \Delta^{-1} \varphi_t(x_t) + \gamma_t.
    \end{align*}
We note that
 \begin{equation}
     \label{eq:gamma-bound}
 |\gamma_t| \leq |\sigma_t - \tau_t| \big(|p_t| \|\nabla^2 u_t\|_{L^\infty} + \|\nabla^2 \Delta^{-1} \varphi_t\|_{L^\infty}\big)  + |y_0 - x_0| |\sigma_t|^2 \big(|p_t| \|\nabla^3 u_t\|_{L^\infty} + \|\nabla^3 \Delta^{-1} \varphi_t\|_{L^\infty}\big).
 \end{equation}
 Then by~\eqref{eq:Hn-bound},~\eqref{eq:y-moment-bound},~\eqref{eq:phi-moment-bound-H-4},~\eqref{eq:sigma-tau-difference}, and~\eqref{eq:x-y-dont-separate}, we see that for any $q,\eta>0$, there exists $C(q,\eta)$ such that
 \[\E \sup_{s \leq t \leq 1} |\gamma_t|^q \leq C|y_0-x_0|^q e^{\eta V(\omega_0)}.\]

We now use this to bound $|r_t|$, noting that
\[\frac{\dif }{\dif t} \frac{1}{2} |r_t|^2 \leq |r_t|^2 (\|\nabla u\|_{L^\infty}+1) + C|\tau_t|^2 |p_t|^2 \|\nabla^2 u_t(x_t)\|_{L^\infty}^2 + C|\tau_t|^2 \|\nabla^2 \Delta^{-1}\varphi_t\|_{L^\infty}^2 + C|\gamma_t|^2,\]
so that, using that $|r_s| = |q_s| \leq 1,$
\[|r_t| \leq C\exp\Big(C\int_s^t \|\omega_r\|_{H^{5/4}}\,\dif r\Big) \big(1 + \sup_{s \leq r \leq t} \big(|\tau_r| |p_r| \|\nabla^2 u_r(x_t)\|_{L^\infty} + |\tau_r| \|\nabla^2 \Delta^{-1}\varphi_r\|_{L^\infty} +|\gamma_r|\big)\big),\]
and so by~\eqref{eq:Hn-bound},~\eqref{eq:exp-omega-H-5/4-moment-bound},~\eqref{eq:y-moment-bound},~\eqref{eq:gamma-bound}, and Lemma~\ref{lem:A-bound}, we conclude that for any $q, \eta >0$, there exists $C(\eta,q)>0$ such that
\begin{equation}
    \label{eq:r-bound}
    \E\sup_{s \leq t \leq 1} |r_t|^q \leq Ce^{\eta V(\omega_0)}.
\end{equation}

We are finally ready to control the difference $|q_t - r_t|$. Computing directly using the equation for $q_t$ and $r_t$, we see
\begin{align*}
    \frac{\dif }{\dif t} |r_t- q_t|^2 =  2 (r_t - q_t) \cdot \nabla u_t(x_t) \cdot (r_t - q_t) +  2 (r_t - q_t) \cdot \gamma_t+2 r_t \cdot (\nabla u_t(y_t) - \nabla u_t(x_t)) \cdot (r_t - q_t),
\end{align*}
thus
\[ \frac{\dif }{\dif t} |r_t- q_t|^2 \leq 2|r_t -q_t|^2 (\|\nabla u\|_{L^\infty}+1) + C |\gamma_t|^2 + C |y_0-x_0|^2|r_t|^2  |\sigma_t|^2 \|\nabla^2 u_t\|_{L^\infty}^2,\]
so by Gr\"onwall, using that $r_s = q_s,$
\[|r_t - q_t| \leq C \exp\Big(C\int_s^t \|\omega_r\|_{H^{5/4}}\,\dif r\Big) \sup_{s \leq r \leq t} \big(|\gamma_t| +  |y_0-x_0||r_t|  |\sigma_t| \|\nabla^2 u_t\|_{L^\infty} \big).\]
Using then~\eqref{eq:exp-omega-H-5/4-moment-bound},~\eqref{eq:gamma-bound}, \eqref{eq:r-bound}, \eqref{eq:x-y-dont-separate}, and \eqref{eq:Hn-bound}, we conclude that for any $q,\eta>0$, there exists $C(q,\eta)>0$ such that
\[\E \sup_{s \leq t \leq 1} |r_t -q_t|^q \leq C |y_0-x_0|^q e^{\eta V(\omega_0)}.\]
By the discussion above, this concludes the proof of Lemma~\ref{lem:two-point-linearization-by-tangent}.

\section{Approximate controllability}
The proofs in this section are elementary but involve lengthy direct computations, which are visually clear. We refer to Figure~\ref{fig:proj-control} and Figure~\ref{fig:2pt-control} for a visual explanation.
Let $M$ be a manifold corresponding to one of the special processes in Table~\ref{table:special-processes}, viewed as an embedded manifold of $\R^N$ and equipped with corresponding Euclidean distance.

\begin{definition}\label{defn:reach}
  Given $E_1,E_2 \subseteq \uspace \times M$, we write $\Reach(E_1, E_2)$ to denote the statement that for every $\varepsilon > 0$, there is $T > 0$ and $\delta > 0$ such that for every $(u_0, p_0) \in E_1$, we have
  \begin{equation}\label{eq:reach-defn}
    \P[\exists (\tilde{u}, \tilde{p}) \in E_2 \; \text{s.t.} \; \|u_T-\tilde{u}\|_{\uspace} + |p_T-\tilde{p}| < \varepsilon] > \delta.
  \end{equation}
  If $(\tilde{u}, \tilde{p}) \in \uspace \times M$, we write $\Reach((u_0, p_0), (\tilde{u}, \tilde{p}))$ to mean $\Reach(\{(u_0, p_0)\}, \{(\tilde{u}, \tilde{p})\})$.
\end{definition}

\begin{lemma}\label{lem:reach-control}
  $\Reach(E_1, E_2)$ is satisfied if, for every $\varepsilon > 0$, there exists a time $T > 0$ such that for every $(u_0, p_0) \in E_1$ there is a Lipschitz control $g \colon F \times \R_{\geq 0} \to \R$ with $|\dot{g}^k| \leq 1$ such that the solution $(u_t^g, p_t^g)$ to the controlled equation
\begin{equation*}
  \begin{cases}
    \dot{u}_t^g = \nu \Delta u_t^g - u_t^g \cdot \nabla u_t^g + \sum_{k \in F} c_k \nabla^\perp \Delta^{-1} e_k(x)\dot{g}^k(t),\\
    \nabla \cdot u_t^g = 0,\\
    \dot{p}_t = \Phi_{u}(p_t)
  \end{cases}
\end{equation*}
satisfies $\|u_T - \tilde{u}\|_{\uspace} + |p_T - \tilde{p}| < \varepsilon$ for some $(\tilde{u}, \tilde{p}) \in E_2$.

Furthermore, if $K \subseteq M$ is compact and $\Reach(E_1, \{u\} \times K)$ and $\Reach((u, y), E_2)$ for all $y \in K$, then $\Reach(E_1, E_2)$ holds. In particular, $\Reach(\cdot, \cdot)$ is transitive on $\uspace \times M$.
\end{lemma}
\begin{proof}
  Immediate from positivity of Wiener measure and stability of the equation.
\end{proof}

\begin{definition}
  Given $k \in F$ and $t \in \R$, we define the diffeomorphism
  \[
    f_{k,t}(x) := x + t\nabla^\perp \Delta^{-1}e_k(x).
  \]
\end{definition}

The next two lemmas follow from the fact that for each $k_0 \in F$, the vector field $\nabla^\perp \Delta^{-1}e_k$ is an exact solution to~\eqref{eq:sns} on which the nonlinearity vanishes. They are immediately proved using Lemma~\ref{lem:reach-control} and a control which is the concatenation in time of controls of the form $g^k(t) = \psi(t)\delta_{k_0,k}$ for an appropriate bump function $\psi$ and where $\delta$ denotes the Kronecker delta.

\begin{lemma}\label{lem:dif-proj}
  If $E_1, E_2 \subseteq M^P$, then $Reach(\{0\} \times E_1, \{0\} \times E_2)$ is satisfied for the projective process if for every $\varepsilon > 0$ there is some $T > 0$ such that for each $(x_0, v_0) \in E_1$ there are finite sequences $t_1, t_2, \dots, t_n$ with $\sum_i |t_i| \leq T$ and $k_1, k_2, \dots, k_n \in F$ such that
  \[
    |(f_{k_n,t_n} \circ \dots \circ f_{k_1,t_1})(x_0) - \tilde{x}| < \varepsilon
  \]
  and
  \[
    |v_0 \cdot \nabla(f_{k_n,t_n} \circ \dots \circ f_{k_1,t_1})(x_0) - \tilde{v}| < \varepsilon
  \]
  for some $(\tilde{x}, \tilde{v}) \in E_2$, where we identify a vector in $\R^2$ with a point on $S^1$ by projection.
\end{lemma}
\begin{lemma}\label{lem:dif-2pt}
  If $E_1, E_2 \subseteq M^2$, then $Reach(\{0\} \times E_1, \{0\} \times E_2)$ holds for the two-point process if for every $\varepsilon > 0$ there is some $T > 0$ such that for each $(x_0, v_0) \in E_1$ there are finite sequences $t_1, t_2, \dots, t_n$ with $\sum_i |t_i| \leq T$ and $k_1, k_2, \dots, k_n \in F$ such that
  \[
    |(f_{k_n,t_n} \circ \dots \circ f_{k_1,t_1})(x_0) - \tilde{x}| < \varepsilon
  \]
  and
  \[
    |(f_{k_n,t_n} \circ \dots \circ f_{k_1,t_1})(y_0) - \tilde{y}| < \varepsilon
  \]
  for some $(\tilde{x}, \tilde{y}) \in E_2$.

\end{lemma}

\subsection{The projective and Jacobian processes}
\begin{figure}
  \centering
  \fbox{\begin{subfigure}{0.47\textwidth}
      \includegraphics[width=\textwidth]{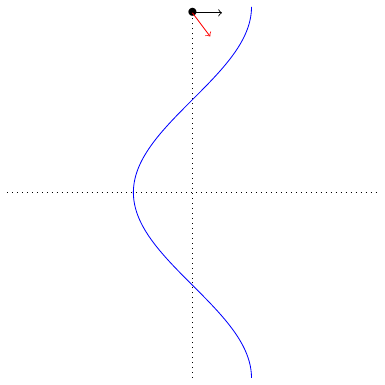}
      \caption{Step 1}
    \end{subfigure}}
  \fbox{\begin{subfigure}{0.47\textwidth}
      \includegraphics[width=\textwidth]{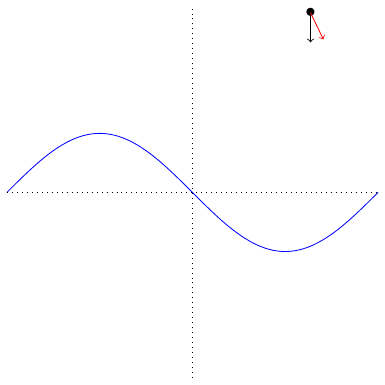}
      \caption{Step 2}
    \end{subfigure}}

  \fbox{\begin{subfigure}{0.47\textwidth}
      \includegraphics[width=\textwidth]{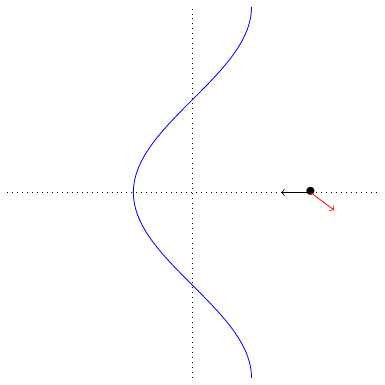}
      \caption{Step 3}
    \end{subfigure}}
  \fbox{\begin{subfigure}{0.47\textwidth}
      \includegraphics[width=\textwidth]{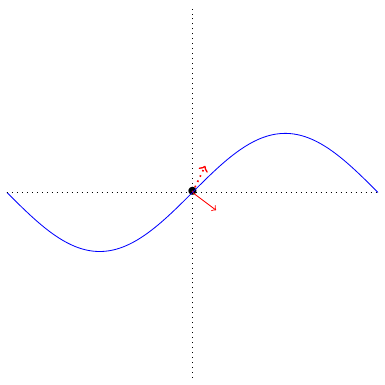}
      \caption{Step 4}
    \end{subfigure}}
  \caption{Controlling the projective process; the black arrow denotes the direction $x$ moves under the flow with profile drawn in blue. The red arrow is the coordinate $v \in S^1$, and the dotted red arrow in Step 4 is $\tilde{v}$.}\label{fig:proj-control}
\end{figure}
\begin{lemma}\label{lem:proj-control}
  For each $\tilde{v} \in S^1$ and $C > 0$, we have
  \[
    \Reach(\{(u, x, v) \in \uspace \times M^P \mid V(u) \leq C\}, (0, 0, \tilde{v})).
  \]
\end{lemma}
\begin{proof}
  Fix $k_0, k_1 \in F$ linearly independent with $k_0, k_1 > (0, 0)$ lexicographically and assume without loss of generality that $c_{k_0}, c_{k_1} > 0$. Fix $(u_0, x_0, v_0)$ with $V(u_0) \leq C$. We will construct a control in several steps which leads to $(0, 0, \tilde{v})$, noting that in each step the dependence on $u_0$ is only through the upper bound $V(u) \leq C$.

  \textbf{Step 0 (decay of $u$).}
  We show $\Reach((u_0, x_0, v_0), \{0\} \times M^P)$ holds. Indeed, the solution $u_t^0$ to the unforced equation tends to $0$ exponentially fast, so taking $g \equiv 0$ suffices.
  Indeed, this follows (for example) by noting that $\frac{\dif}{\dif t}\|\omega_t^0\|_{L^2} \leq -\nu \|\nabla \omega_t^0\|_{L^2}$ (by testing the equation with $\omega_t^0$), and interpolating between this inequality the growth bound obtained by repeating the proof of Proposition~\ref{prop:omega-bounds} with no forced modes (so the system is entirely deterministic).

  For the remaining steps, we use Lemma~\ref{lem:dif-proj} and refer to Figure~\ref{fig:proj-control}

  \textbf{Step 1 (moving $x$ to the workable region).}
  We show that $\Reach((0, x_0, v_0), \{(0, x, v) \mid |\sin(k_0 \cdot x)| \geq \frac{1}{2} \land v \in S^1\})$ holds. We first note that either $|\sin(k_1 \cdot x_0)| \geq \frac{1}{2}$ or $|\cos(k_1 \cdot x_0)| \geq \frac{1}{2}$. Then use the diffeomorphism $f := f_{k_1,T}$ in the first case and $f := f_{-k_1,T}$ in the second case, where $T$ is chosen to be the first time when $|\sin(k_0 \cdot f(x_0))| \geq \frac{1}{2}$.

  \textbf{Step 2 (aligning $x$ with $0$ in the $k_1$ direction).} We show that if $|\sin(k_0 \cdot x_0)| \geq \frac{1}{2}$, then $\Reach\left((0, x_0, v_0), \{(0, x, v) \mid k_1 \cdot x_0 = 0, v \in S^1\}\right)$ holds. Similarly to the previous step, we use the diffeomorphism $f_{k_0,T}$ and choose $T$ to be the first time when the desired condition is met.

  \textbf{Step 3 (moving $x$ to $0$).} We show that if $k_1 \cdot x_0 = 0$, then $\Reach((0, x_0, v_0), \{(0, 0)\} \times S^1)$ holds. This follows as in the previous steps, using the diffeomorphism $f_{-k_1,T}$.

  \textbf{Step 4 (moving $v$ to $\tilde{v}$).}
  We conclude that $\Reach((0, 0, v_0), (0, 0, \tilde{v}))$ holds. Indeed, choose $T > 0$ large and define
  \begin{equation*}
    v_t :=
    \begin{cases}
      v_0 \cdot \nabla f_{k_0,T}(x) &\quad t \in [0, T],\\
      v_T \cdot \nabla f_{k_1,T}(x) &\quad t \in (T, 2T],\\
      v_{2T} \cdot \nabla f_{k_0,-T}(x) &\quad t \in (2T, 3T],\\
      v_{3T} \cdot \nabla f_{k_1,-T}(x) &\quad t \in (3T, 4T],
    \end{cases}
  \end{equation*}
  where we again identify vectors in $\R^2$ with points in $S^1$ by projection. Unwrapping the definition of $f$, direct computation shows that $v_t$ winds around the circle at least once and therefore takes on every value in $S^1$.
\end{proof}
\begin{lemma}\label{lem:Jacobian-control}
  $\Reach((0, 0, \Id), \{(0, 0, A) \in \uspace \times M^J : |A| > M\})$ holds.
\end{lemma}
\begin{proof}
  Fix $k_0 \in F$ with $k_0 > (0, 0)$ lexicographically, so $e_{k_0}(x) = \sin(k_0 \cdot x)$ and in particular $e_{k_0}(0) = 0$. For $\delta > 0$, define $g \colon F \times \R_{\geq 0} \to \R$ by $g^{k_0}(t) := \delta t$ and $g^k(t) := 0$ for $k \neq k_0$. We observe that the solution $(u_t^g, x_t^g, A_t^g)$ to the controlled equation satisfies $\|u_t\|_{\uspace} \leq o_\delta(1)$, $x_t \equiv 0$ uniformly in $t$ and therefore a direct computation shows $|A_t^g| \to \infty$ as $t \to \infty$, so we choose $T$ sufficiently large and $\delta$ sufficiently small to conclude.
\end{proof}
\subsection{The two-point process}
\begin{figure}
  \fbox{\begin{subfigure}{0.47\textwidth}
      \includegraphics[width=\textwidth]{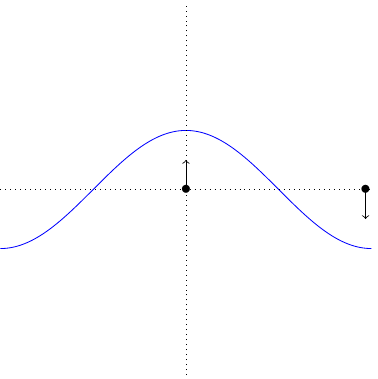}
      \caption{$f_{-k_2,\cdot}$}
    \end{subfigure}}
  \fbox{\begin{subfigure}{0.47\textwidth}
      \includegraphics[width=\textwidth]{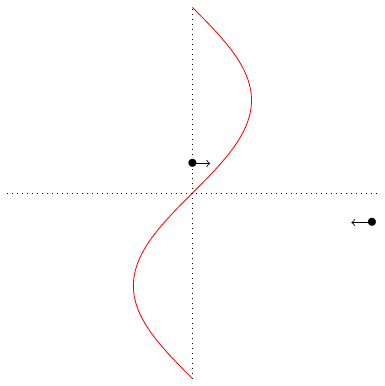}
      \caption{$f_{k_3,.}$}
    \end{subfigure}}

  \fbox{\begin{subfigure}{0.47\textwidth}
      \includegraphics[width=\textwidth]{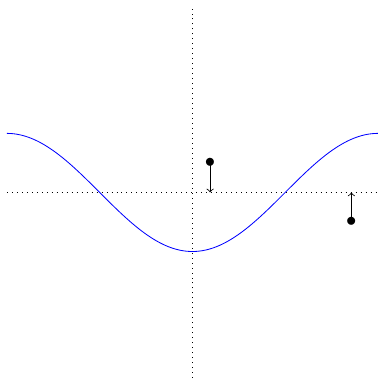}
      \caption{$f_{-k_2,\cdot}$}
    \end{subfigure}}
  \fbox{\begin{subfigure}{0.47\textwidth}
      \includegraphics[width=\textwidth]{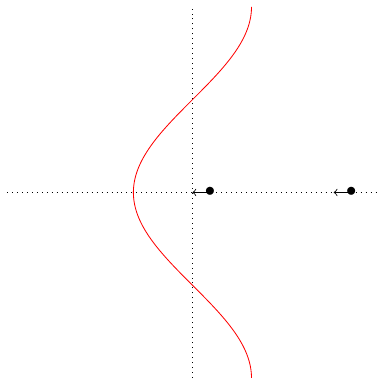}
      \caption{$f_{-k_3,\cdot}$}
    \end{subfigure}}
  \caption{Step 3 (moving $y$ to the workable region) for the two-point process; we illustrate the case where $k_2 = (1, 0)$, $k_3 = (0, 1)$ and draw $e_{\pm k_2}$ in blue and $e_{\pm k_3}$ in red. Observe that the point $x$ which starts at the origin ends at the origin, and the point $y$ which starts at $(\pi, 0)$ ends at $(\pi-\delta, 0)$ for some small constant $\delta > 0$.}\label{fig:2pt-control}
\end{figure}
\begin{lemma}\label{lem:2pt-control}
  There is $\tilde{y} \in \xspace$ such that for any $C > 0$,
  \[
    \Reach(\{(u, x, y) \in \uspace \times M^2 : V(u),|x-y|^{-1} \leq C\}, (0, 0, \tilde{y}))
  \]
  holds.
\end{lemma}
\begin{proof}
  Fix $k_0, k_1 \in F$ linearly independent with $k_0, k_1 > (0, 0)$ lexicographically and assume without loss of generality that $c_{k_0}, c_{k_1} > 0$. Choose $\tilde{y} \in \xspace$ to be such that $|\sin(k_0 \cdot \tilde{y})|, |\sin(k_1 \cdot \tilde{y})| \geq \frac{1}{2}$.

  \textbf{Step 1 (decay of $u$).}
  We show that $\Reach((u_0, x_0, y_0), \{(0, 0)\} \times \xspace)$ holds, and that the point $(\tilde{y}, \tilde{x})$ which witnesses this fact satisfies $|\tilde{y}-\tilde{x}|^{-1} \leq C(\varepsilon)\exp(V(u_0))|y_0-x_0|^{-1}$. We use the same control as in Steps 0, 1 and 2 in the proof of Lemma~\ref{lem:proj-control}, noting that in, Step 0, $u_t$ decays at an exponential rate (and therefore $|x_T-y_T|^{-1} \leq C\exp(V(u_0))$ after that step) and that the remaining steps contract $|x-y|$ by another constant factor, depending on $\varepsilon$.
  We note that this is the only step which depends on $u_0$, and it is clear that the dependence is only through $V(u_0)$.

  \textbf{Step 2 (separation of $x$ and $y$).}
  We show that
  \[
    \Reach\left(\{(0, 0)\} \times \{y \in \xspace : |y| \geq c\}, \{(0, 0)\} \times \left\{y \in \xspace : |y| \geq \frac{1}{100\max(|k_0|, |k_1|)}\right\}\right)
  \]
  holds for each $c \geq 0$. This is immediate by iterating the pair of diffeomorphisms $f_{k_0,10},f_{k_1,10}$.

  \textbf{Step 3 (moving $y$ to the workable region).}
  We show that $\Reach((0, 0, y_0), \{(0, 0)\} \times R)$ holds whenever $y_0 \geq \frac{1}{100\max_{k \in F}|k|}$ and, for each $k \in F$,
  \[
    R_k := \left\{y \in \xspace : |\sin(k \cdot y)| \geq \frac{1}{100\max_{k' \in F}|k'|}\right\},
  \]
  and
  \[
    R := \bigcup_{k_2, k_3 \in F \; \text{linearly independent}} (R_{k_2} \cap R_{k_3})
  \]
  If $y_0 \in R_{k_2}$ for some $k_2 \in F$, then we can simply use the diffeomorphism $f_{k_2,T}$.

  Otherwise, using the assumption that the integer linear span of $F$ is $\Z^2$, there are some linearly independent $k_2,k_3 \in F$, with $k_2,k_3 > (0, 0)$ lexicographically, such that $\cos(k_2 \cdot y_0) < -\frac{1}{2}$. For some $\delta_1, \delta_2, \delta_3, \delta_4 > 0$, the controls $f_{-k_2,\delta_1}$ (until $y \in R_{k_3}$), $f_{k_3,\delta_2}$ (until $y \in R_{k_2}$), $f_{-k_2,-\delta_3}$ (until $k_3 \cdot x = 0$) and $f_{-k_3,-\delta_4}$ (until $x = 0$). One can check (see Figure~\ref{fig:2pt-control}) that after all the diffeomorphisms are applied successively, we have $y \in R_{k_2}$. We therefore conclude this step.

  \textbf{Step 4 (conclusion).}
  We show that if $\sin(k_2 \cdot y_0),\sin(k_3 \cdot y_0) \neq 0$ for some $k_2,k_3 \in F$, then we have $\Reach((0, 0, y_0), (0, 0, \tilde{y}))$ holds for $y_0 \neq 0$. This follows from using the diffeomorphisms $f_{k_2,T}$ and $f_{k_3,T'}$ successively, first ensuring that $k_3 \cdot f_{k_2,T}(y_0) = k_3 \cdot \tilde{y}$, and then aligning in the $k_2$ direction.
\end{proof}

{\small
  \bibliographystyle{alpha}
  \bibliography{references,references-bill}
}

\end{document}